\documentclass[a4paper, reqno, 10pt]{amsart}

\usepackage[margin=1in]{geometry}
\usepackage{amsfonts,amssymb}
\usepackage{amsthm}
\usepackage{graphicx}
\usepackage{float}
\usepackage{lipsum}
\usepackage{comment}
\usepackage{systeme}
\usepackage{lineno}
\usepackage{mathtools}
\usepackage{amsmath}
\usepackage{esint}
\usepackage{mathrsfs}
\usepackage{verbatim}
\usepackage{stmaryrd}
\usepackage{colortbl}
\usepackage[usenames,dvipsnames]{xcolor}
\usepackage{amssymb}
\usepackage[active]{srcltx}

\usepackage{cite}
\usepackage[colorlinks, pdfpagelabels, pdfstartview = FitH,bookmarksopen = true, bookmarksnumbered = true,linkcolor = blue,plainpages = false,hypertexnames = false, citecolor = red]{hyperref}
\usepackage[italian, textsize=tiny, color=BlueGreen, backgroundcolor=BlueGreen, linecolor=BlueGreen]{todonotes}

\usepackage{witharrows} 

\allowdisplaybreaks

\usepackage{booktabs,amsmath}
\def\Xint#1{\mathchoice
    {\XXint\displaystyle\textstyle{#1}}%
    {\XXint\textstyle\scriptstyle{#1}}%
    {\XXint\scriptstyle\scriptscriptstyle{#1}}%
    {\XXint\scriptscriptstyle\scriptscriptstyle{#1}}%
      \!\int}
\def\XXint#1#2#3{{\setbox0=\hbox{$#1{#2#3}{\int}$}
    \vcenter{\hbox{$#2#3$}}\kern-.5\wd0}}
\def\dashint{\Xint-}

\def\YYint#1#2#3{{\setbox0=\hbox{$#1{#2#3}{\int}$}
    \lower1ex\hbox{$#2#3$}\kern-.46\wd0}}
\def\YYYint#1#2#3{{\setbox0=\hbox{$#1{#2#3}{\int}$}
    \lower0.35ex\hbox{$#2#3$}\kern-.48\wd0}}

\def\ZZint#1#2#3{{\setbox0=\hbox{$#1{#2#3}{\int}$}
    \raise1.15ex\hbox{$#2#3$}\kern-.57\wd0}}
\def\ZZZint#1#2#3{{\setbox0=\hbox{$#1{#2#3}{\int}$}
    \raise0.85ex\hbox{$#2#3$}\kern-.53\wd0}}

\def\XXiint#1#2#3{{\setbox0=\hbox{$#1{#2#3}{\iint}$}
    \vcenter{\hbox{$#2#3$}}\kern-.5\wd0}}

\DeclareMathOperator*{\dist}{dist}

\DeclareMathOperator*{\Tailp}{Tail_{s,p}}
\DeclareMathOperator*{\Tails}{Tail_{\frac{s}{p},p}}

\makeatletter
\@namedef{subjclassname@2020}{\textup{2020} Mathematics Subject Classification}
\makeatother

\newcommand{\rom}[1]{\uppercase\expandafter{\romannumeral #1\relax}}

\newcommand{\Chi}{\mbox{\Large$\chi$}}

\newcommand{\Author}{Sun-Sig Byun \and Kyeongbae Kim}
\newcommand{\Title}{ $L^{q}$ estimates for nonlocal p-Laplacian type equations with BMO kernel coefficients \\
in divergence form}
\newcommand{\Shorttitle}{ $L^{q}$ estimates for nonlocal p-Laplacian type equations with BMO kernel coefficients in divergence form}

\theoremstyle{plain}
\newtheorem{thm}{Theorem}[section]

\newtheorem{lem}[thm]{Lemma}
\newtheorem{cor}[thm]{Corollary}

\theoremstyle{definition}

\newtheorem*{defn*}{Definition}

\theoremstyle{remark}
\newtheorem{rmk}{Remark}

\numberwithin{equation}{section}


\newcommand{\dx}{\,dx}

\newcommand{\dy}{\,dy}
\newcommand{\dmu}{\,d\mu}
\newcommand{\dmut}{\,d\mu_{\tau}}

\newcommand{\dlambda}{\,d\lambda}

\newcommand{\ddiv}{\mathrm{div\,}}
\newcommand{\ddivs}{\mathrm{div}_{s}\,}
\newcommand{\ddivt}{\mathrm{div}_{t}\,}

\title[\Shorttitle]{\Title}
\author[\Author]{\Author}
\address{Sun-Sig Byun: Department of Mathematical Sciences and Research Institute of Mathematics, Seoul National University, Seoul 08826, Korea}
\email{byun@snu.ac.kr}

\address{Kyeongbae Kim: Department of Mathematical Sciences, Seoul National University, Seoul 08826, Korea}
\email{kkba6611@snu.ac.kr}

\begin{document}
\maketitle

\begin{abstract}
We study $s$-fractional $p$-Laplacian type equations with discontinuous kernel coefficients in divergence form to establish $W^{s+\sigma,q}$ estimates for any choice of pairs $( \sigma,q)$ with $q\in(p,\infty)$ and $\sigma\in\left(0,\min\left\{\frac{s}{p-1},1-s\right\}\right)$ under the assumption that the associated kernel coefficients have small BMO seminorms near the diagonal. As a consequence, we find in the literature an optimal fractional Sobolev regularity 
of such a non-homogeneous nonlocal equation when the right-hand side is presented by a suitable fractional operator. Our results are new even in the linear case.
\end{abstract}

\section{Introduction}
\subsection{Overview and main results.}
In this paper, we study the nonhomogeneous problem of $s$-fractional $p$-Laplacian type equation  
\begin{equation}
\label{main eq1}
    (-\Delta_{p})_{A}^{s}u=(-\Delta_{p})^{\frac{s}{p}}f\quad\text{in }\Omega,
\end{equation}
where 
\begin{equation*}
    (-\Delta_{p})_{A}^{s}u(x)=\,\mathrm{p.v.}\int_{\mathbb{R}^{n}}A(x,y)\frac{|u(x)-u(y)|^{p-2}(u(x)-u(y))}{|x-y|^{s(p-1)}}\frac{\dy}{|x-y|^{n+s}}
\end{equation*}
and
\begin{equation*}
   (-\Delta_{p})^{\frac{s}{p}}f(x) =\,\mathrm{p.v.}\int_{\mathbb{R}^{n}}|f(x)-f(y)|^{p-2}(f(x)-f(y))\frac{\dy}{|x-y|^{n+s}}=(-\Delta_{p})_{1}^{\frac{s}{p}}f(x).
\end{equation*}
Here, $0<s<1$, $2\leq p<\infty$, $\Omega\subset\mathbb{R}^{n}$ is an open and bounded set with $n\geq 2$, and $A=A(x,y):\mathbb{R}^{n}\times\mathbb{R}^{n}\to\mathbb{R}$ is a kernel coefficient with
\begin{equation*}
    A(x,y)=A(y,x)\quad\text{and}\quad \Lambda^{-1}\leq A(x,y)\leq\Lambda
\end{equation*}
for all $(x,y)\in\mathbb{R}^{n}\times\mathbb{R}^{n}$ and for some constant $\Lambda\geq1$. And $f$ is a given measurable function and $u=u(x):\mathbb{R}^{n}\to\mathbb{R}$ is the unknown.

In fact, \eqref{main eq1} occurs as the Euler-Lagrange equation of the following functional
\begin{equation*}
v\mapsto \int_{\mathbb{R}^{n}}\int_{\mathbb{R}^{n}}\left(\frac{1}{p}A(x,y)\left|\frac{v(x)-v(y)}{|x-y|^{s}}\right|^{p}-\left|f(x)-f(y)\right|^{p-2}(f(x)-f(y))\left(\frac{v(x)-v(y)}{|x-y|^{s}}\right)\right)\frac{\dx\dy}{|x-y|^{n}}
\end{equation*}
defined for $v\in W^{s,p}(\mathbb{R}^{n})$ satisfying $v=g$ on $\mathbb{R}^{n}\setminus \Omega$ for some boundary data $g\in W^{s,p}(\mathbb{R}^{n})$. 
The aim of this paper is to establish that the following implication 
\begin{equation}
\label{main result}
    \big(f(x)-f(y)\big)\in L^{q}_{\mathrm{loc}}\left(\Omega\times\Omega\,;\,\frac{\dx\dy}{|x-y|^{n+\sigma q}}\right)\Longrightarrow \left(\frac{u(x)-u(y)}{|x-y|^{s}}\right)\in L^{q}_{\mathrm{loc}}\left(\Omega\times\Omega\,;\,\frac{\dx\dy}{|x-y|^{n+\sigma q}}\right)
\end{equation}
holds with the desired Calder\'on-Zygmund type estimate \eqref{goalestimate} for each choice of two numbers $q,\sigma$ with
\begin{equation}
\label{mainassumptions}
\mbox{$q\in(p,\infty)$ and $\sigma\in\left(0,\min\left\{\frac{s}{p-1},1-s\right\}\right)$}
\end{equation}
under a possibly discontinuous kernel coefficient $A(x,y)$. For the classical case that $p=q$ and $\sigma=0$, there is a unique weak solution $u\in W^{s,p}(\Omega)\cap L^{p-1}_{sp}(\mathbb{R}^{n})$ with the standard energy estimate \eqref{estimateofp0}, as follows from Lemma \ref{existence} below.

We first discuss a motivation of the study \eqref{main eq1} from the corresponding local problem when $s=1$. According to the well-known elliptic theory, for a weak solution $w$ of 
\begin{equation*}
    -\ddiv(B|Dw|^{p-2}Dw)=-\ddiv(|F|^{p-2}F)\quad\text{in }\Omega,
\end{equation*}
which comes from the Euler-Lagrange equation of the functional
\begin{equation*}
     v\mapsto \int_{\Omega}\frac{1}{p}B|Dv|^{p}-|F|^{p-2}F\cdot Dv\dx
\end{equation*}
defined for $v\in W^{1,p}(\Omega)$ with $v=g$ on $\partial\Omega$,
there holds 
\begin{equation}
\label{czforlocal}
    F\in L^{q}_{\mathrm{loc}}\Longrightarrow Dw\in L^{q}_{\mathrm{loc}}
\end{equation}
for every $q\in[p,\infty)$, provided that the principal coefficient function $B:\Omega\to\mathbb{R}$ has a sufficiently small BMO seminorm, see \cite{CPoo,KZav,BWc,Ip} and references therein. To find a nonlocal analogue of the Calder\'on-Zygumnd theory \eqref{czforlocal}, we first need to introduce a fractional gradient operator and a fractional divergence operator by adopting the notation from the very interesting works \cite{MSf}. For $t\in[0,1)$, the fractional gradient operator $d_{t}:\mathcal{M}\left(\mathbb{R}^{n}\right)\to\mathcal{M}_{od}\left(\mathbb{R}^{n}\times \mathbb{R}^{n}\right)$ is defined by
\begin{equation*}
    (d_{t}g)(x,y)=\frac{g(x)-g(y)}{|x-y|^{t}}\quad\text{if $x\neq y$, }
\end{equation*}
where $\mathcal{M}(\mathbb{R}^{n})$ is the function space consisting of all measurable real-valued functions defined on $\mathbb{R}^{n}$ and $\mathcal{M}_{od}(\mathbb{R}^{n}\times\mathbb{R}^{n})=\{F(x,y)\in\mathcal{M}(\mathbb{R}^{n}\times\mathbb{R}^{n})\,;\, F(x,y)=-F(y,x)\}$. 
Note that $d_{t}g$ is well defined as $\left\{(x,y)\in\mathbb{R}^{n}\times\mathbb{R}^{n}\,;\, x=y\right\}$ is a measure zero set in $\mathbb{R}^{n}\times\mathbb{R}^{n}$. For $t\in(0,1)$, the fractional divergence operator $\ddivt:\mathcal{S}_{od}\left(\mathbb{R}^{n}\times\mathbb{R}^{n}\right)\to\mathcal{S}\left(\mathbb{R}^{n}\right)$ is defined by
\begin{equation*}
    \ddivt(F)(x)=\int_{\mathbb{R}^{n}}\frac{F(x,y)}{|x-y|^{t}}\frac{\dy}{|x-y|^{n}},
\end{equation*}
where $\mathcal{S}\left(\mathbb{R}^{n}\right)$ is the classical Schwartz space and $\mathcal{S}_{od}\left(\mathbb{R}^{n}\times\mathbb{R}^{n}\right)=\mathcal{S}(\mathbb{R}^{n}\times \mathbb{R}^{n})\cap\mathcal{M}_{od}(\mathbb{R}^{n}\times \mathbb{R}^{n})$ (see \cite{Wd}). 
Then our equation \eqref{main eq1} can be rewritten as
\begin{equation*}
    \ddivs(A(x,y)|d_{s}u|^{p-2}d_{s}u)=\ddivs(|d_{0}f|^{p-2}d_{0}f),
\end{equation*}
while the relation \eqref{main result} can be rewritten as 
\begin{equation*}
        d_{0}f\in L^{q}_{\mathrm{loc}}\left(\Omega\times\Omega\,;\,\frac{\dx\dy}{|x-y|^{n+\sigma q}}\right)\Longrightarrow d_{s}u\in L^{q}_{\mathrm{loc}}\left(\Omega\times\Omega\,;\,\frac{\dx\dy}{|x-y|^{n+\sigma q}}\right)
\end{equation*}
for each $q\in(p,\infty)$ and $\sigma\in\left(0,\min\left\{\frac{s}{p-1},1-s\right\}\right)$, which is the corresponding nonlocal version of our Calder\'on-Zygmund theory.

We now mention the recent results for the Calder\'on-Zygmund theory of nonlocal problems. For $p=2$, Mengesha, Schikorra and Yeepo \cite{MSY} establish the Calder\'on-Zygmund theory for a variety of forcing terms, provided that the kernel coefficient $A$ is H\"older continuous by obtaining suitable commutator estimates. On the other hand, in \cite{Nv,Ni}, Nowak obtains Calder\'on-Zygmund-type estimates with respect to non-divergence data with the associated kernel coefficient $A$ having a small $BMO$ seminorm on $\Omega\times\Omega$ via maximal function techniques. For the global estimate, Abdellaoui, Fern\'andez, Leonori and Younes \cite{AFLY} establish global regularity results for the fractional Laplacian, with the zero boundary condition and non-divergence data by employing the Green function formula of the fractional Laplacian.

For the nonlinear case when $p>2$, the Calder\'on-Zygmund theory for non-divergence data is recently investigated by Diening and Nowak in \cite{LNc}. In particular, they accomplish this by obtaining precise pointwise bounds in terms of a certain fractional sharp maximal function. On the other hand, we here wish to find an optimal Calder\'on-Zygmund theory for divergence data by dealing with a suitable form of the right-hand side along with a minimal regularity assumptions on the kernel coefficients. In addition, we point out that unlike the methods used in \cite{Ni,Nv, MSY, AFLY}, we use the maximal function free technique introduced in \cite{AMp} and later extensibly employed in the literature.

Later in this paper, we show that our result covers the main result given in \cite{Ni} (see Subsection \ref{subsectioin12} below). 
The main difficulty to handle such nonlinear problems is that even though $u$ and $v$ are weak solutions to \eqref{main eq1}, $u\pm v$ is not a weak solution to \eqref{main eq1}, as usual. Indeed for $p=2$, this linearity is used to prove comparison estimates of higher-order fractional gradients, which is an essential ingredient to improve the range of the fractional differentiability (see \cite{Ni}). However, in this paper when dealing with nonlinear fractional problems, it seems difficult to find desired comparison estimates of higher-order fractional gradients.
To overcome this, we turn to an interpolation argument along with slightly higher fractional Sobolev regularity of a solution, which turns out to be obtained regardless of the linearity, and then run a boot strap argument in order to prove comparison estimates for fractional gradients of higher order. We believe that this method is applicable for more general nonlinear nonlocal equations with nonstandard growth \cite{BKO,BOS,Olv}. We further refer to \cite{FMSY,KMS1,KMS2,ABES,NHH,BLS2,CKPf,CKPh,KKP,Nexist,KH1,Cr,Sn,Shf,Ci,BL,Hlp,KAAYg, NKSg,FMg,SM,FTASc} for a further discussion of various regularity results of nonlocal problems.

We next give the definition of local weak solutions to \eqref{main eq1}. See Section \ref{section2} for related definitions and notations.
\begin{defn*}
Let $d_{0}f\in L^{p}_{\mathrm{loc}}\left(\frac{\dx\dy}{|x-y|^{n}}\,;\,\Omega\times\Omega\right)$ and $f\in L^{p-1}_{s}(\mathbb{R}^{n})$. We say that $u\in W^{s,p}_{\mathrm{loc}}(\Omega)\cap L^{p-1}_{sp}(\mathbb{R}^{n})$ is a local weak solution to \eqref{main eq1} if it satisfies for any $\phi\in W_{c}^{s,p}(\Omega)$,
\begin{equation}
\label{defn of weak}
\begin{aligned}
    &\int_{\mathbb{R}^{n}}\int_{\mathbb{R}^{n}}A(x,y)\left(\frac{|u(x)-u(y)|}{|x-y|^{s}}\right)^{p-2}\left(\frac{u(x)-u(y)}{|x-y|^{s}}\right)\left(\frac{\phi(x)-\phi(y)}{|x-y|^{s}}\right)\frac{\dx\dy}{|x-y|^{n}}\\
    &=\int_{\mathbb{R}^{n}}\int_{\mathbb{R}^{n}}|f(x)-f(y)|^{p-2}\left(f(x)-f(y)\right)\left(\frac{\phi(x)-\phi(y)}{|x-y|^{s}}\right)\frac{\dx\dy}{|x-y|^{n}}.
\end{aligned}
\end{equation}
\end{defn*}
\begin{rmk}
    Since $d_{0}f\in L^{p}_{\mathrm{loc}}\left(\frac{\dx\dy}{|x-y|^{n}};\Omega\times\Omega\right)$ and $f\in L^{p-1}_{s}(\mathbb{R}^{n})$, Lemma \ref{existence} ensures that the right-hand side is well-defined.
\end{rmk}

We next introduce a regularity assumption on the associated kernel coefficient $A$, so called, $(\delta,R)$-vanishing condition.
\begin{defn*}Let $A:\mathbb{R}^{n}\times \mathbb{R}^{n}\to\mathbb{R}$ be a kernel coefficient. We say that for $\delta,R>0$, $A$ is $(\delta,R)$-vanishing in $\Omega\times\Omega$ if   
\begin{equation}
\label{bmocond1}
    \sup_{(x_{0},y_{0})\in\Omega\times\Omega}\sup_{0<r\leq R}\dashint_{B_{r}(x_{0})}\dashint_{B_{r}(y_{0})}|A(x,y)-(A)_{B_{r}(x_{0})\times B_{r}(y_{0})}|\dx\dy\leq \delta.
\end{equation}
In particular, we say that $A$ is $(\delta,R)$-vanishing in $\Omega\times\Omega$ only at the diagonal, if 
\begin{equation}
\label{bmocond2}
    \sup_{x_{0}\in\Omega}\sup_{0<r\leq R}\dashint_{B_{r}(x_{0})}\dashint_{B_{r}(x_{0})}|A(x,y)-(A)_{B_{r}(x_{0})\times B_{r}(x_{0})}|\dx\dy\leq \delta.
\end{equation}
\end{defn*}
We now mention that the class of functions with the $(\delta,R)$-vanishing condition contains not only all continuous functions, but also a large class of functions with discontinuity.
\begin{rmk}
We note that given $\delta>0$, any function belonging to the class of functions with vanishing mean oscillation $VMO$ satisfies $(\delta, R)$-vanishing condition, whenever $R>0$ is sufficiently small depending on $\delta$ (see \cite{Sfv}). On the other hand, there are many $(\delta,R)$-vanishing kernel coefficients which does not belong to VMO space. Let us assume that $K$ is a merely measurable kernel which is not in VMO space. Choose 
\begin{equation*}
    A(x,y)=\frac{\delta}{4\Lambda}K(x,y)+\frac{\Lambda}{2}\quad(x,y\in\mathbb{R}^{n})
\end{equation*}
to see that \eqref{bmocond1} for any $R>0$, but $A$ does not belong to VMO space.
\end{rmk}
The next remark is to explain that \eqref{bmocond2} is the more general assumption than \eqref{bmocond1}.
\begin{rmk}
Note that if $A$ is $(\delta,R)$-vanishing in $\Omega\times\Omega$, then $A$ is also $(\delta,R)$-vanishing in $\Omega\times\Omega$ only at the diagonal. However, the converse is not true. For instance if $A(x,y)=\frac{K_{1}(x,y)\Chi_{\{|x-y|> R\}}+K_{2}(x,y)}{2\Lambda}+\frac{\Lambda}{2}$ where $K_{1}$ is merely measurable and $K_{2}$ is $(\delta,R)$-vanishing in $\Omega\times\Omega$, then $A$ is $(\delta,R)$-vanishing in $\Omega\times\Omega$ only at the diagonal, but not $(\delta,R)$-vanishing in $\Omega\times\Omega$.
\end{rmk}

We clearly point out that our problem \eqref{main eq1} has a scaling invariant property which we now state.
\begin{lem} 
Let $u$ be a weak solution to \eqref{main eq1} and $r\in(0,1)$. Assume that $A$ is $(\delta,R)$-vanishing in $\Omega\times \Omega$ only at the diagonal. Define 
\begin{equation*}
    \tilde{u}(x)=\frac{u(rx)}{r^{s}},\quad \tilde{f}(x)=f(rx)\quad\text{and}\quad \tilde{A}(x,y)=A(rx,ry)\quad(x,y\in\mathbb{R}^{n}).
\end{equation*}
Then $\tilde{u}$ is a weak solution to 
\begin{equation*}
     (-\Delta_{p})_{\tilde{A}}^{s}\tilde{u}=(-\Delta_{p})^{\frac{s}{p}}\tilde{f}\quad\text{in }\frac{1}{r}\Omega,
\end{equation*}
and $\tilde{A}$ is $\left(\delta,\frac{R}{r}\right)$-vanishing in $\frac{1}{r}\Omega\times \frac{1}{r}\Omega$ only at the diagonal. 
\end{lem}
We now introduce our main result.
\begin{thm}
\label{main theorem}
Let $q\in(p,\infty)$, $\sigma\in\left(0,\min\left\{\frac{s}{p-1},1-s\right\}\right)$ and $R>0$ be given. Then there exists a small positive constant $\delta$ depending only on $n,s,p,\Lambda,q $ and $\sigma$ such that for each $A$ with $(\delta,R)$-vanishing in $\Omega\times\Omega$ only at the diagonal and for each $f\in L^{p-1}_{s}(\mathbb{R}^{n})$ with $d_{0}f\in L^{q}_{\mathrm{loc}}\left(\Omega\times\Omega;\frac{\dx\dy}{|x-y|^{n+\sigma q}}\right)$, any weak solution $u\in W_{\mathrm{loc}}^{s,p}(\Omega)\cap L^{p-1}_{sp}(\mathbb{R}^{n})$ of \eqref{main eq1} satisfies that $d_{s}u\in L^{q}_{\mathrm{loc}}\left(\Omega\times\Omega;\frac{\dx\dy}{|x-y|^{n+\sigma q}}\right)$. In particular, there is a positive constant $c=c(n,s,p,\Lambda,q,\sigma)$ such that
\begin{equation}
\label{goalestimate}
\begin{aligned}
    &\left(\dashint_{B_{r}(x_{0})}\int_{B_{r}(x_{0})}|d_{s}u|^{q}\frac{\dx\dy}{|x-y|^{n+\sigma q}}\right)^{\frac{1}{q}}\\
    &\leq c\left(\left(\dashint_{B_{2r}(x_{0})}\int_{B_{2r}(x_{0})}\left|\frac{d_{s}u}{(2r)^{\sigma}}\right|^{p}\frac{\dx\dy}{|x-y|^{n}}\right)^{\frac{1}{p}}+\Tailp\left(\frac{u-(u)_{B_{2r}(x_{0})}}{(2r)^{\sigma+s}};B_{2r}(x_{0})\right)\right)\\
    &+c\left(\left(\dashint_{B_{2r}(x_{0})}\int_{B_{2r}(x_{0})}|d_{0}f|^{q}\frac{\dx\dy}{|x-y|^{n+\sigma q}}\right)^{\frac{1}{q}}+\Tails\left(\frac{f-(f)_{B_{2r}(x_{0})}}{(2r)^{\sigma}};B_{2r}(x_{0})\right)\right)
\end{aligned}
\end{equation}
whenever $B_{2r}(x_{0})\Subset\Omega$ and $r\in(0,R]$.
\end{thm}
\begin{rmk}
We observe the following equivalent relation
\begin{equation*}
    d_{s}u\in
    L^{q}_{\mathrm{loc}}\left(\Omega\times\Omega;\frac{\dx\dy}{|x-y|^{n+\sigma q}}\right)\Longleftrightarrow u\in W^{s+\sigma,q}_{\mathrm{loc}}(\Omega),
\end{equation*}
which implies that
\begin{equation*}
    f\in W^{\sigma,q}_{\mathrm{loc}}(\Omega)\Longrightarrow u\in W^{s+\sigma,q}_{\mathrm{loc}}(\Omega)
\end{equation*}
for any choice of $q\in(p,\infty)$ and $\sigma\in\left(0,\min\left\{\frac{s}{p-1},1-s\right\}\right)$.
\end{rmk}
\begin{rmk}
Let us assume $d_{0}f\in L^{q}_{\mathrm{loc}}\left(\Omega\times\Omega;\frac{\dx\dy}{|x-y|^{n}}\right)$ for some $q\in(p,\infty)$. However, it is not always true that $d_{0}f\in L^{p}_{\mathrm{loc}}\left(\Omega\times\Omega;\frac{\dx\dy}{|x-y|^{n}}\right)$, as $\int_{B}\int_{B'}\frac{\dx\dy}{|x-y|^{n}}<\infty$ for any ball $B,B'\Subset\Omega$ if and only if $\mathrm{dist}(B,B')>0$. Thus, we do not ensure that the right-hand side of \eqref{defn of weak} is well-defined. If $d_{0}f\in L^{q}_{\mathrm{loc}}\left(\Omega\times\Omega;\frac{\dx\dy}{|x-y|^{n+\sigma q q}}\right) $ for some $q\in(p,\infty)$ and $\sigma>0$, then $f\in W^{\sigma,q}_{\mathrm{loc}}(\Omega)$. By Lemma \ref{embedding sob}, we have $f\in W^{\frac{\sigma}{2},p}_{\mathrm{loc}}(\Omega)$, which gives that $d_{0}f\in L^{p}_{\mathrm{loc}}\left(\Omega\times\Omega;\frac{\dx\dy}{|x-y|^{n}}\right)$. Therefore, it is natural to take $\sigma>0$. In addition, the regularity of solutions to the homogeneous fractional p-Laplacian equation is known only for $C^{\beta}$ for any $\beta<\min\left\{\frac{sp}{p-1},1\right\}$, if $p>2$ (see \cite{BLS2}). Accordingly, the upper bound of $\sigma$ is $\min\left\{\frac{s}{p-1},1-s\right\}$.
\end{rmk}

\subsection{Derivation of regularity results for non-divergence data $g$.}
\label{subsectioin12}
In this subsection, we show that for any weak solution $u\in W^{s,2}_{\mathrm{loc}}(\Omega)\cap L^{1}_{2s}(\mathbb{R}^{n})$ to
\begin{equation}
\label{nondiv1}
    (-\Delta)_{A}^{s}u=g\quad\text{in }\Omega,
\end{equation}
if $A$ has a sufficiently small BMO seminorom only at the diagonal, then the following implication
\begin{equation}
\label{ndivre}
    g\in L_{\mathrm{loc}}^{\frac{nq}{n+(2s-t)q}}(\Omega)\Longrightarrow u\in W^{t,q}_{\mathrm{loc}}(\Omega)
\end{equation}
holds for each
\begin{equation*}
    q\in(2,\infty)\quad\text{and}\quad t\in\left(s,\min\left\{2s,1-s\right\}\right).
\end{equation*}
Let $g\in L_{\mathrm{loc}}^{\frac{nq}{n+(2s-t)q}}(\Omega)$ for some $q\in (2,\infty)$ and $t\in\left(s,\min\left\{2s,1-s\right\}\right)$, and let $\Omega'\Subset\Omega$ be an open set. We first note from \cite[Theorem 4.4]{Nv} that there is a weak solution $f\in W^{\frac{s}{2},2}(\mathbb{R}^{n})$ to 
\begin{equation}
\label{nondiv2}
    (-\Delta)^{\frac{s}{2}}f=g\quad\text{in }\Omega''
\end{equation}
satisfying $f\in H^{s,\frac{nq}{n+(2s-t)q}}(\Omega')$, where $\Omega''$ is an open set such that $\Omega'\Subset\Omega''\Subset\Omega$. In light of \eqref{nondiv1} and \eqref{nondiv2}, we have that $u\in W^{s,2}_{\mathrm{loc}}(\Omega'')\cap L^{1}_{2s}(\mathbb{R}^{n})$ is a weak solution to
\begin{equation*}
    (-\Delta)_{A}^{s}u=(-\Delta)^{\frac{s}{2}}f\quad\text{in }\Omega''.
\end{equation*}
Applying \cite[Proposition 2.5]{Nv} with $p=\frac{nq}{n+(2s-t)q}$, $s_{1}=t-s$ and $p_{1}=q$ leads to
\begin{equation*}
    f\in W^{t-s,q}_{\mathrm{loc}}(\Omega'').
\end{equation*}
In addition, we recall that the fact that $W^{\frac{s}{2},2}(\mathbb{R}^{n})\subset L^{1}_{s}(\mathbb{R}^{n})$ to observe that $f\in L^{1}_{s}(\mathbb{R}^{n})$.
Therefore, our main theorem \ref{main theorem} implies that
\begin{equation}
\label{resultg}
    u\in W^{t,q}_{\mathrm{loc}}(\Omega'')\subset W^{t,q}(\Omega'),
\end{equation}
provided that the kernel coefficient is $(\delta,R)$-vanishing in $\Omega\times\Omega$ only at the diagonal for sufficiently small $\delta$ depending only on $n,s,p,\Lambda,q$ and $t$,
as $q\in(2,\infty)$ and $t-s\in \left(0,\min\left\{s.1-s\right\}\right)$. Since $\Omega'$ is arbitrarily selected to be embedded in $\Omega$, \eqref{resultg} yields $u\in W^{t,q}_{\mathrm{loc}}(\Omega)$, which is \eqref{ndivre}. 
This is the same result given in \cite[Theorem 1.1]{Ni}. 

\subsection{Plan of the paper.}This paper is organized as follows. In Section \ref{section2}, we introduce some notations, function spaces, lemmas about embeddings and tail estimates, the existence result of the corresponding boundary value problem to \eqref{main eq1} and a technical lemma. In Section \ref{section3}, we derive comparison estimates. In Section \ref{Section4}, we obtain a covering lemma to deal with upper level sets of fractional gradients of weak solutions. Finally, in Section \ref{section5}, we prove our main result.

\section{Preliminaries and Notations}
\label{section2}
In what follows, we write $c$ to mean a general constant equal or bigger than 1 and it possibly changes from line to line. Furthermore, we use parentheses to denote the relevant dependencies on parameters such as $c\equiv c(n,s,p)$, and we denote
\[\mathsf{data}=\mathsf{data}(n,s,p,\Lambda,q,\sigma).\]

We first introduce geometric and functional notations.
\begin{enumerate}
\item Let us denote $B_{r}(x)$ by the ball in $\mathbb{R}^{n}$ with center $x\in\mathbb{R}^{n}$ and a radius $r>0$, and let $B_{r}\equiv B_{r}(0)$. 
\item We denote $\mathcal{B}_{r}(x,y)=B_{r}(x)\times B_{r}(y)$ for $x,y\in\mathbb{R}^{n}$ and $r>0$. In particular, we write $\mathcal{B}_{r}(x)=\mathcal{B}_{r}(x,x)$ and $\mathcal{B}_{r}=\mathcal{B}_{r}(0)$.
\item The cube in $\mathbb{R}^{n}$ with center $x$ and side-length $r$ is denoted by $Q_{r}(x)$, and let $Q_{r}\equiv Q_{r}(0)$.
\item We denote $\mathcal{Q}_{r}(x,y)=Q_{r}(x)\times Q_{r}(y)$ for $x,y\in\mathbb{R}^{n}$ and $r>0$. Moreover, we write $\mathcal{Q}_{r}(x)=\mathcal{Q}_{r}(x,x)$ and $\mathcal{Q}_{r}=\mathcal{Q}_{r}(0)$.
\item Given cubes $Q^{1}$ and $Q^{2}$ in $\mathbb{R}^{n}$, we denote $P^{i}\mathcal{Q}=Q^{i}\times Q^{i}$ for $i=1,2$, where $\mathcal{Q}=Q^{1}\times Q^{2}$.
\item For a locally integrable function $v:\mathbb{R}^{n}\to\mathbb{R}$ and a bounded set $B\subset\mathbb{R}^{n}$, we denote the average of $v$ over $B$ by $(v)_{B}=\dashint_{B}v(x)\dx$.
\item For a locally integrable function $F:\mathbb{R}^{n}\times\mathbb{R}^{n}\to\mathbb{R}$, $x_{0}\in\mathbb{R}$ and $r>0$, we define a function $F_{r.x_{0}}:\mathbb{R}^{n}\times\mathbb{R}^{n}\to\mathbb{R}$ by 
\begin{equation}
\label{defnofA2}
    F_{r,x_{0}}(x,y)=\begin{cases}
        (F)_{\mathcal{B}_{r}(x_{0})}&\quad\text{if }(x,y)\in \mathcal{B}_{r}(x_{0}),\\
        F(x,y)&\quad\text{otherwise}.
    \end{cases}
\end{equation}

\end{enumerate}
We next introduce the fractional Sobolev space and tail space. For a measurable function $v:\Omega\to\mathbb{R}$, we say that $v\in W^{s,p}(\Omega)$ if $v\in L^{p}(\Omega)$ and 
\begin{equation*}
    [v]_{W^{s,p}(\Omega)}=\left(\int_{\Omega}\int_{\Omega}\frac{|v(x)-v(y)|^{p}}{|x-y|^{n+sp}}\dx\dy\right)^{\frac{1}{p}}<\infty.
\end{equation*}
In particular, we say $v\in W^{s,p}_{c}(\Omega)$ if $v\in W^{s,p}(\Omega)$ and $v$ has a compact support embedded in $\Omega$.
For a given open set $\Omega'\subset\mathbb{R}^{n}$ such that $\Omega\Subset\Omega'$ and a measurable function $g:\mathbb{R}^{n}\to\mathbb{R}$, we define
\begin{equation*}
    X_{g}^{s,p}(\Omega,\Omega')=\left\{v\in W^{s,p}(\Omega')\,;\,v=g\text{ for a.e. on }\mathbb{R}^{n}\setminus\Omega\right\}.
\end{equation*}
We now introduce the associated tail space. We say that $v\in L^{p-1}_{sp}(\mathbb{R}^{n})$ if 
\begin{equation*}
   \int_{\mathbb{R}^{n}}\frac{|v(y)|^{p-1}}{(1+|y|)^{n+sp}}\dy<\infty,
\end{equation*}
and we write 
\begin{equation*}
    \Tailp\left(v;B_{r}(x_{0})\right)=\left(r^{sp}\int_{\mathbb{R}^{n}\setminus B_{r}(x_{0})}\frac{|v(y)|^{p-1}}{|y-x_{0}|^{n+sp}}\dy\right)^{\frac{1}{p-1}}.
\end{equation*}
After a simple algebraic computation, we observe that for any $r>0$,
\begin{equation}
\label{scalingpropertytail}
    \Tailp(1;B_{1})=\Tailp(1;B_{r}),
\end{equation}
which is a positive number depending only on $n,s$ and $p$.
We now introduce dual pairs of operators and measures $\left(D^{\tau},\,\mu_{\tau}\right)$ for $\tau\in\left(0,\frac{n}{p}\right)$, which were first introduced in \cite{KMS1}.
Define an operator $D^{\tau}:\mathcal{M}\left(\mathbb{R}^{n}\times \mathbb{R}^{n}\right)\to \mathcal{M}\left(\mathbb{R}^{n}\times \mathbb{R}^{n}\right)$ by 
\begin{equation*}
    \left(D^{\tau}F\right)(x,y)=\frac{F(x,y)}{|x-y|^{\tau}}\quad\text{if }x\neq y,
\end{equation*}
and a measure $\mu_{\tau}$ on $\mathbb{R}^{n}\times\mathbb{R}^{n}$ by  
\begin{equation}
\label{mutaudefnf}
    \mu_{\tau}(\mathcal{A})=\int_{\mathcal{A}}\frac{\dx\dy}{|x-y|^{n-p\tau}}\quad\text{for any measurable set }\mathcal{A}\subset\mathbb{R}^{n}\times\mathbb{R}^{n}.
\end{equation}  
In this notation, we observe that
$u\in W^{s,p}_{\mathrm{loc}}(\Omega)$ if and only if $D^{\tau}d_{s}u\in L^{p}_{\mathrm{loc}}\left(\Omega\times\Omega\,;\mu_{\tau}\right)$. We now prove some properties of the measure $\mu_{\tau}$ defined in \eqref{mutaudefnf}.
\begin{lem}
\begin{enumerate}
\item There exists a constant $v_{0}$ depending only on $n$ and $p$ such that
\begin{equation}
\label{size of b mut}
\mu_{\tau}\left(\mathcal{B}_{R}(x_{0})\right)=v_{0}\frac{R^{n+p\tau}}{\tau}\quad\text{for any }x_{0}\in\mathbb{R}^{n}\text{ and }R>0.
\end{equation}
\item Let $\rho$ and $R$ be any positive numbers, and let $x_{0},y_{0}\in\mathbb{R}^{n}$. Then 
\begin{equation}
\label{doubling}
\frac{\mu_{\tau}\left(\mathcal{B}_{R}(x_{0},y_{0}))\right)}{\mu_{\tau}\left(\mathcal{B}_{\rho}(x_{0},y_{0})\right)}=\left(\frac{R}{\rho}\right)^{n+p\tau}.
\end{equation}
\item Let $\mathcal{Q}_{r}(x_{0},y_{0})$ be any cube in $\mathcal{B}_{R}$ for $r,R>0$ and $x_{0},y_{0}\in\mathbb{R}^{n}$. Then 
\begin{equation}
\label{inclusion measure}
    \frac{\mu_{\tau}\left(\mathcal{B}_{R}\right)}{\mu_{\tau}\big(\mathcal{Q}_{r}(x_{0},y_{0})\big)}\leq 2^{n}\frac{v_{0}}{\tau}\left(\frac{R}{r}\right)^{2n}.
\end{equation}
\end{enumerate}
\end{lem}
\begin{proof}
For \eqref{size of b mut} and \eqref{doubling}, we refer to \cite[Proposition 4.1]{KMS1}. A direct computation needs to  
\begin{equation*}
\begin{aligned}
\mu_{\tau}(\mathcal{B}_{R})&=v_{0}\frac{R^{n+p\tau}}{\tau}\\
&=\frac{v_{0}}{\tau}\frac{R^{n+p\tau}}{r^{2n}}\int_{Q_{r}^{1}(x_{0})}\int_{Q_{r}^{2}(y_{0})}\dx\dy\\
&\leq \frac{v_{0}}{\tau}\frac{R^{n+p\tau}}{r^{2n}}\int_{Q_{r}^{1}(x_{0})}\int_{Q_{r}^{2}(y_{0})}\left(\frac{2R}{|x-y|}\right)^{n-p\tau}\dx\dy\\
&\leq 2^{n}\frac{v_{0}}{\tau}\left(\frac{R}{r}\right)^{2n}\mu_{\tau}\big(\mathcal{Q}_{r}\left(x_{0},y_{0}\right)\big),
\end{aligned}
\end{equation*}
where we have used the fact that
\begin{equation*}
    |x-y|\leq 2R\quad\text{for any }x\in Q_{r}^{1}\text{ and }y\in Q_{r}^{2}.
\end{equation*}
This implies \eqref{inclusion measure}.
\end{proof}
We now provide embeddings of the fractional Sobolev spaces.
The first embedding lemma is the Sobolev-Poincar\'e inequality (see \cite[Lemma 2.4]{Nv} and \cite[Theorem 6.7]{DPGV}).
\begin{lem}
\label{SPIL}
Let $u\in W^{s,p}(B_{r})$. Then for any 
\begin{equation*}
    \gamma\in \begin{cases}
        \left[1,\frac{np}{n-sp}\right]&\quad\text{if } n>ps,\\
        [1,\infty) &\quad\text{if } n\leq ps,
    \end{cases}
\end{equation*}there holds that
\begin{equation}
\label{spi}
    \left(\dashint_{B_{r}}\left|u-(u)_{B_{r}}\right|^{\gamma}\dx\right)^{\frac{1}{\gamma}}\leq cr^{s}\left(\dashint_{B_{r}}\int_{B_{r}}\frac{|u(x)-u(y)|^{p}}{|x-y|^{n+sp}}\dx\dy\right)^\frac{1}{p}
\end{equation}
for some constant $c=c(n,s,p,\gamma)$. In particular, \eqref{spi} also holds if we take $Q_{r}$ instead of $B_{r}$.
\end{lem}
We next introduce another type of the Poincar\'e inequality.
\begin{lem}
\label{SPIL2}
For any $u\in X_{0}^{s,p}\left(B_{\rho},B_{R}\right)$ with $0<\rho<R$, there is a positive constant $c=c\left(n,s,p,\rho,R\right)$ such that
\begin{equation*}
    \int_{B_{R}}|u(x)|^{p}\dx\leq c\int_{B_{R}}\int_{B_{R}}\frac{|u\left(x\right)-u\left(y\right)|^{p}}{|x-y|^{n+sp}}\dx\dy.
\end{equation*}
\end{lem}
\begin{proof}
Using the fact that $u(x)\equiv 0$ on $B_{R}\setminus B_{\rho}$ and $1\leq\frac{2R}{|x-y|}$ for any $x,y\in B_{R}$, we have
\begin{equation*}
\begin{aligned}
    \int_{B_{R}}|u(x)|^{p}\dx&=\int_{B_{\rho}}|u(x)|^{p}\dx=\dashint_{B_{R}\setminus B_{\frac{R+\rho}{2}}}\int_{B_{\rho}}|u(x)-u(y)|^{p}\dx\dy\\
    &\leq cR^{n+sp}\dashint_{B_{\frac{R+\rho}{2}}\setminus B_{\rho}}\int_{B_{\rho}}\frac{|u(x)-u(y)|^{p}}{|x-y|^{n+sp}}\dx\dy\leq c\int_{B_{R}}\int_{B_{R}}\frac{|u(x)-u(y)|^{p}}{|x-y|^{n+sp}}\dx\dy.
\end{aligned}
\end{equation*}
This completes the proof.
\end{proof}

We now give an embedding lemma between fractional Sobolev spaces.
\begin{lem}(See \cite[Proposition 2.5]{Ni})
\label{embedding sob}
Let $u\in W^{s_{2},q}(B_{r})$. For $s_{1}\in(0,s_{2})$ and $p\in(1,q]$, we have that
\begin{equation*}
    [u]_{W^{s_{1},p}(B_{r})}\leq cr^{\left(s_{2}-\frac{n}{q}\right)-\left(s_{1}-\frac{n}{p}\right)}[u]_{W^{s_{2},q}(B_{r})}
\end{equation*}
for some constant $c=c(n,s_{1},s_{2},p,q)$.
\end{lem}
We introduce an interpolation lemma which will be used later in Lemma \ref{diagonal induction}.
\begin{lem}(Interpolation) 
\label{InterpoL}
Let $u\in W^{s_{1},p}(B_{r})\cap W^{s_{2},p}(B_{r})$ for $0<s_{1}<s_{2}<1$. Then we have that for any $t\in[0,1]$,
\begin{equation}
\label{interpesti}
    [u]_{W^{ts_{1}+(1-t)s_{2},p}(B_{r})}\leq [u]_{W^{s_{1},p}(B_{r})}^{t}[u]_{W^{s_{2},p}(B_{r})}^{1-t}.
\end{equation}
\end{lem}
\begin{proof}
If $t=0$ or $t=1$, then it follows directly. We may assume that $t\in(0,1)$.
Using H\"older's inequality, we have 
\begin{equation*}
    [u]_{W^{ts_{1}+(1-t)s_{2},p}(B_{r})}^{p}=\int_{B_{r}}\int_{B_{r}}\frac{|u(x)-u(y)|^{tp}}{|x-y|^{t(n+s_{1})}}\frac{|u(x)-u(y)|^{(1-t)p}}{|x-y|^{(1-t)(n+s_{2})}}\leq  [u]_{W^{s_{1},p}(B_{r})}^{tp}[u]_{W^{s_{2},p}(B_{r})}^{(1-t)p}.
\end{equation*}
By taking the power of $\frac{1}{p}$ for both sides of the above inequality, \eqref{interpesti} follows.
\end{proof}

We now prove two useful tail estimates which will be frequently employed later. 
\begin{lem}
\label{tail estimate for local notioin}
Let $u\in L^{p-1}_{sp}(\mathbb{R}^{n})$ with 
$D^{\tau}d_{\alpha+t}u\in L^{p}_{\mathrm{loc}}\left(\Omega\times\Omega\,;\mu_{\tau}\right)$ for some $t\in[0,s]$ and $\alpha\in[0,1)$, and  $B_{\rho}(y_{0})\Subset \Omega$.

\noindent
(1). For any nonnegative integer $i$ with $B_{2^{i}\rho}(y_{0})\Subset\Omega$, there is a constant $c=c(n,s,p)$ such that
\begin{equation}
\label{tail estimate of u local}
\begin{aligned}
\Tailp\left(u-(u)_{B_{\rho}(y_{0})};B_{\rho}(y_{0})\right)&\leq  c\rho^{\alpha+\tau+t}\sum_{j=1}^{i}2^{-j\left(\frac{sp}{p-1}-\left(\alpha+\tau+t\right)\right)}\left(\frac{1}{\tau}\dashint_{\mathcal{B}_{2^{j}\rho}(y_{0})}|D^{\tau}d_{\alpha+t}u|^{p}\dmut\right)^{\frac{1}{p}}\\
&\quad+c2^{-i\frac{sp}{p-1}}\Tailp\left(u-(u)_{B_{2^{i}\rho}(y_{0})};B_{2^{i}\rho}(y_{0})\right).
\end{aligned}
\end{equation}
(2). Let $B_{R}(x_{0})\Subset \Omega$ with $x_{0}\in B_{\rho}(y_{0})$ and $R>\rho$ such that $B_{\rho}(y_{0})\Subset B_{R}(x_{0})$.
Then there is a constant $c=c(n,s,p)$ such that
\begin{equation}
\label{tail estimate of u local2}
\begin{aligned}
\Tailp\left(u-(u)_{B_{\rho}(y_{0})};B_{\rho}(y_{0})\right)&\leq c\left(\frac{R}{\rho}\right)^{\frac{n}{p-1}}R^{\alpha+t+\tau}\left(\frac{1}{\tau}\dashint_{\mathcal{B}_{R}(x_{0})}\left|D^{\tau}d_{\alpha+t}u\right|^{p}\dmut\right)^{\frac{1}{p}}\\
&\quad+c\left(\frac{R^{n+sp}}{(R-\rho)^{n+sp}}\frac{\rho^{sp}}{R^{sp}}\right)^{\frac{1}{p-1}}\Tailp\left(u-(u)_{B_{R}(x_{0})};B_{R}(x_{0})\right).
\end{aligned}
\end{equation}
\end{lem}
\begin{proof}
We are first going to show \eqref{tail estimate of u local}. Since the above three terms in \eqref{tail estimate of u local} are translation invariant, we may assume that $y_{0}$ is the origin.
 Note that
\begin{equation*}
\begin{aligned}
    \Tailp\left(u-(u)_{B_{\rho}};B_{\rho}\right)^{p-1}&=\rho^{sp}\int_{\mathbb{R}^{n}\setminus B_{\rho}}\frac{|u-(u)_{B_{\rho}}|^{p-1}}{|y|^{n+sp}}\dy\\
    &=\sum_{k=0}^{i-1}\rho^{sp}\int_{B_{2^{k+1}\rho}\setminus B_{2^{k}\rho}}\frac{|u-(u)_{B_{\rho}}|^{p-1}}{|y|^{n+sp}}\dy\\
    &\quad+\rho^{sp}\int_{\mathbb{R}^{n}\setminus B_{2^{i}\rho}}\frac{|u-(u)_{B_{\rho}}|^{p-1}}{|y|^{n+sp}}\dy=\sum_{k=0}^{i-1}T_{k}+T.
\end{aligned}
\end{equation*}
Using Minkowski's inequality and H\"older's inequality, we first estimate $T_{k}^{\frac{1}{p-1}}$ as
\begin{equation*}
\begin{aligned}
    T_{k}^{\frac{1}{p-1}}&\leq c\left(2^{-ksp}\dashint_{B_{2^{k+1}\rho}}|u-(u)_{B_{\rho}}|^{p-1}\dy\right)^\frac{1}{p-1}\\
    &\leq c2^{-k\frac{sp}{p-1}}\left[\left(\dashint_{B_{2^{k+1}\rho}}|u-(u)_{B_{2^{k+1}\rho}}|^{p-1}\dy\right)^\frac{1}{p-1}+\sum_{j=0}^{k}\left|(u)_{B_{2^{j+1}\rho}}-(u)_{B_{2^{j}\rho}}\right|\right]\\
    &\leq c2^{-k\frac{sp}{p-1}}\sum_{j=1}^{k+1}\left(\dashint_{B_{2^{j}\rho}}|u-(u)_{B_{2^{j}\rho}}|^{p}\dy\right)^\frac{1}{p}.
\end{aligned}
\end{equation*}
For the last inequality, we have used the fact that
\begin{equation}
\label{meansize}
    \left|(u)_{B_{2^{j+1}\rho}}-(u)_{B_{2^{j}\rho}}\right|\leq c\left(\dashint_{B_{2^{j+1}\rho}}|u-(u)_{B_{2^{j+1}\rho}}|^{p}\dy\right)^\frac{1}{p}.
\end{equation}
Then using Jensen's inequality, the fact that
\begin{equation*}
    1\leq\frac{2^{j+1}\rho}{|x-y|}\quad\text{for any }x,y\in B_{2^{j}\rho}
\end{equation*} and \eqref{size of b mut}, we have
\begin{equation}
\label{tkestimate}
\begin{aligned}
    \left(\dashint_{B_{2^{j}\rho}}|u-(u)_{B_{2^{j}\rho}}|^{p}\dy\right)^\frac{1}{p}&\leq c\left(\dashint_{B_{2^{j}}}\dashint_{B_{2^{j}\rho}}|u(x)-u(y)|^{p}\dy\right)^\frac{1}{p}\\
    &\leq c(2^{j}\rho)^{\alpha+t}\left(\dashint_{B_{2^{j}}}\int_{B_{2^{j}\rho}}\frac{|u(x)-u(y)|^{p}}{|x-y|^{n+(\alpha+t)p}}\dy\right)^\frac{1}{p}\\
    &\leq c\frac{\left(2^{j}\rho\right)^{\alpha+\tau+t}}{\tau^{\frac{1}{p}}}\left(\dashint_{\mathcal{B}_{2^{j}\rho}}|D^{\alpha+\tau}d_{t}u|^{p}\dmut\right)^\frac{1}{p}.
\end{aligned}
\end{equation}
Since $p-1\geq 1$, we observe that
\begin{equation*}
    \left(\sum_{k=0}^{i-1}\left(T_{k}^{\frac{1}{p-1}}\right)^{p-1}\right)^{\frac{1}{p-1}}\leq \sum_{k=0}^{i-1}T_{k}^{\frac{1}{p-1}},
\end{equation*}
which implies that
\begin{equation*}
\begin{aligned}
    \left(\sum_{k=0}^{i-1}T_{k}\right)^{\frac{1}{p-1}}&\leq c\sum_{k=0}^{i-1}2^{-k\frac{sp}{p-1}}\sum_{j=1}^{k+1}\frac{\left(2^{j}\rho\right)^{\alpha+\tau+t}}{\tau^{\frac{1}{p}}}\left(\dashint_{\mathcal{B}_{2^{j}\rho}}|D^{\alpha+\tau}d_{t}u|^{p}\dmut\right)^\frac{1}{p}\\
    &\leq c\sum_{j=1}^{i}\sum_{k=j-1}^{i-1}2^{-k\frac{sp}{p-1}}\frac{\left(2^{j}\rho\right)^{\alpha+\tau+t}}{\tau^{\frac{1}{p}}}\left(\dashint_{\mathcal{B}_{2^{j}\rho}}|D^{\alpha+\tau}d_{t}u|^{p}\dmut\right)^\frac{1}{p}\\
    &\leq c\sum_{j=1}^{i}2^{-j\frac{sp}{p-1}}\frac{\left(2^{j}\rho\right)^{\alpha+\tau+t}}{\tau^{\frac{1}{p}}}\left(\dashint_{\mathcal{B}_{2^{j}\rho}}|D^{\alpha+\tau}d_{t}u|^{p}\dmut\right)^\frac{1}{p},
\end{aligned}
\end{equation*}
where we have used \eqref{tkestimate} for the first inequality, Fubini's theorem for the second inequality and then the fact that 
\begin{equation*}
    \sum_{k=j-1}^{i}2^{-k\frac{sp}{p-1}}\leq c(s,p)2^{-j\frac{sp}{p-1}}
\end{equation*}
for the last inequality. We now estimate $T^{\frac{1}{p-1}}$ as
\begin{equation*}
\begin{aligned}
    T^{\frac{1}{p-1}}&\leq \left(\rho^{sp}\int_{\mathbb{R}^{n}\setminus B_{2^{i}\rho}}\frac{|u-(u)_{B_{2^{i}\rho}}|^{p-1}}{|y|^{n+sp}}\dy\right)^{\frac{1}{p-1}}\\
    &\quad+\sum_{j=0}^{i-1}\left(\rho^{sp}\int_{\mathbb{R}^{n}\setminus B_{2^{i}\rho}}\frac{|(u)_{B_{2^{j+1}\rho}}-(u)_{B_{2^{j}\rho}}|^{p-1}}{|y|^{n+sp}}\dy\right)^{\frac{1}{p-1}},
\end{aligned}
\end{equation*}
where we have used Minkowski's inequality. We further estimate the above second term in the right-hand side as
\begin{equation*}
\begin{aligned}
    &\sum_{j=0}^{i-1}\left(\rho^{sp}\int_{\mathbb{R}^{n}\setminus B_{2^{i}\rho}}\frac{|(u)_{B_{2^{j+1}\rho}}-(u)_{B_{2^{j}\rho}}|^{p-1}}{|y|^{n+sp}}\dy\right)^{\frac{1}{p-1}}\\
    &\leq c\sum_{j=0}^{i-1}2^{\frac{-isp}{p-1}}|(u)_{B_{2^{j+1}\rho}}-(u)_{B_{2^{j}\rho}}|\\
    &\leq c\sum_{j=1}^{i}2^{\frac{-isp}{p-1}}\frac{\left(2^{j}\rho\right)^{\alpha+\tau+t}}{\tau^{\frac{1}{p}}}\left(\dashint_{\mathcal{B}_{2^{j}\rho}}|D^{\alpha+\tau}d_{t}u|^{p}\dmut\right)^\frac{1}{p},
\end{aligned}
\end{equation*}
where we have used \eqref{scalingpropertytail} for the first inequality, and \eqref{meansize}, \eqref{tkestimate} for the last inequality.
Combine all the  estimates for  $\left(\sum\limits_{k=0}^{i-1}T_{k}\right)^{\frac{1}{p-1}}$ and $T^{\frac{1}{p-1}}$ to see that
\begin{equation*}
\begin{aligned}
    \Tailp\left(u-(u)_{B_{\rho}};B_{\rho}\right)&\leq \left(\sum_{k=0}^{i-1}T_{k}+T\right)^{\frac{1}{p-1}}\\
    &\leq \left(\sum_{k=0}^{i-1}T_{k}\right)^{\frac{1}{p-1}}+T^{\frac{1}{p-1}}\\
    &\leq c\frac{\rho^{\alpha+\tau+t}}{\tau^{\frac{1}{p}}}\sum_{j=1}^{i}2^{-j\left(\frac{sp}{p-1}-\left(\alpha+\tau+t\right)\right)}\left(\dashint_{\mathcal{B}_{2^{j}\rho}}|D^{\alpha+\tau}d_{t}u|^{p}\dmut\right)^{\frac{1}{p}}\\
&\quad+c2^{-i\frac{sp}{p-1}}\Tailp\left(u-(u)_{B_{2^{i}\rho}};B_{2^{i}\rho}\right).
\end{aligned}
\end{equation*}
We are now in the position to prove \eqref{tail estimate of u local2}. Note that
\begin{equation*}
\begin{aligned}
    \Tailp\left(u-(u)_{B_{\rho}(y_{0})};B_{\rho}(y_{0})\right)&\leq \Tailp\left(u-(u)_{B_{R}(x_{0})};B_{\rho}(y_{0})\right)\\
    &\quad+\Tailp\left((u)_{B_{R}(x_{0})}-(u)_{B_{\rho}(y_{0})};B_{\rho}(y_{0})\right)\eqqcolon I_{1}+I_{2}.
\end{aligned}
\end{equation*}
We now estimate $I_{1}$ and $I_{2}$ as
\begin{align*}
    I_{1}&\leq c\left(\rho^{sp}\int_{B_{R}(x_{0})\setminus B_{\rho}(y_{0}) }\frac{|u-(u)_{B_{R}(x_{0})}|^{p-1}}{|y-y_{0}|^{n+sp}}\dy\right)^{\frac{1}{p-1}}+c\left(\rho^{sp}\int_{\mathbb{R}^{n}\setminus B_{R}(x_{0})}\frac{|u-(u)_{B_{R}(x_{0})}|^{p-1}}{|y-y_{0}|^{n+sp}}\dy\right)^{\frac{1}{p-1}}\\
    &\leq c\left(\frac{R}{\rho}\right)^{\frac{n}{p-1}}\left(\dashint_{B_{R}(x_{0})}|u-(u)_{B_{R}(x_{0})}|^{p}\dy\right)^{\frac{1}{p}}\\
    &\quad+c\left(\frac{R}{R-\rho}\right)^{\frac{n+sp}{p-1}}\left(\frac{\rho}{R}\right)^{\frac{sp}{p-1}}\left(R^{sp}\int_{\mathbb{R}^{n}\setminus B_{R}(x_{0})}\frac{|u-(u)_{B_{R}(x_{0})}|^{p-1}}{|y-x_{0}|^{n+sp}}\dy\right)^{\frac{1}{p-1}}
\end{align*}
and
\begin{align*}
    I_{2}\leq c\left|(u)_{B_{R}(x_{0})}-(u)_{B_{\rho}(y_{0})}\right|\leq c\left(\frac{R}{\rho}\right)^{\frac{n}{p}}\left(\dashint_{B_{R}(x_{0})}|u-(u)_{B_{R}(x_{0})}|^{p}\dy\right)^{\frac{1}{p}},
\end{align*}
where we have used H\"older's inequality and the fact that
\begin{equation*}
    |y-y_{0}|\geq |y-x_{0}|-|x_{0}-y_{0}|\geq \frac{R-\rho}{R}|y-x_{0}| \quad\text{any }y\in \mathbb{R}^{n}\setminus B_{R}(x_{0}).
\end{equation*}  
As in \eqref{tkestimate}, we have
\begin{equation}
\label{sobolevtau0}
\begin{aligned}
\left(\dashint_{B_{R}(x_{0})}|u-(u)_{B_{R}(x_{0})}|^{p}\dy\right)^{\frac{1}{p}}\leq cR^{\alpha+\tau+t}\left(\frac{1}{\tau}\dashint_{\mathcal{B}_{R}(x_{0})}|D^{\tau}d_{t+\alpha}u|^{p}\dmut\right)^{\frac{1}{p}}
\end{aligned}
\end{equation}
for some constant $c=c(n,s,p)$.
Combine the estimates $I_{1}$, $I_{2}$ and \eqref{sobolevtau0} to obtain \eqref{tail estimate of u local2}.
\end{proof}
We now provide the existence result of the corresponding boundary value problem to \eqref{main eq1} and the standard energy estimate.
\begin{lem} 
\label{existence}
Let $\Omega'$ be an open and bounded set in $\mathbb{R}^{n}$ such that $\Omega\Subset\Omega'$, and let $g\in W^{s,p}(\Omega')\cap L^{p-1}_{sp}(\mathbb{R}^{n})$. If $f\in L^{p-1}_{s}(\mathbb{R}^{n})$ with $d_{0}f\in L^{p}\left(\Omega'\times\Omega';\frac{\dx\dy}{|x-y|^{n}}\right)$, then there is a unique weak solution $u\in W^{s,p}(\Omega)\cap L^{p-1}_{sp}(\mathbb{R}^{n})$ to
\begin{equation}
\label{dirichlet}
\begin{cases}
    (-\Delta_{p})_{A}^{s}u=(-\Delta_{p})^{\frac{s}{p}}f&\quad\text{in }\Omega\\
    u=g&\quad\text{on }\mathbb{R}^{n}\setminus \Omega
\end{cases}
\end{equation}
with the estimate
\begin{equation}
\label{estimateofp0}
    \int_{\Omega}\int_{\Omega}|d_{s}u|^{p}\frac{\dx\dy}{|x-y|^{n}}\leq c\left(\int_{\Omega'}\int_{\Omega'}|d_{0}f|^{p}\frac{\dx\dy}{|x-y|^{n}}+\Tails(f-(f)_{\Omega'};\Omega')^{p}+\|g\|^{p}_{W^{s,p}(\Omega')}\right)
\end{equation}
for some constant $c=c(n,s,p,\Omega,\Omega')$.
\end{lem}
\begin{proof}
We first note that for any $\phi\in X_{0}^{s,p}(\Omega,\Omega')$,
\begin{equation*}
\begin{aligned}
    &\left|\int_{\mathbb{R}^{n}}\int_{\mathbb{R}^{n}}|f(x)-f(y)|^{p-2}(f(x)-f(y))(\phi(x)-\phi(y))\frac{\dx\dy}{|x-y|^{n+s}}\right|\\
    &\leq\int_{\Omega'}\int_{\Omega'}|f(x)-f(y)|^{p-1}|\phi(x)-\phi(y)|\frac{\dx\dy}{|x-y|^{n+s}}\\
    &\quad+2\int_{\mathbb{R}^{n}\setminus\Omega'}\int_{\Omega}|f(x)-f(y)|^{p-1}|\phi(x)|\frac{\dx\dy}{|x-y|^{n+s}}\eqqcolon I_{1}+I_{2}.
\end{aligned}
\end{equation*}
Using H\"older's inequality, we estimate $I_{1}$ as
\begin{equation*}
    I_{1}\leq \left(\int_{\Omega'}\int_{\Omega'}|d_{0}f|^{p}\frac{\dx\dy}{|x-y|^{n}}\right)^{\frac{p-1}{p}}[\phi]_{W^{s,p}(\Omega')}.
\end{equation*}
In light of Jensen's inequality and the fact that 
\begin{equation*}
    |x-y|\geq c(1+|y|)\quad\text{for any }x\in\Omega,\,y\in \mathbb{R}^{n}\setminus \Omega'\quad\text{and}\quad \int_{\mathbb{R}^{n}\setminus\Omega'}\frac{\dy}{|y|^{n+s}}\leq c
\end{equation*}
for some constant $c=c(n,s,\Omega,\Omega')$, we find
\begin{equation*}
\begin{aligned}
    I_{2}&\leq  c\int_{\Omega}|f(x)|^{p-1}|\phi(x)|\dx+c\int_{\mathbb{R}^{n}\setminus\Omega'}|f(y)|^{p-1}\frac{\dy}{(1+|y|)^{n+s}}\int_{\Omega}|\phi(x)|\dx\eqqcolon I_{2,1}+I_{2,2}.
\end{aligned}
\end{equation*}
We now estimate $I_{2,1}$ as
\begin{equation*}
\begin{aligned}
    I_{2,1}&\leq c\int_{\Omega}|f(x)-(f)_{\Omega}|^{p-1}|\phi(x)|\dx+c|(f)_{\Omega}|^{p-1}\int_{\Omega}|\phi(x)|\dx\\
    &\leq c\left[\left(\int_{\Omega}|f(x)-(f)_{\Omega}|^{p}\dx\right)^{\frac{p-1}{p}}+|(f)_{\Omega}|^{p-1}\right]\left(\int_{\Omega}|\phi(x)|^{p}\dx\right)^{\frac{1}{p}}\\
    &\leq c\left[\left(\int_{\Omega}\int_{\Omega}|f(x)-f(y)|^{p}\frac{\dx\dy}{|x-y|^{n}}\right)^{\frac{p-1}{p}}+|(f)_{\Omega}|^{p-1}\right]\|\phi\|_{L^{p}(\Omega)}, 
\end{aligned}
\end{equation*}
where we have used H\"older's inequality, the second inequality in \eqref{tkestimate} with $u(x)$, $B_{2^{j}\rho}$, $\alpha$ and $\tau$ there, replaced by $f(x)$, $\Omega$, 0 and 0, respectively,  and the fact that $\frac{2\mathrm{diam}(\Omega)}{|x-y|}\geq1$ for any $x,y\in \Omega$. In addition, H\"older's inequality yields
\begin{equation*}
    I_{2,2}\leq c\int_{\mathbb{R}^{n}\setminus\Omega'}|f(y)|^{p-1}\frac{\dy}{1+|y|^{n+s}}\|\phi\|_{L^{p}(\Omega)}.
\end{equation*}
Combine all the estimates $I_{1}$ and $I_{2}$ to see that the linear functional
\begin{equation*}
    T_{f}:\phi\to \int_{\mathbb{R}^{n}}\int_{\mathbb{R}^{n}}|f(x)-f(y)|^{p-2}(f(x)-f(y))(\phi(x)-\phi(y))\frac{\dx\dy}{|x-y|^{n+s}},\quad \phi\in X_{0}^{s,p}(\Omega,\Omega') 
\end{equation*}
is the element in the dual space of $X_{0}^{s,p}(\Omega,\Omega')$. Therefore, the existence of a unique weak solution follows from \cite[Proposition 2.12]{BLS2}. 

We are going to prove \eqref{estimateofp0}.
We first note from \cite[Remark A.4]{BLS2} that there is a constant $c_{p}$ depending only on $p$ such that for any $a,b\in\mathbb{R}$,
\begin{equation}
\label{structrual}
\left(|a|^{p-2}a-|b|^{p-2}b\right)(a-b)\geq\frac{1}{c_{p}}|a-b|^{p}.
\end{equation} 
We now test $u-g$ to \eqref{dirichlet}, and use \eqref{structrual} and the fact that $A\geq\Lambda^{-1}$, in order to see that
\begin{equation*}
\begin{aligned}
        J\coloneqq\int_{\mathbb{R}^{n}}\int_{\mathbb{R}^{n}}|d_{s}(u-g)|^{p}\frac{\dx\dy}{|x-y|^{n}}&\leq c\int_{\mathbb{R}^{n}}\int_{\mathbb{R}^{n}}|d_{0}f|^{p-1}|d_{s}(u-g)|\frac{\dx\dy}{|x-y|^{n}}\\
        &\quad+c\int_{\mathbb{R}^{n}}\int_{\mathbb{R}^{n}}|d_{s}g|^{p-1}|d_{s}(u-g)|\frac{\dx\dy}{|x-y|^{n}}\eqqcolon J_{1}+J_{2}.
\end{aligned}
\end{equation*}
We first note from $u=g$ on $\mathbb{R}^{n}\setminus\Omega$ that
\begin{equation*}
    J_{1}\leq c\int_{\Omega'}\int_{\Omega'}|d_{0}f|^{p-1}|d_{s}(u-g)|\frac{\dx\dy}{|x-y|^{n}}+c\int_{\mathbb{R}^{n}\setminus\Omega'}\int_{\Omega}|d_{0}f|^{p-1}|(u-g)(x)|\frac{\dx\dy}{|x-y|^{n}}\eqqcolon J_{1,1}+J_{1,2}.
\end{equation*}
Using H\"older's inequality and Young's inequality, we find that
\begin{equation*}
    J_{1,1}\leq \frac{1}{4}\int_{\Omega'}\int_{\Omega'}|d_{s}(u-g)|^{p}\frac{\dx\dy}{|x-y|^{n}}+c\int_{\Omega'}\int_{\Omega'}|d_{0}f|^{p}\frac{\dx\dy}{|x-y|^{n}}.
\end{equation*}
For the estimate of $J_{1,2}$, we follow the same arguments as in the estimate of $I_{2,2}$ in Lemma \ref{digagonal comparison} below with $B_{3}$, $B_{4}$ and $w(x)$ replaced by $\Omega$, $\Omega'$ and $g(x)$, respectively, to see that
\begin{equation*}
    J_{1,2}\leq \frac{1}{4}\int_{\Omega'}\int_{\Omega'}|d_{s}(u-g)|^{p}\frac{\dx\dy}{|x-y|^{n}}+c\Tails(f-(f)_{\Omega'};\Omega')^{p}+c\int_{\Omega'}\int_{\Omega'}|d_{0}f|^{p}\frac{\dx\dy}{|x-y|^{n}}.
\end{equation*}
Combining the estimates of $J$, $J_{1}$ and $J_{2}$, we get 
\begin{equation*}
    \int_{\Omega'}\int_{\Omega'}|d_{s}(u-g)|^{p}\frac{\dx\dy}{|x-y|^{n}}\leq c\Tails(f-(f)_{\Omega'};\Omega')^{p}+c\int_{\Omega'}\int_{\Omega'}|d_{0}f|^{p}\frac{\dx\dy}{|x-y|^{n}}.
\end{equation*}
Thus, \eqref{estimateofp0} follows by plugging the fact that 
\begin{equation*}
    \int_{\Omega'}\int_{\Omega'}|d_{s}u|^{p}\frac{\dx\dy}{|x-y|^{n}}-c(p)\int_{\Omega'}\int_{\Omega'}|d_{s}g|^{p}\frac{\dx\dy}{|x-y|^{n}}\leq c(p)\int_{\Omega'}\int_{\Omega'}|d_{s}(u-g)|^{p}\frac{\dx\dy}{|x-y|^{n}}
\end{equation*}
into the right-hand side of the above inequality.
\end{proof}

We end this section with the following technical lemma.
\begin{lem}(See \cite[Lemma 6.1]{G})
\label{technicallemma}
Let $\phi:[1,2]\to \mathbb{R}$ be a nonnegative bounded function. For $1\leq r_{1}<r_{2}\leq 2$, we assume that
\begin{equation*}
    \phi(r_{1})\leq \frac{1}{2}\phi(r_{2})+\frac{\mathcal{M}}{\left(r_{2}-r_{1}\right)^{2n}},
\end{equation*}
where $\mathcal{M}>0$. Then, 
\begin{equation*}
    \phi(1)\leq c\mathcal{M}
\end{equation*}for some constant $c=c(n)$.
\end{lem}


\section{Comparison estimates}
\label{section3}
In this section, we prove a comparison lemma based on a freezing argument which will be an important ingredient in Section \ref{section5}. Before proving this, we need to prove the following two lemmas. The first one is a self-improving property of a weak solution to the corresponding homogeneous problem of \eqref{main eq1}.
\begin{lem}
\label{higher inte}
Let $w\in W^{s,p}(B_{3})\cap L^{p-1}_{sp}(\mathbb{R}^{n})$ be a local weak solution to 
\begin{equation*}
    (-\Delta_{p})_{A}^{s}w=0\quad\text{in }B_{3}.
\end{equation*} Then there are constant $\epsilon_{0}=\epsilon_{0}(n,s,p,\Lambda)\in(0,1)$ and $c=c(n,s,p,\Lambda)$ such that
\begin{equation}
\label{selefimprovingresult}
    \left(\frac{1}{\tau}\dashint_{\mathcal{B}_{2}}|D^{\tau}d_{s}w|^{p(1+\epsilon_{0})}\dmu_{\tau}\right)^{\frac{1}{p(1+\epsilon_{0})}}\leq c\left[\left(\frac{1}{\tau}\dashint_{\mathcal{B}_{3}}|D^{\tau}d_{s}w|^{p}\dmu_{\tau}\right)^{\frac{1}{p}}+\Tailp\left(\frac{w-(w)_{B_{3}}}{3^{s+\tau}};B_{3}\right)\right].
\end{equation}  
\end{lem}
\begin{proof}
In light of \cite[Theorem 1.1, Remark 1]{BKK}, we find that there exist constants $\Tilde{\delta}\in(0,1)$ and $c$ depending only on $n,s,p$ and $\Lambda$ such that
\begin{equation}
\label{resultofself}
  [w]_{W^{s+\frac{n\Tilde{\delta}}{p(1+\Tilde{\delta})},p(1+\Tilde{\delta})}\left(B_{\frac{1}{8}(y_{0})}\right)}\leq c\left[\left(\int_{\mathcal{B}_{\frac{1}{2}}\left(y_{0}\right)}|D^{\tau}d_{s}w|^{p}\dmu_{\tau}\right)^{\frac{1}{p}}+\Tailp\left(w-(w)_{B_{\frac{1}{2}}(y_{0})};B_{\frac{1}{2}}(y_{0})\right)\right] 
\end{equation}
for any $y_{0}\in B_{2}$. Apply \eqref{tail estimate of u local2} with $u=w$, $x_{0}=0$, $\rho=\frac{1}{2}$, $R=3$, $\alpha=0$ and $t=s$ to observe that
\begin{equation}
\label{tailsobo}
\begin{aligned}
    \Tailp\left(w-(w)_{B_{\frac{1}{2}}(y_{0})};y_{0},\frac{1}{2}\right)\leq c\left[\left(\frac{1}{\tau}\dashint_{\mathcal{B}_{3}}|D^{\tau}d_{s}w|^{p}\dmu_{\tau}\right)^{\frac{1}{p}}+\Tailp(w-(w)_{B_{3}};B_{3})\right].
\end{aligned}
\end{equation}
On the other hand, there is a sufficiently small $\epsilon_{0}=\epsilon_{0}(n,s,p)<\Tilde{\delta}$ such that
\begin{equation}
\label{epsilon0cond}
     p(1+\epsilon_{0})<\begin{cases}\frac{np}{n-sp}\quad&\text{if }n>sp,\\
    2p\quad&\text{if }n\leq sp.
    \end{cases}
\end{equation}
Since $\tau<\frac{n}{p}$, we apply Lemma \ref{embedding sob} with $u=w$, $r=\frac{1}{8}$,  $s_{1}=s+\tau\frac{\epsilon_{0}}{1+\epsilon_{0}}$, $p=p(1+\epsilon_{0})$, $s_{2}=s+\frac{n\Tilde{\delta}}{p(1+\Tilde{\delta})}$ and $q=p(1+\Tilde{\delta})$ to see that 
\begin{equation}
\begin{aligned}
\label{bembeddingself}
    \left(\int_{\mathcal{B}_{\frac{1}{8}}\left(y_{0}\right)}|D^{\tau}d_{s}w|^{p(1+\epsilon_{0})}\dmut\right)^{\frac{1}{p(1+\epsilon_{0})}}=[w]_{W^{s+\tau\frac{\epsilon_{0}}{1+\epsilon_{0}},p(1+\epsilon_{0})}\left(B_{\frac{1}{8}}(y_{0})\right)}\leq c[w]_{W^{s+\frac{n\Tilde{\delta}}{p(1+\Tilde{\delta})},p(1+\Tilde{\delta})}\left(B_{\frac{1}{8}(y_{0})}\right)}.
\end{aligned}
\end{equation} We combine \eqref{resultofself}, \eqref{tailsobo} and \eqref{bembeddingself} to get 
\begin{equation}
\label{embeddingself}
\left(\int_{\mathcal{B}_{\frac{1}{8}}\left(y_{0}\right)}|D^{\tau}d_{s}w|^{p(1+\epsilon_{0})}\dmut\right)^{\frac{1}{p(1+\epsilon_{0})}}\leq c\left(\int_{\mathcal{B}_{3}}|D^{\tau}d_{s}w|^{p}\dmu_{\tau}\right)^{\frac{1}{p}}+c\Tailp\left(w-(w)_{B_{3}};B_{3}\right). 
\end{equation}
We will find similar estimates like \eqref{embeddingself} for off-diagonal cubes.
Let us take $\mathcal{Q}\equiv\mathcal{Q}_{\frac{1}{8\sqrt{n}}}\left(z_{1},z_{2}\right)\Subset \mathcal{B}_{3}$ satisfying 
\begin{equation}
\label{distcube0}
    \mathrm{dist}\left(Q^{1},Q^{2}\right)>\frac{1}{8\sqrt{n}},
\end{equation}
where $Q^{1}=Q_{\frac{1}{8\sqrt{n}}}(z_{1})$ and $Q^{2}=Q_{\frac{1}{8\sqrt{n}}}(z_{2})$.
We now set 
\begin{equation}
\label{gammaindex}
    \gamma=\begin{cases}\frac{np}{n-sp}\quad&\text{if }n>sp,\\
    2p\quad&\text{if }n\leq sp.
    \end{cases}
\end{equation}
By \eqref{epsilon0cond} and \eqref{gammaindex},  we use H\"older's inequality and then apply Lemma \ref{nd lemma} with $\tilde{\gamma}=\gamma$, $\tilde{q}=p$, $\tilde{p}=p$ and $\alpha=0$, in order to see that
\begin{equation*}
\begin{aligned}
\left(\dashint_{\mathcal{Q}}|D^{\tau}d_{s}u|^{p(1+\epsilon_{0})}\dmut\right)^{\frac{1}{p(1+\epsilon_{0})}}&\leq \left(\dashint_{\mathcal{Q}}|D^{\tau}d_{s}u|^{\gamma}\dmut\right)^{\frac{1}{\gamma}}\\
&\leq c\left(\dashint_{\mathcal{Q}}|D^{\tau}d_{s}u|^{p}\dmut\right)^{\frac{1}{p}}\\
&\quad+\frac{c}{\tau^{\frac{1}{p}}}\left(\frac{d(\mathcal{Q})}{\dist(Q^{1},Q^{2})}\right)^{s+\tau}\left[\sum_{d=1}^{2}\left(\dashint_{P^{d}\mathcal{Q}}|D^{\tau}d_{s}u|^{p}\dmut\right)^{\frac{1}{p}}\right]
\end{aligned}
\end{equation*}
for some constant $c=c(n,s,p)$. After a few simple algebraic calculations, we observe
\begin{equation}
\label{reverseoff}
\begin{aligned}
\left(\int_{\mathcal{Q}}|D^{\tau}d_{s}u|^{{p(1+\epsilon_{0})}}\dmut\right)^{\frac{1}{p(1+\epsilon_{0})}}&\leq \frac{c\mu\left(\mathcal{Q}\right)^{\frac{1}{p(1+\epsilon_{0})}}}{\mu\left(\mathcal{Q}\right)^{\frac{1}{p}}}\left(\int_{\mathcal{Q}}|D^{\tau}d_{s}u|^{p}\dmut\right)^{\frac{1}{p}}\\
&\quad+\frac{c}{\tau^{\frac{1}{p}}}\left[\sum_{d=1}^{2}\frac{\mu\left(\mathcal{Q}\right)^{\frac{1}{p(1+\epsilon_{0})}}}{\mu\left(P^{d}\mathcal{Q}\right)^{\frac{1}{p}}}\left(\int_{P^{d}\mathcal{Q}}|D^{\tau}d_{s}u|^{p}\dmut\right)^{\frac{1}{p}}\right].
\end{aligned}
\end{equation}
We now take $z_{1},z_{2}\in B_{2}$, then we have that $P^{1}\mathcal{Q}$, $P^{2}\mathcal{Q}$ and $\mathcal{Q}$ are embedded in $\mathcal{B}_{3}$, and that 
\begin{equation*}
    \frac{1}{8\sqrt{n}}<\mathrm{dist}\left(Q^{1},Q^{2}\right)<4.
\end{equation*} Thus, we further estimate the right-hand side of \eqref{reverseoff} as 
\begin{equation}
\label{cubehigher}
    \left(\int_{\mathcal{Q}}|D^{\tau}d_{s}u|^{{p(1+\epsilon_{0})}}\dmut\right)^{\frac{1}{p(1+\epsilon_{0})}}\leq c\left(\int_{\mathcal{B}_{3}}|D^{\tau}d_{s}u|^{p}\dmut\right)^{\frac{1}{p}},
\end{equation}
where we have used \eqref{measure of qintau} below and \eqref{size of b mut}.
Since $\mathcal{B}_{2}$ is covered by finitely many diagonal balls $\mathcal{B}\left(y_{i},\frac{1}{8}\right)$ and off-diagonal cubes $\mathcal{Q}_{\frac{1}{8\sqrt{n}}}\left(z_{1,i},z_{2,i}\right)$ satisfying \eqref{distcube0} with $y_{i},\,z_{1,i},\,z_{2,i}\in B_{2}$, the standard covering argument along with \eqref{embeddingself} and \eqref{cubehigher} yields
\begin{equation*}
    \left(\int_{\mathcal{B}_{2}}|D^{\tau}d_{s}w|^{p(1+\epsilon_{0})}\dmu_{\tau}\right)^{\frac{1}{p(1+\epsilon_{0})}}\leq c\left[\left(\int_{\mathcal{B}_{3}}|D^{\tau}d_{s}w|^{p}\dmu_{\tau}\right)^{\frac{1}{p}}+\Tailp\left(w-(w)_{B_{3}};B_{3}\right)\right].
\end{equation*} 
After a simple algebraic computations together with \eqref{size of b mut}, \eqref{selefimprovingresult} follows.
\end{proof}
The second one is higher H\"older regularity of weak solutions to nonlocal $p$-Laplacian type equations with locally constant kernel coefficients. 
\begin{lem}
\label{almost lipschitz}
Let $v \in W^{s,p}(B_{2})\cap L^{p-1}_{sp}(\mathbb{R}^{n})$ be a local weak solution to
\begin{equation}
\label{al lip}
(-\Delta_{p})_{A_{2}}^{s}v=0\quad\text{in }B_{2},
\end{equation}
where $A_{2}\equiv A_{2,0}$ defined in \eqref{defnofA2}.
Then $v\in C^{\alpha}(B_{1})$ for every $\alpha\in\left(0,\min\left\{1,\frac{sp}{p-1}\right\}\right)$ with the estimate
\begin{equation}
\label{average lip}
    [v]_{C^{\alpha}(B_{1})}\leq c \left(\dashint_{B_{2}}|v-(v)_{B_{2}}|^{p}\dx\right)^{\frac{1}{p}}+c\Tailp(v-(v)_{B_{2}};B_{2}),
\end{equation}
where $c=c(n,s,p,\Lambda,\alpha)$.
\end{lem}
\begin{proof}
Applying \cite[Lemma 4.1]{BKK} with $q=p$ and $b=0$, we find that
\begin{equation*}
    [v-(v)_{B_{2}}]_{C^{\alpha}(B_{1})}\leq c \left(\dashint_{B_{2}}|v-(v)_{B_{2}}|^{p}\dx\right)^{\frac{1}{p}}+c\Tailp(v-(v)_{B_{2}};B_{2}),
\end{equation*}
where we have used the fact that $v-(v)_{B_{2}}$ is also a local weak solution to \eqref{al lip}. Then the desired estimate \eqref{average lip} follows from the fact that $v-(v)_{B_{2}}$ and $v$ have the same H\"older seminorm. 
\end{proof}
Let us restrict the range of $\tau$ as
\begin{equation}
\label{firstcondoftau}
\tau\in\left(0,\min\left\{\frac{s}{p-1},1-s\right\}\right).
\end{equation} 
We note that $\tau<\frac{n}{p}$, since $n\geq2$ and $\frac{p}{p-1}\leq 2$. We now give the following comparison lemma. 
\begin{lem}
\label{digagonal comparison}
For any $\epsilon>0$, there is a constant $\delta=\delta(n,s,p,\Lambda,\epsilon)$ such that for any weak solution $u$ to \eqref{main eq1} with
\begin{equation*}
    B_{4}\subset\Omega
\end{equation*}and
\begin{equation}
\begin{aligned}
\label{scaling assumption}
    &\frac{1}{\tau}\dashint_{\mathcal{B}_{4}}|D^{\tau}d_{s}u|^{p}\dmu_{\tau}+\Tailp\left(\frac{u-(u)_{B_{4}}}{4^{\tau+s}};B_{4}\right)^{p}\leq 1,\\
    &\frac{1}{\tau}\dashint_{\mathcal{B}_{4}}|D^{\tau}d_{0}f|^{p}\dmu_{\tau}+\Tails\left(\frac{f-(f)_{B_{4}}}{4^{\tau}};B_{4}\right)^{p}+\left(\dashint_{B_{2}}\dashint_{B_{2}}\left|A(x,y)-(A)_{B_{2}\times B_{2}}\right|\dx\dy\right)^{p}\leq \delta^{p},
\end{aligned}
\end{equation}
there exists a weak solution $v\in W^{s,p}(B_{2})\cap L^{p-1}_{sp}(\mathbb{R}^{n})$ to 
\begin{equation*}
    (-\Delta_{p})_{A_{2}}^{s}v=0\quad\text{in }B_{2}
\end{equation*}
such that 
\begin{equation}
\label{comp main res}
    \frac{1}{\tau}\dashint_{\mathcal{B}_{1}}|D^{\tau}d_{s}(u-v)|^{p}\dmu_{\tau}\leq\epsilon^{p}\quad\text{and}\quad\|D^{\tau}d_{s}v\|_{L^{\infty}(\mathcal{B}_{1})}\leq c_{0}
\end{equation}
for some constant $c_{0}=c_{0}(n,s,p,\Lambda,\tau)$.
\end{lem}
\begin{proof}
From \cite[Proposition 2.12]{BLS2}, there is a unique weak solution $w\in X^{s,p}_{u}(B_{3},B_{4})\cap L^{p-1}_{sp}(\mathbb{R}^{n})$ to 
\begin{equation*}
\begin{cases}
    (-\Delta_{p})_{A}^{s}w=0&\quad\text{in }B_{3},\\
    w=u&\quad\text{on }\mathbb{R}^{n}\setminus B_{3}.
\end{cases}
\end{equation*}
Then we observe that
\begin{equation*}
\begin{aligned}
    \left\langle (-\Delta_{p})_{A}^{s}u -(-\Delta_{p})_{A}^{s}w,u-w\right\rangle=\left\langle (-\Delta_{p})^{\frac{s}{p}}f,u-w\right\rangle,
\end{aligned}
\end{equation*}
 where 
\begin{equation*}
\begin{aligned}
    &\left\langle (-\Delta_{p})_{A}^{s}u ,u-w\right\rangle\\
    &=\int_{\mathbb{R}^{n}}\int_{\mathbb{R}^{n}}A(x,y)\left(\frac{|u(x)-u(y)|}{|x-y|^{s}}\right)^{p-2}\left(\frac{u(x)-u(y)}{|x-y|^{s}}\right)\left(\frac{(u-w)(x)-(u-w)(y)}{|x-y|^{s}}\right)\frac{\dx\dy}{|x-y|^{n}}.
\end{aligned}
\end{equation*}
We now estimate $I_{1}$ and $I_{2}$.

\noindent
\textbf{Estimate of $I_{1}$.} 
In light of \eqref{size of b mut} and \eqref{structrual}, we have
\begin{equation*}
\begin{aligned}
I_{1}\geq \frac{1}{c_{p}\Lambda}\int_{\mathbb{R}^{n}}\int_{\mathbb{R}^{n}}\frac{|(u-w)(x)-(u-w)(y)|^{p}}{|x-y|^{n+sp}}\dx\dy\geq\frac{\nu_{0}}{c_{p}\Lambda\tau}\dashint_{\mathcal{B}_{4}}|D^{\tau}d_{s}(u-w)|^{p}\dmu_{\tau}.
\end{aligned}
\end{equation*}
\textbf{Estimate of $I_{2}$.} By the fact that $u=w$ on $\mathbb{R}^{n}\setminus B_{3}$, we first estimate $I_{2}$ as
\begin{equation*}
\begin{aligned}
I_{2}&\leq \int_{B_{4}}\int_{B_{4}}\frac{|d_{0}f(x,y)|^{p-1}|(u-w)(x)-(u-w)(y)|}{|x-y|^{n+s}}\dx\dy\\
&\quad+2\int_{\mathbb{R}^{n}\setminus B_{4}}\int_{B_{3}}\frac{|d_{0}f(x,y)|^{p-1}|(u-w)(x)|}{|x-y|^{n+s}}\dx\dy\eqqcolon I_{2,1}+2I_{2,2}.
\end{aligned}
\end{equation*}
Applying Young's inequality to $I_{2,1}$, we find that there is a constant $c=c(n,s,p,\Lambda)$ such that
\begin{equation*}
    I_{2,1}\leq \frac{c}{\tau}\dashint_{\mathcal{B}_{4}}|D^{\tau}d_{0}f|^{p}\dmut+\frac{1}{4}\frac{\nu_{0}}{c_{p}\Lambda\tau}\dashint_{\mathcal{B}_{4}}|D^{\tau}d_{s}(u-w)|^{p}\dmut.
\end{equation*}
To estimate of $I_{2,2}$, we first observe
\begin{equation*}
\begin{aligned}
I_{2,2}&\leq \int_{\mathbb{R}^{n}\setminus B_{4}}\int_{B_{3}}\frac{|f(x)-(f)_{B_{4}}|^{p-1}|(u-w)(x)|}{|x-y|^{n+s}}\dx\dy\\
&\quad+\int_{\mathbb{R}^{n}\setminus B_{4}}\int_{B_{3}}\frac{|f(y)-(f)_{B_{4}}|^{p-1}|(u-w)(x)|}{|x-y|^{n+s}}\dx\dy\eqqcolon I_{2,2,1}+I_{2,2,2}.
\end{aligned}
\end{equation*}
By the fact that
\begin{equation*}
    \int_{\mathbb{R}^{n}\setminus B_{4}}\frac{1}{|x-y|^{n+s}}\dy\leq c(n,s)\quad\text{for any }x\in B_{3}
\end{equation*}
and Young's inequality, we show that for any $\varsigma\in(0,1)$,
\begin{equation}
\label{i221esti}
\begin{aligned}
    I_{2,2,1}&\leq \frac{c}{\varsigma^{\frac{1}{p-1}}}\int_{B_{4}}|f(x)-(f)_{B_{4}}|^{p}\dx+\varsigma\int_{B_{4}}|u(x)-w(x)|^{p}\dx.
\end{aligned}
\end{equation}
We next apply \eqref{tkestimate} and Lemma \ref{SPIL2} to the first term and the second term in the right-hand side of \eqref{i221esti}, respectively, in order to get that
\begin{equation*}
    I_{2,2,1}\leq \frac{c}{\tau}\dashint_{\mathcal{B}_{4}}|D^{\tau}d_{0}f|^{p}\dmut+ \frac{1}{4}\frac{\nu_{0}}{c_{p}\Lambda\tau}\dashint_{\mathcal{B}_{4}}|D^{\tau}d_{s}(u-w)|^{p}\dmu_{\tau},
\end{equation*}
where we take $\varsigma$ so small depending on $n,s,p$ and $\Lambda$. By the fact that $|x-y|\geq c|y|$ for any $x\in B_{3}$ and $y\in \mathbb{R}^{n}\setminus B_{4}$, we estimate $I_{2,2,2}$ as
\begin{equation*}
\begin{aligned}
    I_{2,2,2}\leq \int_{\mathbb{R}^{n}\setminus B_{4}}\int_{B_{3}}\frac{|f(y)-(f)_{B_{4}}|^{p-1}|(u-w)(x)|}{|x-y|^{n+s}}\dx\dy.
\end{aligned}
\end{equation*}
As in \eqref{i221esti}, Young's inequality and Lemma \ref{SPIL} yield that
\begin{equation*}
    I_{2,2,2}\leq c\Tails\left(f-(f)_{B_{4}};B_{4}\right)^{p}+\frac{1}{4}\frac{\nu_{0}}{c_{p}\Lambda\tau}\dashint_{\mathcal{B}_{4}}|D^{\tau}d_{s}(u-w)|^{p}\dmut.
\end{equation*}
Combine all the estimates $I_{1}$ and $I_{2}$ to see that
\begin{equation}
\label{uminusw}
    \frac{1}{\tau}\dashint_{\mathcal{B}_{4}}|D^{\tau}d_{s}(u-w)|^{p}\dmut\leq c\left(\frac{1}{\tau}\dashint_{\mathcal{B}_{4}}|D^{\tau}d_{0}f|^{p}\dmut+ \Tails\left(f-(f)_{B_{4}};B_{4}\right)^{p}\right)\leq c\delta^{p}.
\end{equation}
where we have used \eqref{scaling assumption} for the last inequality. We next apply Lemma \ref{higher inte} to see that there exists a small number $\epsilon_{0}\in(0,1)$ depending only on  $n,s,p$ and $\Lambda$ such that 
\begin{equation*}
    \left(\frac{1}{\tau}\dashint_{\mathcal{B}_{2}}|D^{\tau}d_{s}w|^{p(1+\epsilon_{0})}\dmu_{\tau}\right)^{\frac{1}{p(1+\epsilon_{0})}}\leq c\left[\left(\frac{1}{\tau}\dashint_{\mathcal{B}_{3}}|D^{\tau}d_{s}w|^{p}\dmu_{\tau}\right)^{\frac{1}{p}}+\Tailp(w-(w)_{B_{3}};{B_{3}})\right]
\end{equation*}
for some constant $c=c(n,s,p,\Lambda)$. Note that using Minkowski's inequality, \eqref{uminusw} and \eqref{scaling assumption}, we have that
\begin{equation*}
    \left(\frac{1}{\tau}\dashint_{\mathcal{B}_{3}}|D^{\tau}d_{s}w|^{p}\dmu_{\tau}\right)^{\frac{1}{p}}\leq \left(\frac{1}{\tau}\dashint_{\mathcal{B}_{3}}\left|D^{\tau}d_{s}(w-u)\right|^{p}\dmu_{\tau}\right)^{\frac{1}{p}}+\left(\frac{1}{\tau}\dashint_{\mathcal{B}_{3}}|D^{\tau}d_{s}u|^{p}\dmu_{\tau}\right)^{\frac{1}{p}}\leq c
\end{equation*}
and
\begin{equation*}
\begin{aligned}
    \Tailp(w-(w)_{B_{3}};B_{3})\leq \Tailp(w-u-(w-u)_{B_{3}};B_{3})+\Tailp(u-(u)_{B_{3}};B_{3})\eqqcolon T_{1}+T_{2}.
\end{aligned}
\end{equation*}
\textbf{Estimate of $T_{1}$.}
Since $u=w$ on $\mathbb{R}^{n}\setminus B_{3}$, it follows from \eqref{scalingpropertytail}, Lemma \ref{SPIL2}, \eqref{size of b mut} and \eqref{uminusw} that
\begin{equation*}
\begin{aligned}
    T_{1}&=\Tailp((w-u)_{B_{3}};B_{3})= \nu_{1}|(w-u)_{B_{3}}|\leq \frac{c}{\tau}\dashint_{\mathcal{B}_{3}}|D^{\tau}d_{s}(w-u)|^{p}\dmut\leq c.
\end{aligned}
\end{equation*}

\noindent
\textbf{Estimate of $T_{2}$.} Applying \eqref{tail estimate of u local2} with $y_{0}=0$, $x_{0}=0$, $\rho=3$, $R=4$, $\alpha=0$ and  $t=s$ leads to
\begin{equation*}
\begin{aligned}
T_{2}&\leq\Tailp(u-(u)_{B_{4}};B_{4})+\left(\frac{c}{\tau}\dashint_{\mathcal{B}_{4}}|D^{\tau}d_{s}u|^{p}\dmut\right)^{\frac{1}{p}}\leq c.
\end{aligned}
\end{equation*}
For the last inequality, we have used \eqref{scaling assumption}.
Therefore, we have 
\begin{equation}
\label{higher const}
    \left(\frac{1}{\tau}\dashint_{\mathcal{B}_{2}}|D^{\tau}d_{s}w|^{p(1+\epsilon_{0})}\dmu_{\tau}\right)^{\frac{1}{p(1+\epsilon_{0})}}\leq c
\end{equation}
for some constant $c=c(n,s,p,\Lambda)$. We now note from \cite[Proposition 2.12]{BLS2} that there is a unique weak solution $v\in X_{w}^{s,p}(B_{2},B_{3})\cap L^{p-1}_{sp}(\mathbb{R}^{n})$ to 
\begin{equation*}
\begin{cases}
    (-\Delta_{p})_{A_{2}}^{s}v=0&\quad\text{in }B_{2},\\
    v=w&\quad\text{on }\mathbb{R}^{n}\setminus B_{2}.
\end{cases}
\end{equation*}
Therefore, we observe that
\begin{equation*}
    \left\langle (-\Delta_{p})_{A}^{s}w-(-\Delta_{p})_{A_{2}}^{s}v,v-w\right\rangle=0
\end{equation*}
and rearrange it as
\begin{equation*}
\begin{aligned}
    J_{1}\coloneqq\left\langle(-\Delta_{p})_{A_{2}}^{s}w-(-\Delta_{p})_{A_{2}}^{s}v,v-w\right\rangle=-\left\langle(-\Delta_{p})_{A}^{s}w-(-\Delta_{p})_{A_{2}}^{s}w,v-w\right\rangle\eqqcolon J_{2}.
\end{aligned}
\end{equation*}
We first note from \eqref{structrual} that
\begin{equation*}
    J_{1}\geq \frac{\nu_{0}}{c_{p}\Lambda\tau}\dashint_{\mathcal{B}_{2}}|D^{\tau}d_{s}(w-v)|^{p}\dmut.
\end{equation*}
\textbf{Estimate of $J_{2}.$}
Using the definition of $A_{2}$ given in \eqref{defnofA2}, we next estimate $J_{2}$ as 
\begin{equation*}
\begin{aligned}
    J_{2}&\leq \int_{B_{2}}\int_{B_{2}}\frac{|(w(x)-w(y))|^{p-1}|A-A_{2}|||(w-v)(x)-(w-v)(y)|}{|x-y|^{n+sp}}\dx\dy.
\end{aligned}
\end{equation*}
We then use H\"older's inequality, \eqref{size of b mut}, Young's inequality, \eqref{scaling assumption} and \eqref{higher const} to the above estimate so as to obtain that
\begin{equation*}
\begin{aligned}
J_{2}&\leq c\left(\frac{1}{\tau}\dashint_{\mathcal{B}_{2}}|D^{\tau}d_{s}w|^{p(1+\epsilon_{0})}\dmu_{\tau}\right)^{\frac{p-1}{p(1+\epsilon_{0})}}\left(\dashint_{B_{2}}\dashint_{B_{2}}|A-A_{2}|^{\frac{p}{p-1}\frac{1+\epsilon_{0}}{\epsilon_{0}}}\dx\dy\right)^{\frac{p-1}{p}\frac{\epsilon_{0}}{1+\epsilon_{0}}}\\
&\quad\times\left(\frac{1}{\tau}\dashint_{\mathcal{B}_{2}}|D^{\tau}d_{s}(w-v)|^{p}\dmut\right)^{\frac{1}{p}}\\
&\leq \frac{1}{4}\frac{\nu_{0}}{c_{p}\Lambda\tau}\dashint_{\mathcal{B}_{2}}|D^{\tau}d_{s}(w-v)|^{p}\dmut+c\delta^{\nu}
\end{aligned}
\end{equation*}
for some constant $\nu=\nu(n,s,p,\Lambda)>0$. Thus, we get that
\begin{equation}
\label{wminusv}
    \frac{1}{\tau}\dashint_{\mathcal{B}_{2}}|D^{\tau}d_{s}(w-v)|^{p}\dmut\leq c\delta^{\nu}.
\end{equation}
Taking $\delta$ sufficiently small depending on $n,s,p,\Lambda$ and $\epsilon$, we have that
\begin{equation}
\begin{aligned}
\label{UminusVd}
    \frac{1}{\tau}\dashint_{\mathcal{B}_{2}}|D^{\tau}d_{s}(u-v)|^{p}\dmut&\leq \frac{c}{\tau}\dashint_{\mathcal{B}_{2}}|D^{\tau}d_{s}(u-v)|^{p}\dmut \\
    &\leq \frac{c}{\tau}\left(\dashint_{\mathcal{B}_{2}}|D^{\tau}d_{s}(u-w)|^{p}\dmut+\dashint_{\mathcal{B}_{2}}|D^{\tau}d_{s}(w-v)|^{p}\dmut\right)\\
    &\leq c\delta^p+c\delta^\nu\leq \epsilon^{p},
\end{aligned}
\end{equation}
where we have used Jensen's inequality, \eqref{uminusw} and \eqref{wminusv}. 
We now use Lemma \ref{almost lipschitz} to see that 
\begin{equation}
\label{tailvandvp}
\begin{aligned}
    \|D^{\tau}d_{s}v\|_{L^{\infty}(\mathcal{B}_{1})}&\leq c[v]_{C^{s+\tau}(\mathcal{B}_{1})}\leq c \left(\dashint_{B_{2}}|v-(v)_{B_{2}}|^{p}\dx\right)^{\frac{1}{p}}+c\Tailp(v-(v)_{B_{2}};B_{2})
\end{aligned}
\end{equation}
for some constant $c=c(n,s,p,\Lambda,\tau)$.
As in \eqref{tkestimate} along with \eqref{higher const} and \eqref{wminusv}, we get that
\begin{equation}
\label{vspi}
\left(\dashint_{B_{2}}|v-(v)_{B_{2}}|^{p}\dx\right)^{\frac{1}{p}}\leq c\left(\frac{1}{\tau}\dashint_{\mathcal{B}_{2}}|D^{\tau}d_{s}v|^{p}\dmut\right)^{\frac{1}{p}}\leq c.
\end{equation}
Like the proof for estimates of $T_{1}$ and $T_{2}$, we obtain
\begin{equation*}
\begin{aligned}
\Tailp(v-(v)_{B_{2}};B_{2})&\leq \Tailp((v-w)-(v-w)_{B_{2}};B_{2})\\
    &\quad+\Tailp((w-u)-(w-u)_{B_{2}};B_{2})+\Tailp(u-(u)_{B_{2}};B_{2})\\
    &\leq c\left(|(v-w)_{B_{2}}|+|(u-w)_{B_{2}}|\right)\\
    &\quad+\Tailp((w-u);B_{2})+\Tailp(u-(u)_{B_{2}};B_{2}).
\end{aligned}
\end{equation*}
We note that
\begin{equation}
\label{uminuswspl2}
    |(u-w)_{B_{2}}|\leq c\int_{B_{3}}|(u-w)(x)|^{p}\dx\leq  \frac{c}{\tau}\dashint_{\mathcal{B}_{3}}|D^{\tau}d_{s}(u-w)|^{p}\dmut\leq c
\end{equation}
and
\begin{equation*}
    |(v-w)_{B_{2}}|\leq \frac{1}{\tau}\dashint_{\mathcal{B}_{2}}|D^{\tau}d_{s}(w-v)|^{p}\dmut\leq c,
\end{equation*}
where we have used Lemma \ref{SPIL2}, \eqref{uminusw} and \eqref{wminusv}. In light of \eqref{uminuswspl2}, \eqref{tail estimate of u local} and \eqref{scaling assumption}, we find 
\begin{equation*}
\begin{aligned}
    \Tailp((w-u);B_{2})+\Tailp(u-(u)_{B_{2}};B_{2})&\leq  c\int_{B_{3}}|(u-w)(x)|^{p}\dx\\
    &\quad+\frac{c}{\tau}\dashint_{\mathcal{B}_{4}}|D^{\tau}d_{s}u|^{p}\dmut+c\Tailp(u-(u)_{B_{4}};B_{4})\\
    &\leq c.
\end{aligned}
\end{equation*}
Therefore, combining the above three inequalities leads to
\begin{equation}
\label{vtail}
\Tailp(v-(v)_{B_{2}};B_{2})\leq c.
\end{equation}
Taking into account \eqref{tailvandvp}, \eqref{vspi} and \eqref{vtail}, we conclude the second inequality in \eqref{comp main res}. This completes the proof.
\end{proof}

Using scaling and translation arguments, we have the non-scaled version of Lemma \ref{digagonal comparison}.
\begin{lem}
\label{digagonal comparison2}
For any $\epsilon>0$, there is a constant $\delta=\delta(n,s,p,\Lambda,\epsilon)$ such that for any weak solution $u$ to \eqref{main eq1} with
\begin{equation*}
    B_{20\rho_{x_{i}}}(x_{i})\subset\Omega
\end{equation*}and
\begin{equation}
\label{nscaling assumption}
\begin{aligned}
    &\frac{1}{\tau}\dashint_{\mathcal{B}_{20\rho_{x_{i}}}(x_{i})}|D^{\tau}d_{s}u|^{p}\dmu_{\tau}+\Tailp\left(\frac{u-(u)_{B_{20\rho_{x_{i}}}(x_{i})}}{(20\rho_{x_{i}})^{\tau+s}};B_{20\rho_{x_{i}}}(x_{I})\right)^{p}\leq \lambda^{p},\\
    &\frac{1}{\tau}\dashint_{\mathcal{B}_{20\rho_{x_{i}}}(x_{i})}|D^{\tau}d_{0}f|^{p}\dmu_{\tau}+\Tails\left(\frac{f-(f)_{B_{20\rho_{x_{i}}}(x_{i})}}{(20\rho_{x_{i}})^{\tau}};B_{20\rho_{x_{i}}}(x_{i})\right)^{p}\\
    &\quad+\left(\dashint_{B_{10\rho_{x_{i}}}(x_{i})}\dashint_{B_{10\rho_{x_{i}}}(x_{i})}\lambda\left|A(x,y)-(A)_{B_{10\rho_{x_{i}}}(x_{i})\times B_{10\rho_{x_{i}}}(x_{i})}\right|\dx\dy\right)^{p}\leq (\delta\lambda)^{p},
\end{aligned}
\end{equation}
there exists a weak solution $v\in W^{s,p}(B_{10\rho_{x_{i}}}(x_{i}))\cap L^{p-1}_{sp}(\mathbb{R}^{n})$ to 
\begin{equation}
\label{ncompsol}
    (-\Delta_{p})_{A_{10\rho_{x_{i}},x_{i}}}^{s}v=0\quad\text{in }B_{10\rho_{x_{i}}}(x_{i})
\end{equation}
such that 
\begin{equation}
\label{ncomp main res}
    \frac{1}{\tau}\dashint_{\mathcal{B}_{5\rho_{x_{i}}}(x_{i})}|D^{\tau}d_{s}(u-v)|^{p}\dmu_{\tau}\leq(\epsilon\lambda)^{p}\quad\text{and}\quad\|D^{\tau}d_{s}v\|_{L^{\infty}(\mathcal{B}_{5\rho_{x_{i}}}(x_{i}))}\leq c_{0}\lambda
\end{equation}
for some constant $c_{0}=c_{0}(n,s,p,\Lambda,\tau)$.
\end{lem}
\begin{proof}
Let us define the scaled functions for $x,y\in\mathbb{R}^{n}$,
\begin{equation*}
    \tilde{u}(x)=\frac{u(5\rho_{x_{i}}x+x_{i})}{\left(5\rho_{x_{i}}\right)^{\tau+s}\lambda},\quad\tilde{f}(x)=\frac{f(5\rho_{x_{i}}x+x_{i})}{\left(5\rho_{x_{i}}\right)^{\tau}\lambda},\quad\text{and}\quad\tilde{A}(x,y)=A(5\rho_{x_{i}}x+x_{i},5\rho_{x_{i}}y+x_{i})
\end{equation*}
to see that 
\begin{equation*}
    (-\Delta_{p})_{\tilde{A}}^{s}\tilde{u}=(-\Delta_{p})^{\frac{s}{p}}\tilde{f}\quad\text{in }B_{4}
\end{equation*}
and \eqref{scaling assumption} with $u=\tilde{u},\, f=\tilde{f}$ and $A=\tilde{A}$ by \eqref{nscaling assumption}. Therefore, using Lemma \ref{digagonal comparison}, there is a weak solution $\tilde{v}\in W^{s,p}(B_{2})\cap L^{p-1}_{sp}(\mathbb{R}^{n})$ to 
\begin{equation*}
    (-\Delta_{p})_{\tilde{A}_{2}}^{s}\tilde{v}=0\quad\text{in }B_{2}
\end{equation*}
such that \eqref{comp main res} with $u=\tilde{u}$ and $v=\tilde{v}$ holds. Taking $v(x)=\lambda \left(5\rho_{x_{i}}\right)^{\tau+s}\tilde{v}\left(\frac{x-x_{i}}{5\rho_{x_{i}}}\right)$ for $x\in\mathbb{R}^{n}$ and changing variables, we obtain that $v\in W^{s,p}(B_{10\rho_{x_{i}}}(x_{i}))\cap L^{p-1}_{sp}(\mathbb{R}^{n})$ is a weak solution to 
\eqref{ncompsol} satisfying \eqref{ncomp main res}.

\end{proof}



\section{Coverings of upper level sets}
\label{Section4}


\noindent
In this section, we construct coverings for upper level sets of $D^{\tau}d_{\alpha+s}u$ under the following assumption:
\begin{equation}
\label{ind assum}
\begin{aligned}    \int_{\mathcal{B}_{2}}|D^{\tau}d_{\alpha+s}u|^{\tilde{p}}+|D^{\tau}d_{\alpha}f|^{\tilde{q}}\dmut<\infty
\end{aligned}
\end{equation}
for some
\begin{equation}
\label{tildeptildeqalphacond}
\tilde{p},\,\tilde{q}\in [p,\infty)\text{ and }\alpha\in\left[0,1\right).
\end{equation}
We note that the the given parameter $\alpha$ will be used to improve the range of $\sigma$ as in \eqref{mainassumptions} (see Section \ref{section52} below). Let us fix $\delta\in(0,1]$.
To handle the upper level set of $D^{\tau}d_{\alpha+s}u$, we define for any $x\in B_{2}$ and $\rho\in(0,2-|x|]$,
\begin{equation}
\begin{aligned}
\label{defn of theta}
\Theta\left(x,\rho\right)&=\left(\dashint_{\mathcal{B}_{\rho}(x)}|D^{\tau}d_{\alpha+s}u|^{\tilde{p}}\dmu_{\tau}\right)^{\frac{1}{\tilde{p}}}+\frac{1}{\delta}\left(\dashint_{\mathcal{B}_{\rho}(x)}|D^{\tau}d_{\alpha}f|^{\tilde{q}}\dmu_{\tau}\right)^{\frac{1}{\tilde{q}}}.
\end{aligned}
\end{equation}
In this setting, based on the techniques given in \cite{KMS1}, we now prove the following lemma which is a modified version of \cite[Proposition 5.1]{KMS1}. 
\begin{lem}
\label{coveringoflevelset}
Let $u\in L^{p-1}_{sp}(\mathbb{R}^{n})$ and $f\in L^{p-1}_{s}(\mathbb{R}^{n})$ satisfy \eqref{ind assum} along with \eqref{tildeptildeqalphacond}, and let $1\leq r_{1}<r_{2}\leq2$. Then there are families of countable disjoint diagonal balls and off-diagonal cubes,  $\left\{\mathcal{B}_{\rho_{x_{i}}}\left(x_{i}\right)\right\}_{i\in \mathbb{N}}$ and $\left\{\mathcal{Q}\right\}_{\mathcal{Q}\in \mathcal{A}}$, respectively, such that
\begin{equation}
\label{coveringforelambda}
    E_{\lambda}\coloneqq\left\{(x,y)\in \mathcal{B}_{r_{1}}\,;\,|D^{\tau}d_{\alpha+s}u|(x,y)>\lambda\right\}\subset\left(\bigcup\limits_{i\in\mathbb{N}}\mathcal{B}_{5\rho_{x_{i}}}\left(x_{i}\right)\right)\bigcup \left(\bigcup\limits_{\mathcal{Q}\in\mathcal{A}}\mathcal{Q}\right)
\end{equation}
whenever $\lambda\geq \lambda_{0}$, where
\begin{equation}
\label{lambda0}
\begin{aligned}
    \lambda_{0}&\coloneqq \frac{c_{d}}{\tau^{\frac{1}{p}}}\left(\frac{20}{r_{2}-r_{1}}\right)^{2n}\left(\left(\dashint_{\mathcal{B}_{2}}|D^{\tau}d_{\alpha+s}u|^{\tilde{p}}\dmut\right)^{\frac{1}{\tilde{p}}}+\Tailp\left(\frac{u-(u)_{B_{2}}}{2^{\alpha+\tau+s}};B_{2}\right)\right)\\
&\quad+\frac{c_{d}}{\tau^{\frac{1}{p}}}\left(\frac{20}{r_{2}-r_{1}}\right)^{2n}\frac{1}{\delta}\left(\left(\dashint_{\mathcal{B}_{2}}|D^{\tau}d_{\alpha}f|^{\tilde{q}}\dmu_{\tau}\right)^{\frac{1}{\tilde{q}}}+\Tails\left(\frac{f-(f)_{B_{2}}}{2^{\alpha+\tau}};B_{2}\right)\right)
\end{aligned}
\end{equation} for some constant $c_{d}=c_{d}(n,s,p)$ independent of $\tau$ and $\alpha$.  In particular, we have that
\begin{equation}
\label{exitradius}
    \Theta(x_{i},\rho_{x_{i}})\geq\frac{\tau^{\frac{1}{p}}}{c_{d}}\lambda\quad\text{and}\quad\Theta(x_{i},\rho)\leq\frac{\tau^{\frac{1}{p}}}{c_{d}}\lambda\quad\text{ if }\rho_{x_{i}}\leq\rho\leq\frac{r_{2}-r_{1}}{10}\,\,(i=1,2,\ldots),
\end{equation}
\begin{equation}
\label{gamma norm}
\left(\dashint_{\mathcal{Q}}|D^{\tau}d_{\alpha+s}u|^{\tilde{\gamma}}\dmut\right)^{\frac{1}{\tilde{\gamma}}}\leq \frac{c_{oh}}{\tau^{\frac{1}{p}}}\lambda\quad\text{for any }\mathcal{Q}\in\mathcal{A}
\end{equation}
and
\begin{equation}
\label{sumofmeasq}
    \sum_{\mathcal{Q}\in\mathcal{A}}\mu_{\tau}(\mathcal{Q})\leq\frac{c_{od}}{\lambda^{\tilde{p}}}\int_{\mathcal{B}_{r_{2}}\cap\left\{\,|D^{\tau}d_{\alpha+s}u|>c_{u}\lambda\right\}}|D^{\tau}d_{\alpha+s}u|^{\tilde{p}}\dmut
\end{equation}
for some constants $c_{oh}=c_{oh}(n,s,p,\tilde{p},\tilde{\gamma})$ and $c_{od}=c_{od}(n,s,p,\tilde{p})$ independent of $\tau$ and $\alpha$, where $c_{u}=\frac{\tau^{\frac{1}{p}}}{c_{d}}$ and
\begin{equation}
\label{condoftildegamma}
\tilde{\gamma}=\begin{cases}
    \frac{n\tilde{p}}{n-s\tilde{p}}&\quad\text{if }s\tilde{p}<n,\\
    2q&\quad\text{if }s\tilde{p}\geq n.
\end{cases}
\end{equation}
\end{lem}

\begin{proof}
    
\noindent
\textbf{Step 1. Diagonal balls and Vitali's covering.} We first find diagonal balls satisfying \eqref{exitradius} using the exit time argument and Vitali's covering lemma.
Let us consider a constant $\lambda$ satisfying
\begin{equation}
\label{cond of lambda0}
\begin{aligned}
\lambda&\geq\frac{1}{\kappa}\left(\frac{20}{r_{2}-r_{1}}\right)^{2n}\left(\left(\dashint_{\mathcal{B}_{2}}|D^{\tau}d_{\alpha+s}u|^{\tilde{p}}\dmut\right)^{\frac{1}{\tilde{p}}}+\Tailp\left(\frac{u-(u)_{B_{2}}}{2^{\alpha+\tau+s}};B_{2}\right)\right)\\
&\quad+\frac{1}{\kappa}\left(\frac{20}{r_{2}-r_{1}}\right)^{2n}\frac{1}{\delta}\left(\left(\dashint_{\mathcal{B}_{2}}|D^{\tau}d_{\alpha}f|^{\tilde{q}}\dmu_{\tau}\right)^{\frac{1}{\tilde{q}}}+\Tails\left(\frac{f-(f)_{B_{2}}}{2^{\alpha+\tau}};B_{2}\right)\right),
\end{aligned}
\end{equation}
where a free parameter $\kappa\in(0,1]$ will be selected later. We next define
\begin{equation*}
D_{\kappa\lambda}=\left\{(x,x)\in \mathcal{B}_{r_{1}}\,;\, \sup_{0<\rho\leq\frac{r_{2}-r_{1}}{10}}\Theta(x,\rho)>\kappa\lambda\right\}.
\end{equation*}
We now note from \eqref{size of b mut} that for any $x\in B_{r_{1}}$,
\begin{equation}
\label{thetau}
\begin{aligned}
    \left(\dashint_{\mathcal{B}_{\frac{r_{2}-r_{1}}{10}}\left(x\right)}|D^{\tau}d_{\alpha+s}u|^{\tilde{p}}\dmu_{\tau}\right)^{\frac{1}{\tilde{p}}}&\leq \left(\frac{\mu_{\tau}\left(\mathcal{B}_{2}\right)}{\mu_{\tau}\left(\mathcal{B}_{\frac{r_{2}-r_{1}}{10}}\left(x\right)\right)}\dashint_{\mathcal{B}_{2}}|D^{\tau}d_{\alpha+s}u|^{\tilde{p}}\dmu_{\tau}\right)^{\frac{1}{\tilde{p}}}\\
    &\leq \left(\left(\frac{20}{r_{2}-r_{1}}\right)^{2n}\dashint_{\mathcal{B}_{2}}|D^{\tau}d_{\alpha+s}u|^{\tilde{p}}\dmu_{\tau}\right)^{\frac{1}{\tilde{p}}}
\end{aligned}
\end{equation}
and
\begin{equation}
\label{thetaf}
\begin{aligned}
    \frac{1}{\delta}\left(\dashint_{\mathcal{B}_{\frac{r_{2}-r_{1}}{10}}\left(x\right)}|D^{\tau}d_{\alpha}f|^{\tilde{q}}\dmu_{\tau}\right)^{\frac{1}{\tilde{q}}}\leq\frac{1}{\delta} \left(\left(\frac{20}{r_{2}-r_{1}}\right)^{2n}\dashint_{\mathcal{B}_{2}}|D^{\tau}d_{\alpha}f|^{\tilde{q}}\dmu_{\tau}\right)^{\frac{1}{\tilde{q}}}.
\end{aligned}
\end{equation}
Combine \eqref{defn of theta}, \eqref{cond of lambda0}, \eqref{thetau} and \eqref{thetaf} to see that
\begin{equation*}
\begin{aligned}
\Theta\left(x,\frac{r_{2}-r_{1}}{10}\right)\leq \kappa\lambda\quad\text{for any $x\in B_{r_{1}}$}.
\end{aligned}
\end{equation*}
Thus, there is an exit radius $\rho_{x}\in\left(0,\frac{r_{2}-r_{1}}{10}\right]$ such that 
\begin{equation}
\label{bexitradius}
    \Theta(x,\rho_{x})\geq\kappa\lambda\quad\text{and}\quad\Theta(x,\rho)\leq\kappa\lambda\quad\text{ if }\rho_{x}\leq\rho\leq\frac{r_{2}-r_{1}}{10},
\end{equation}
for each $(x,x)\in D_{\kappa\lambda}$.
Using the Vitali covering lemma, we have a family of countable disjoint diagonal balls $\left\{\mathcal{B}_{\rho_{x_{i}}}\left(x_{i}\right)\right\}_{i\in\mathbb{N}}$ satisfying 
\begin{equation}
\label{bcoveringdia}
    \bigcup_{(x,x)\in D_{\kappa\lambda}}\mathcal{B}_{\rho_{x}}\left(x\right)\subset \bigcup_{i\in\mathbb{N}}\mathcal{B}_{5\rho_{x_{i}}}\left(x_{i}\right).
\end{equation}


\noindent
\textbf{Step2. Dyadic cubes and Calder\'on-Zygmund decomposition.} In this step, we find a collection of dyadic cubes contained in $\mathcal{B}_{r_{2}}$ which covers $\mathcal{B}_{r_{1}}$ and apply Calder\'on-Zygmund decomposition to the each dyadic cube in order to find a covering of $E_{\lambda}$. We first find finitely many disjoint dyadic cubes $\{\mathcal{K}
_{j}\}_{j}$ such that
\begin{equation*}
    \mathcal{B}_{r_{1}}\subset \bigcup_{j}\{\mathcal{K}
_{j}\}\subset \mathcal{B}_{r_{2}}\quad\text{and}\quad \mathcal{K}
_{j}\cap \mathcal{B}_{r_{1}}\neq\emptyset.
\end{equation*}
This follows from \cite[Proposition 5-5B]{KMS1}. In particular, the cube is of the form $\mathcal{K}_{j}=\mathcal{Q}_{2^{-k_{0}}}\left(2^{-k_{0}}X_{j}\right)$ for some $X_{j}\in \mathbb{Z}^{n}\times\mathbb{Z}^{n}$ and the positive integer $k_{0}$ such that
\begin{equation}
\label{k0 cond}
\frac{r_{2}-r_{1}}{20}\leq 2^{-k_{0}}20\sqrt{n}<\frac{r_{2}-r_{1}}{10}.
\end{equation} 
We recall $P^{d}\mathcal{K}=K^{d}\times K^{d}$ for a given cube $\mathcal{K}=K^{1}\times K^{2}$.
In this setting, we note from  \eqref{size of b mut}, \eqref{inclusion measure} and \eqref{cond of lambda0} that
\begin{equation}
\label{first red of cube}
\begin{aligned}
   &\left(\dashint_{\mathcal{K}}|D^{\tau}d_{\alpha+s}u|^{\tilde{p}}\dmut\right)^{\frac{1}{\tilde{p}}}+\left[\sum_{d=1}^{2}\left(\frac{1}{\tau}\dashint_{P^{d}\mathcal{K}}|D^{\tau}d_{\alpha+s}u|^{\tilde{p}}\dmut\right)^{\frac{1}{\tilde{p}}}\right]\\
   &\leq \left(\frac{\mu_{\tau}(\mathcal{B}_{2})}{\mu_{\tau}(\mathcal{K})}\dashint_{\mathcal{B}_{2}}|D^{\tau}d_{\alpha+s}u|^{\tilde{p}}\dmut\right)^{\frac{1}{\tilde{p}}}+\left[\sum_{d=1}^{2}\left(\frac{1}{\tau}\frac{\mu_{\tau}(\mathcal{B}_{2})}{\mu_{\tau}(P^{d}\mathcal{K})}\dashint_{\mathcal{B}_{2}}|D^{\tau}d_{\alpha+s}u|^{\tilde{p}}\dmut\right)^{\frac{1}{\tilde{p}}}\right]\\
   &\leq 3\left(2^{n}\frac{\nu_{0}+1}{\tau}\left(\frac{800\sqrt{n}}{r_{2}-r_{1}}\right)^{2n} \dashint_{\mathcal{B}_{2}}|D^{\tau}d_{\alpha+s}u|^{\tilde{p}}\dmut\right)^{\frac{1}{\tilde{p}}}\\
   &\leq\lambda
\end{aligned}
\end{equation}
for any $\mathcal{K}\in \left\{\mathcal{K}_{j}\right\}$ by taking $\kappa\in(0,\kappa_{0}]$, where 
\begin{equation}
\label{kappa0}
\kappa_{0}=\left(\frac{\tau}{(\nu_{0}+1)(80n)^{4n}4^{\tilde{p}}}\right)^{\frac{1}{\tilde{p}}}.
\end{equation}
We next observe a simple fact for a dyadic cube $\mathcal{Q}=Q^{1}\times Q^{2}$ and its predecessor $\tilde{\mathcal{Q}}=\tilde{Q}^{1}\times\tilde{Q}^{2}$. If $\mathrm{dist}\left(\tilde{Q}^{1},\tilde{Q}^{2}\right)\geq d\left(\tilde{\mathcal{Q}}\right)$, where we define $d\big({\mathcal{Q}}\big)=\text{side length of }{\mathcal{Q}}$, then
\begin{equation}
\label{sizetilde}
\mathrm{dist}\left(Q^{1},Q^{2}\right)\geq \mathrm{dist}\left(\tilde{Q}^{1},\tilde{Q}^{2}\right)\geq d\left(\tilde{\mathcal{Q}}\right)> d(\mathcal{Q}).
\end{equation}

In light of \eqref{first red of cube} and \eqref{sizetilde}, we apply the same arguments as in the proof of the classical Calder\'on-Zygmund decomposition lemma to each cube $\mathcal{K}\in \left\{\mathcal{K}_{j}\right\}$, in order to find a family of countable disjoint dyadic cubes $\widetilde{\mathcal{A}}=\left\{\mathcal{Q}^{i}\right\}$ such that  $\mathcal{Q}^{i}\subset \mathcal{K}_{j}$ for some $j$ satisfying
\begin{equation}
\begin{aligned}
\label{czdere}
    &\left(\dashint_{\mathcal{Q}^{i}}|D^{\tau}d_{\alpha+s}u|^{\tilde{p}}\dmut\right)^{\frac{1}{\tilde{p}}}\\
    &\quad+\mathfrak{A}(\mathcal{Q}^{i})\left(\frac{d(\mathcal{Q}^{i})}{\dist(Q^{i,1},Q^{i,2})}\right)^{s+\alpha+\tau}\left[\sum_{d=1}^{2}\left(\frac{1}{\tau}\dashint_{P^{d}\mathcal{Q}^{i}}|D^{\tau}d_{\alpha+s}u|^{\tilde{p}}\dmut\right)^{\frac{1}{\tilde{p}}}\right]>\lambda\quad\text{and}\\
    &\left(\dashint_{\Tilde{\mathcal{Q}}^{i}}|D^{\tau}d_{\alpha+s}u|^{\tilde{p}}\dmut\right)^{\frac{1}{\tilde{p}}}\\
    &\quad+\mathfrak{A}(\tilde{\mathcal{Q}}^{i})\left(\frac{d(\Tilde{\mathcal{Q}^{i}})}{\dist(\Tilde{Q}^{i,1},\Tilde{Q}^{i,2})}\right)^{s+\alpha+\tau}\left[\sum_{d=1}^{2}\left(\frac{1}{\tau}\dashint_{P^{d}\Tilde{\mathcal{Q}}^{i}}|D^{\tau}d_{\alpha+s}u|^{\tilde{p}}\dmut\right)^{\frac{1}{\tilde{p}}}\right]\leq\lambda,
\end{aligned}
\end{equation}
where $\Tilde{\mathcal{Q}}^{i}=\tilde{Q}^{i,1}\times\tilde{Q}^{i,2}$ is the predecessor of $\mathcal{Q}^{i}=Q^{i,1}\times Q^{i,2}$ and
\begin{equation}
\label{cubecoeff}
\begin{aligned}
    \mathfrak{A}\left(\mathcal{Q}^{i}\right)=\begin{cases}
    1&\quad\text{if }\mathrm{dist}\left(Q^{i,1},Q^{i,2}\right)\geq d(\mathcal{Q}^{i}),\\
        0&\quad\text{if }\mathrm{dist}\left(Q^{i,1},Q^{i,2}\right)< d(\mathcal{Q}^{i}).
    \end{cases}
\end{aligned}
\end{equation}
In addition, if $(x,y)\in\mathcal{B}_{r_{1}}\setminus \bigcup\limits_{i}\mathcal{Q}^{i}$, then by using H\"older's inequality and \eqref{czdere}, we find a sequence of dyadic cubes $\mathcal{K}^{j}\notin \widetilde{\mathcal{A}}$ containing $(x,y)$ such that
\begin{equation*}
    \left(\dashint_{\mathcal{K}^{j}}|D^{\tau}d_{\alpha+s}u|\dmut\right)\leq \left(\dashint_{\mathcal{K}^{j}}|D^{\tau}d_{\alpha+s}u|^{\tilde{p}}\dmut\right)^{\frac{1}{\tilde{p}}}\leq \lambda\quad\text{and}\quad \mathrm{diam}\left(\mathcal{K}^{j}\right)\to0\text{ as }j\to\infty.
\end{equation*}
Therefore, the fact that $\mu_{\tau}$ has the doubling property \eqref{doubling} and is absolutely continuous with respect to the Lebesgue measure $\dx\dy$ yields that  
\begin{equation}
\label{coveringforoffdiagonal}
|D^{\tau}d_{\alpha+s}u|\leq\lambda \quad\mbox{a.e. on $\mathcal{B}_{r_{1}}\setminus \bigcup\limits_{i}\mathcal{Q}^{i}$ with respect to the measure $\mu_{\tau}$.}
\end{equation}

\noindent
\textbf{Step 3. First elimination of nearly diagonal cubes.}
We first want to reduce the family of off-diagonal cubes as follows:
For any $\mathcal{Q}=Q^{1}\times Q^{2}\in\widetilde{\mathcal{A}}$,
\begin{equation}
\label{1rm of nondi}
\mathcal{Q}\subset\bigcup\limits_{i}\mathcal{B}_{5\rho_{x_{i}}}\left(x_{i}\right)
\end{equation}
provided that
\begin{equation}
\label{dist less q1q2}
    \dist\left(\tilde{Q}^{1},\tilde{Q}^{2}\right)<d\left(\tilde{\mathcal{Q}}\right).
\end{equation}
To this end, we note from \eqref{czdere} and \eqref{cubecoeff} that one of the following three inequalities must hold:
\begin{equation}
\label{three ineq}
\begin{aligned}
    &\left(\dashint_{\mathcal{Q}}|D^{\tau}d_{\alpha+s}u|^{\tilde{p}}\dmut\right)^{\frac{1}{\tilde{p}}}>\frac{\lambda}{3},\quad\left(\frac{1}{\tau}\dashint_{P^{1}\mathcal{Q}}|D^{\tau}d_{\alpha+s}u|^{\tilde{p}}\dmut\right)^{\frac{1}{\tilde{p}}}>\frac{\lambda}{3},\\
    &\left(\frac{1}{\tau}\dashint_{P^{2}\mathcal{Q}}|D^{\tau}d_{\alpha+s}u|^{\tilde{p}}\dmut\right)^{\frac{1}{\tilde{p}}}>\frac{\lambda}{3}.
\end{aligned}
\end{equation}
Suppose that the first inequality in \eqref{three ineq} holds.
In light of \cite[Proposition 5.3]{KMS1}, we observe that there are points $(x_{0},x_{0})\in \mathbb{R}^{2n}$ and $(x_{1},x_{2})\in\tilde{\mathcal{Q}}=\tilde{Q}^{1}\times \tilde{Q}^{2}$ such that
\begin{equation*}
    \dist\left(\tilde{Q}^{1},\tilde{Q}^{2}\right)=\sqrt{2}\dist(\tilde{\mathcal{Q}},(x_{0},x_{0}))=\sqrt{2}|(x_{1},x_{2})-(x_{0},x_{0})|.
\end{equation*}
Then using \eqref{dist less q1q2},  we find that for any $(a,b)\in\tilde{\mathcal{Q}}$,
\begin{equation*}
\begin{aligned}
    |(a,b)-(x_{0},x_{0})|&\leq |(a,b)-(x_{1},x_{2})|+|(x_{1},x_{2})-(x_{0},x_{0})|\\
    &\leq 2\sqrt{n}d(\tilde{\mathcal{Q}})+\frac{1}{\sqrt{2}}\dist(\tilde{Q}^{1},\tilde{Q}^{2})\\
    &\leq 4\sqrt{n}d(\mathcal{Q})+\frac{2}{\sqrt{2}}d\left(\mathcal{Q}\right)\\
    &\leq 10\sqrt{n}d(\mathcal{Q}),
\end{aligned}
\end{equation*}
which implies that $\mathcal{Q}\subset\tilde{\mathcal{Q}}\subset \mathcal{B}_{10\sqrt{n}d(\mathcal{Q})}(x_{0})$. Since $\mathcal{Q}\cap\mathcal{B}_{r_{1}}\neq\emptyset$, we take $z_{0}\in Q^{1}\cap B_{r_{1}} $. Then we observe that
\begin{equation}
\label{z0r1}
    \mathcal{B}_{10\sqrt{n}d(\mathcal{Q})}(x_{0})\subset \mathcal{B}_{20\sqrt{n}d(\mathcal{Q})}(z_{0}),
\end{equation}
where we have used fact that  $z_{0}\in Q^{1}\subset B_{10\sqrt{n}d(\mathcal{Q})}(x_{0})\subset B_{20\sqrt{n}d(\mathcal{Q})}(z_{0})$.
Note from \eqref{inclusion measure}, the first inequality in \eqref{three ineq} and \eqref{kappa0} that
\begin{equation}
\label{first ineq}
\begin{aligned}
    \dashint_{\mathcal{B}(z_{0},20\sqrt{n}d(\mathcal{Q}))}|D^{\tau}d_{\alpha+s}u|^{\tilde{p}}\dmut&\geq \frac{\mu_{\tau}(\mathcal{Q})}{\mu_{\tau}\Big(\mathcal{B}\big(z_{0},20\sqrt{n}d(\mathcal{Q})\big)\Big)}\dashint_{\mathcal{Q}}|D^{\tau}d_{\alpha+s}u|^{\tilde{p}}\dmut\\
    &\geq \frac{\tau}{2^{n}v_{0}(20\sqrt{n})^{2n}}
    \left(\frac{\lambda}{3}\right)^{\tilde{p}}\\
    &\geq \left(\kappa\lambda\right)^{\tilde{p}}.
\end{aligned}
\end{equation}
Using \eqref{bexitradius}, \eqref{first ineq} and the fact that $z_{0}\in B_{r_{1}}$ and $20\sqrt{n}d(\mathcal{Q})<\frac{r_{2}-r_{1}}{10}$ by \eqref{k0 cond}, we have 
\begin{equation}
\label{lastforfirstremoval}
\mathcal{B}_{20\sqrt{n}d(\mathcal{Q})}(z_{0})\subset \mathcal{B}_{\rho_{z_{0}}}\left(z_{0}\right).
\end{equation}
From \eqref{bcoveringdia}, \eqref{z0r1} and \eqref{lastforfirstremoval}, we have \eqref{1rm of nondi}. We next assume that the second inequality in \eqref{three ineq} holds. We write $\tilde{\mathcal{Q}}=\tilde{\mathcal{Q}}_{d(\tilde{\mathcal{Q}})}\left(x,y\right)$ to see that $P^{1}\tilde{\mathcal{Q}}=\tilde{\mathcal{Q}}_{d(\tilde{\mathcal{Q}})}\left(x\right)$.
For any $(a,b)\in \tilde{\mathcal{Q}}$, we find that 
\begin{equation*}
\begin{aligned}
    |(a,b)-(x,x)&|\leq |(a,b)-(x,y_{2})|+|(x,y_{2})-(x,y_{1})|+|(x,y_{1}),(x,x)|\\
    &\leq \frac{3}{2}\sqrt{n}d(\tilde{\mathcal{Q}})+\mathrm{dist}\left(\tilde{Q}^{1},\tilde{Q}^{2}\right)+\frac{1}{2}\sqrt{n}d\left(\tilde{\mathcal{Q}}\right)\\
    &<10\sqrt{n}d(\mathcal{Q}),
\end{aligned}
\end{equation*}
where $y_{1}\in \tilde{Q^{1}}$ and $y_{2}\in \tilde{Q}^{2}$ such that $|y_{1}-y_{2}|=\mathrm{dist}\left(\tilde{Q}^{1},\tilde{Q}^{2}\right)$. Thus, this yields that $\mathcal{Q}\subset\tilde{\mathcal{Q}}\subset \mathcal{B}_{10\sqrt{n}d(\mathcal{Q})}(x)\subset \mathcal{B}_{20\sqrt{n}d(\mathcal{Q})}(z_{0})$ for some $z_{0}\in B_{r_{1}}$. Like in the estimate of \eqref{first ineq}, we deduce that
\begin{equation*}
\begin{aligned}
\dashint_{\mathcal{B}_{10\sqrt{n}d(\mathcal{Q})}(x)}|D^{\tau}d_{\alpha+s}u|^{\tilde{p}}\dmut&\geq \frac{\mu_{\tau}\left(P^{1}\mathcal{Q}\right)}{\mu_{\tau}\left(\mathcal{B}_{10\sqrt{n}d(\mathcal{Q})}(x)\right)}\dashint_{P^{1}\mathcal{Q}}|D^{\tau}d_{\alpha+s}u|^{\tilde{p}}\dmut \geq\left(\kappa\lambda\right)^{\tilde{p}},
\end{aligned}
\end{equation*}
where we have used \eqref{size of b mut} and the second inequality in \eqref{three ineq} for the third inequality. This implies \eqref{1rm of nondi}.
Similarly, if the third inequality in \eqref{three ineq} is true, then we have \eqref{1rm of nondi}. Therefore, we now restrict our concentration on a subfamily of $\widetilde{\mathcal{A}}$, $\mathcal{A}$ which is defined by
\begin{equation}
\label{tildeadefn}
    \mathcal{A}=\left\{\mathcal{Q}\in \widetilde{\mathcal{A}}\,\;\,\mathcal{Q}\in \widetilde{\mathcal{A}}\setminus \bigcup\limits_{i}\mathcal{B}_{5\rho_{x_{i}}}\left(x\right)\right\}.
\end{equation}
Then we note from \eqref{1rm of nondi}, \eqref{dist less q1q2} and \eqref{sizetilde} that 
\begin{equation}
\label{tildeA}
    \dist\left(Q^{1},Q^{2}\right)\geq\dist\left(\tilde{Q}^{1},\tilde{Q}^{2}\right)\geq d\left(\tilde{\mathcal{Q}}\right)>d(\mathcal{Q})\quad\mbox{for any $\mathcal{Q}=Q^{1}\times Q^{2}\in \mathcal{A}$.}
\end{equation}

\noindent
\textbf{Step 4. Off-diagonal reverse H\"older's inequality.}
We now introduce an important lemma to handle the off-diagonal cubes in $\mathcal{A}$.

\begin{lem}[Reverse H\"older inequality on off-diagonal cubes]
\label{nd lemma}
Let $\mathcal{Q}=Q^{1}\times Q^{2}$ be a cube such that
\begin{equation}
\label{dist cube}
\dist\left(Q^{1},Q^{2}\right)\geq d(\mathcal{Q}).
\end{equation}
Then there is a constant $c=c(n,s,p,\tilde{p},\tilde{\gamma})$ independent of $\tau$ and $\alpha$ such that
\begin{equation*}
\begin{aligned}
\left(\dashint_{\mathcal{Q}}|D^{\tau}d_{\alpha+s}u|^{\tilde{\gamma}}\dmut\right)^{\frac{1}{\tilde{\gamma}}}&\leq c\left(\dashint_{\mathcal{Q}}|D^{\tau}d_{\alpha+s}u|^{p}\dmut\right)^{\frac{1}{p}}\\
&\quad+c\left(\frac{d(\mathcal{Q})}{\dist(Q^{1},Q^{2})}\right)^{s+\alpha+\tau}\left[\sum_{d=1}^{2}\left(\frac{1}{\tau}\dashint_{P^{d}\mathcal{Q}}|D^{\tau}d_{\alpha+s}u|^{\tilde{p}}\dmut\right)^{\frac{1}{\tilde{p}}}\right].
\end{aligned}
\end{equation*}
\end{lem}
\begin{proof}
We note from \eqref{dist cube} that for any $x\in Q^{1}$ and $y\in Q^{2}$,
\begin{equation*}
    |x-y|\leq |x-x_{0}|+|x_{0}-y_{0}|+|y_{0}-y|\leq \sqrt{n}d(\mathcal{Q})+\dist(Q^{1},Q^{2})+\sqrt{n}d(\mathcal{Q})\leq 3\sqrt{n}\dist(Q^{1},Q^{2}),
\end{equation*}
where $x_{0}\in Q^{1}$ and $y_{0}\in Q^{2}$ such that $\dist(Q^{1},Q^{2})=|x_{0}-y_{0}|$.
Therefore, we have that
\begin{equation}
\label{xydist}
\dist\left(Q^{1},Q^{2}\right)\leq |x-y|\leq 3\sqrt{n}\dist\left(Q^{1},Q^{2}\right)
\end{equation}
whenever $x\in Q^{1}$ and $y\in Q^{2}$. The above inequality \eqref{xydist} and Jensen's inequality yield that 
\begin{equation}
\label{measure of qintau}
    \frac{d(\mathcal{Q})^{2n}}{c(n)\dist\left(Q^{1},Q^{2}\right)^{n-p\tau}}\leq\mu_{\tau}(\mathcal{Q})\leq \frac{d(\mathcal{Q})^{2n}}{\dist\left(Q^{1},Q^{2}\right)^{n-p\tau}}
\end{equation}
and 
\begin{equation*}
\begin{aligned}
&\dashint_{\mathcal{Q}}|D^{\tau}d_{\alpha+s}u|^{\tilde{\gamma}}\dmut\\
&\leq \frac{c\dist\left(Q^{1},Q^{2}\right)^{n-p\tau}}{d(\mathcal{Q})^{2n}}\frac{1}{\dist\left(Q^{1},Q^{2}\right)^{n-p\tau+\tilde{\gamma}\left(s+\alpha+\tau\right)}}\int_{Q^{1}}\int_{Q^{2}}|u(x)-u(y)|^{\tilde{\gamma}}\dx\dy\\
&\leq \sum_{i=1}^{2}\frac{c\dist\left(Q^{1},Q^{2}\right)^{-\tilde{\gamma}\left(s+\alpha+\tau\right)}}{d(\mathcal{Q})^{2n}}\int_{Q^{1}}\int_{Q^{2}}|u(x)-(u)_{Q^{i}}|^{\tilde{\gamma}}\dx\dy\\
&\quad+\frac{c\dist\left(Q^{1},Q^{2}\right)^{-\tilde{\gamma}\left(s+\alpha+\tau\right)}}{d(\mathcal{Q})^{2n}}\int_{Q^{1}}\int_{Q^{2}}|(u)_{Q^{1}}-(u)_{Q^{2}}|^{\tilde{\gamma}}\dx\dy\eqqcolon\sum_{i=1}^{2}I_{i}+J
\end{aligned}
\end{equation*}
for some constant $c=c(n,\tilde{\gamma})$. We now further estimate $I_{1}$, $I_{2}$ and $J$.

\noindent
\textbf{Estimates of $I_{1}$ and $I_{2}$.}
From Lemma \ref{SPIL} together with \eqref{condoftildegamma}, there is a constant $c=c(n,s,\tilde{p},\tilde{\gamma})$ such that
\begin{equation}
\label{SPILreverse}
    \left(\dashint_{Q^{1}}|u-(u)_{Q^{1}}|^{\tilde{\gamma}}\dx\right)^\frac{1}{\tilde{\gamma}}\leq cd\left(Q^{1}\right)^{s}\left(\int_{Q^{1}}\dashint_{Q^{1}}\frac{|u(x)-u(y)|^{\tilde{p}}}{|x-y|^{n+s\tilde{p}}}\dx\dy\right)^{\frac{1}{\tilde{p}}}.
\end{equation}
In light of \eqref{SPILreverse}, the fact that 
\begin{equation*}
    1\leq\left(\frac{2d(\mathcal{Q})}{|x-y|}\right)^{\tilde{p}(\alpha+\tau)-p\tau}\quad\text{any }x,y\in Q^{1}
\end{equation*} and \eqref{size of b mut}, we estimate $I_{1}$ as
\begin{equation*}
\begin{aligned}
I_{1}&\leq \frac{c\dist(Q^{1},Q^{2})^{-\tilde{\gamma}\left(s+\alpha+\tau\right)}}{d(\mathcal{Q})^{2n}}d(\mathcal{Q})^{2n+\tilde{\gamma} s}\left(\int_{Q^{1}}\dashint_{Q^{1}}\frac{|u(x)-u(y)|^{\tilde{p}}}{|x-y|^{n+s\tilde{p}}}\dx\dy\right)^{\frac{\tilde{\gamma}}{\tilde{p}}}\\
&\leq \frac{cd(\mathcal{Q})^{\tilde{\gamma}s}}{\dist(Q^{1},Q^{2})^{\tilde{\gamma}\left(s+\alpha+\tau\right)}}\left(\int_{Q^{1}}\dashint_{Q^{1}}\left(\frac{2d(\mathcal{Q})}{|x-y|}\right)^{\tilde{p}(\alpha+\tau)-p\tau}\frac{|u(x)-u(y)|^{\tilde{p}}}{|x-y|^{n+s\tilde{p}}}\dx\dy\right)^{\frac{\tilde{\gamma}}{\tilde{p}}}\\
&\leq c\left(\frac{d(\mathcal{Q})}{\dist(Q^{1},Q^{2})}\right)^{\tilde{\gamma}\left(s+\alpha+\tau\right)}\left(\frac{1}{\tau}\dashint_{Q^{1}\times Q^{1}}|D^{\tau}d_{\alpha+s}u|^{\tilde{p}}\dmut\right)^{\frac{\tilde{\gamma}}{\tilde{p}}}
\end{aligned}
\end{equation*}
for some constant $c=c(n,s,p,\tilde{p},\tilde{\gamma})$. Likewise, we have 
\begin{equation*}
    I_{2}\leq c\left(\frac{d(\mathcal{Q})}{\dist(Q^{1},Q^{2})}\right)^{\tilde{\gamma}\left(s+\alpha+\tau\right)}\left(\frac{1}{\tau}\dashint_{Q^{2}\times Q^{2}}|D^{\tau}d_{\alpha+s}u|^{\tilde{p}}\dmut\right)^{\frac{\tilde{\gamma}}{\tilde{p}}}.
\end{equation*}
\textbf{Estimate of $J$.}
In light of Jensen's inequality, \eqref{xydist} and \eqref{measure of qintau}, we estimate $J$ as
\begin{equation*}
\begin{aligned}
J&\leq c\frac{1}{\dist(Q^{1},Q^{2})^{\tilde{\gamma}\left(s+\alpha+\tau\right)}}\left(\dashint_{Q^{1}}\dashint_{Q^{2}}|u(x)-u(y)|^{p}\dx\dy\right)^{\frac{\tilde{\gamma}}{p}}\\
&\leq c\left(\dashint_{Q^{1}}\dashint_{Q^{2}}\left(\frac{|u(x)-u(y)|}{|x-y|^{s+\alpha+\tau}}\right)^{p}\dx\dy\right)^{\frac{\tilde{\gamma}}{p}}\\
&\leq c\left(\dashint_{\mathcal{Q}}|D^{\tau}d_{\alpha+s}u|^{p}\dmut\right)^{\frac{\tilde{\gamma}}{p}}.
\end{aligned}
\end{equation*}
Combine all the estimates $I_{1}$, $I_{2}$ and $J$ to get the desired result.
\end{proof}

\noindent
\begin{rmk}
We now prove that if $\mathcal{Q}\in\mathcal{A}$, then there is a constant $c_{oh}=c_{oh}(n,s,p,\tilde{p},\tilde{\gamma})$ such that
\begin{equation}
\label{bgamma norm}
\left(\dashint_{\mathcal{Q}}|D^{\tau}d_{\alpha+s}u|^{\tilde{\gamma}}\dmut\right)^{\frac{1}{\tilde{\gamma}}}\leq c_{oh}\lambda\quad\text{for any }\mathcal{Q}\in\mathcal{A}.
\end{equation}
By $p\leq\tilde{p}$, $\tilde{\mathcal{Q}}\subset 4\mathcal{Q}$ and $P^{d}\tilde{\mathcal{Q}}\subset 4P^{d}\mathcal{Q}$, we observe  that
\begin{equation*}
\begin{aligned}
    \left(\dashint_{\mathcal{Q}}|D^{\tau}d_{\alpha+s}u|^{p}\dmut\right)^{\frac{1}{p}}&\leq \left(\dashint_{\mathcal{Q}}|D^{\tau}d_{\alpha+s}u|^{\tilde{p}}\dmut\right)^{\frac{1}{\tilde{p}}}\leq \left(\frac{\mu_{\tau}(4\mathcal{Q})}{\mu_{\tau}(\mathcal{Q})}\dashint_{\Tilde{\mathcal{Q}}}|D^{\tau}d_{\alpha+s}u|^{\tilde{p}}\dmut\right)^{\frac{1}{\tilde{p}}}
\end{aligned}
\end{equation*}
and
\begin{equation*}
\begin{aligned}
&\left(\frac{d(\mathcal{Q})}{\dist(Q^{1},Q^{2})}\right)^{s+\alpha+\tau}\left[\sum_{d=1}^{2}\left(\frac{1}{\tau}\dashint_{P^{d}\mathcal{Q}}|D^{\tau}d_{\alpha+s}u|^{\tilde{p}}\dmut\right)^{\frac{1}{\tilde{p}}}\right]\\
&\leq \left(\frac{d(\tilde{\mathcal{Q}})}{\dist(\tilde{Q}^{1},\tilde{Q}^{2})}\right)^{s+\alpha+\tau}\left[\sum_{d=1}^{2}\left(\frac{\mu_{\tau}(4P^{d}\mathcal{Q})}{\mu_{\tau}(P^{d}\mathcal{Q})}\frac{1}{\tau}\dashint_{P^{d}\tilde{\mathcal{Q}}}|D^{\tau}d_{\alpha+s}u|^{\tilde{p}}\dmut\right)^{\frac{1}{\tilde{p}}}\right].
\end{aligned}
\end{equation*}
Using \eqref{czdere}, \eqref{tildeA}, Lemma \ref{nd lemma}, \eqref{doubling} and the above observations, we have \eqref{bgamma norm}.
\end{rmk}

\noindent
\textbf{Step 5. Decomposition of $\mathcal{A}$.}
We now decompose the family $\mathcal{A}$ into $AD_{\lambda}=\bigcap\limits_{d=1}^{2}AD_{\lambda}^{d}$ and $ND_{\lambda}=\bigcup\limits_{d=1}^{2}ND_{\lambda}^{d}$, where we define
\begin{equation}
\label{almost diagonal cubes2}
    AD_{\lambda}^{d}=\left\{\mathcal{Q}\in \mathcal{A}\,;\,\frac{1}{\tau}\dashint_{P^{d}\mathcal{Q}}|D^{\tau}d_{\alpha+s}u|^{\tilde{p}}\dmut\leq\left(\frac{\lambda}{4}\right)^{\tilde{p}}\right\}
\end{equation}
and
\begin{equation}
\label{non diagonal cubes2}
    ND_{\lambda}^{d}=\left\{\mathcal{Q}\in \mathcal{A}\,;\,\frac{1}{\tau}\dashint_{P^{d}\mathcal{Q}}|D^{\tau}d_{\alpha+s}u|^{\tilde{p}}\dmut>\left(\frac{\lambda}{4}\right)^{\tilde{p}}\right\}.
\end{equation}
We now estimate $\mu_{\tau}(\mathcal{Q})$ for each cubes $\mathcal{Q}\in \mathcal{A}$.

\noindent
\textbf{Step 6. Measure estimate of $\mathcal{Q}\in AD_{\lambda}$.}
We first estimate $\mu_{\tau}(\mathcal{Q})$ for $\mathcal{Q}\in AD_{\lambda}$.
\begin{lem}
There is a constant $c_{ad}=c_{ad}(\tilde{p})$ such that for every $\mathcal{Q}=Q^{1}\times Q^{2}\in AD_{\lambda}$,
\begin{equation}
\label{measqad}
\mu_{\tau}(\mathcal{Q})\leq \frac{c_{ad}}{\lambda^{\tilde{p}}}\int_{\mathcal{Q}\cap\left\{|D^{\tau}d_{\alpha+s}u|>\kappa\lambda\right\}}|D^{\tau}d_{\alpha+s}u|^{\tilde{p}}\dmut.
\end{equation}
\end{lem}
\begin{proof}
In light of \eqref{czdere} and Jensen's inequality, we have that
\begin{equation}
\label{nd app}
\begin{aligned}
    \lambda^{\tilde{p}}
    &\leq 3^{\tilde{p}-1}\left(\dashint_{\mathcal{Q}}|D^{\tau}d_{\alpha+s}u|^{\tilde{p}}\dmut\right)\\
    &\quad+3^{\tilde{p}-1}\left(\frac{d(\mathcal{Q})}{\dist(Q^{1},Q^{2})}\right)^{\tilde{p}\left(s+\alpha+\tau\right)}\left[\sum_{d=1}^{2}\left(\frac{1}{\tau}\dashint_{P^{d}\mathcal{Q}}|D^{\tau}d_{\alpha+s}u|^{\tilde{p}}\dmut\right)\right].
\end{aligned}
\end{equation}
Note that
\begin{equation}
\label{ptermesti}
    \left(\dashint_{\mathcal{Q}}|D^{\tau}d_{\alpha+s}u|^{\tilde{p}}\dmut\right)\leq (\kappa\lambda)^{\tilde{p}}+\frac{1}{\mu_{\tau}(\mathcal{Q})}\int_{\mathcal{Q}\cap\left\{|D^{\tau}d_{\alpha+s}u|>\kappa\lambda\right\}}|D^{\tau}d_{\alpha+s}u|^{\tilde{p}}\dmut.
\end{equation}
Taking into account \eqref{kappa0}, \eqref{tildeA}, \eqref{almost diagonal cubes2},  \eqref{nd app} and \eqref{ptermesti}, we find that
\begin{equation*}
    \frac{\lambda^{\tilde{p}}}{2}\leq \frac{3^{\tilde{p}-1}}{\mu_{\tau}(\mathcal{Q})}\int_{\mathcal{Q}\cap\left\{|D^{\tau}d_{\alpha+s}u|>\kappa\lambda\right\}}|D^{\tau}d_{\alpha+s}u|^{\tilde{p}}\dmut,
\end{equation*}
which implies the desired result \eqref{measqad}.
\end{proof}

\noindent
\textbf{Step 7. Measure estimate of $\mathcal{Q}\in ND_{\lambda}$.}
We next want to estimate $\mu_{\tau}(\mathcal{Q})$ for $\mathcal{Q}\in ND_{\lambda}$. To do this, we introduce families of countable diagonal cubes 
\begin{equation*}
    P^{d}ND_{\lambda}=\left\{P^{d}\mathcal{Q}\,;\, \mathcal{Q}=Q^{1}\times Q^{2}\in ND_{\lambda}\right\} \quad(d=1,2).
\end{equation*}
Then we observe the following two facts (see \cite[Section 5]{KMS1}):
\begin{enumerate}
\item $P^{1}ND_{\lambda}=P^{2}ND_{\lambda}$.
\item There is a subfamily of disjoint cubes $NH\subset P^{1}ND_{\lambda}$ such that
\begin{equation}
\label{hnd}
    \bigcup\limits_{\mathcal{H}\in NH}\mathcal{H}=\bigcup_{\mathcal{Q}\in P^{1}ND_{\lambda}}\mathcal{Q}=\bigcup_{\mathcal{Q}\in P^{2}ND_{\lambda}}\mathcal{Q}.
\end{equation}
\end{enumerate}

We now introduce the following lemma which gives an information of qualitative distance between $Q^{1}$ and $Q^{2}$ for $\mathcal{Q}=Q^{1}\times Q^{2}\in ND_{\lambda}$. 
\begin{lem}
\label{dist of ndl}
Let $\mathcal{Q}=Q^{1}\times Q^{2}\in ND_{\lambda}$. Assume that $Q^{d}\times Q^{d}\subset \mathcal{H}$ for some $d\in\{1,2\}$ and $\mathcal{H}\in NH$. Then we have
\begin{equation*}
    \dist\left(Q^{1},Q^{2}\right)\geq d(\mathcal{H}).
\end{equation*}
\end{lem}
\begin{proof}
We first assume $d=1$. We prove this Lemma by contradiction. Suppose that 
\begin{equation}
\label{cotra}
\dist\left(Q^{1},Q^{2}\right)< d(\mathcal{H}).
\end{equation}
Since $\mathcal{Q}$ and $\mathcal{H}$ are dyadic cubes with $d(\mathcal{Q})\leq \dist\left(Q^{1},Q^{2}\right)<d(\mathcal{H})$, we observe 
\begin{equation}
\label{dqdhdist}
2d(\mathcal{Q})\leq d(\mathcal{H}).
\end{equation}
Let $(x_{c},x_{c})$ be the center of $\mathcal{H}$. Then we find that for $\mathcal{B}\equiv \mathcal{B}_{\frac{d(\mathcal{H})}{2}}\left(x_{c}\right)$,
\begin{equation}
\label{sizeofbh}
    \mathcal{B}\subset\mathcal{H}\subset\sqrt{n}\mathcal{B}.
\end{equation}
For any $(x,y)\in \mathcal{Q}$, we first note from \eqref{cotra} and \eqref{dqdhdist} that
\begin{equation}
\label{xyxcxcdist}
\begin{aligned}
    |(x,y)-(x_{c},x_{c})|&\leq |(x,y)-(x,y')|+|(x,y')-(x_{c},x')|+|(x_{c},x')-(x_{c},x_{c})|\\
    &\leq \sqrt{n}d(\mathcal{Q})+\sqrt{n}\frac{d(\mathcal{H})}{2}+\dist\left(Q^{1},Q^{2}\right)+\sqrt{n}\frac{d(\mathcal{H})}{2}\\
    &\leq \sqrt{n}\left[\frac{d(\mathcal{H})}{2}+\frac{d(\mathcal{H})}{2}+\frac{d(\mathcal{H})}{2}\right]+\dist\left(Q^{1},Q^{2}\right)< 5\sqrt{n}d(\mathcal{H}),
\end{aligned}
\end{equation}
where $(x',y')\in\mathcal{Q}$ such that $\dist\left(Q^{1},Q^{2}\right)=|x'-y'|$.
Thus, \eqref{sizeofbh} and \eqref{xyxcxcdist} imply that $\mathcal{Q},\mathcal{H}\subset \mathcal{B}_{5\sqrt{n}d(\mathcal{H})}\left(x_{c}\right)$. In addition, by following the same arguments as in \eqref{z0r1} and using the fact that $\mathcal{B}_{5\sqrt{n}d(\mathcal{H})}\left(x_{c}\right)\subset \mathcal{B}_{10\sqrt{n}d(\mathcal{H})}\left(z_{0}\right)$, there is a point $z_{0}\in B_{r_{1}}\cap Q^{1}$ such that
\begin{equation}
\label{hqrelationb}
    \mathcal{Q}\subset \mathcal{H}\subset \mathcal{B}_{10\sqrt{n}d(\mathcal{H})}\left(z_{0}\right).
\end{equation}
Then using \eqref{non diagonal cubes2}, \eqref{hnd}, \eqref{hqrelationb} and \eqref{size of b mut}, we observe that
\begin{equation*}
\begin{aligned}
    \left(\frac{\lambda}{4}\right)^{\tilde{p}}&< \frac{1}{\tau}\dashint_{\mathcal{H}}|D^{\tau}d_{\alpha+s}u|^{\tilde{p}}\dmut\\
    &\leq \frac{1}{\tau}\frac{\mu_{\tau}\left(\mathcal{B}_{10\sqrt{n}d(\mathcal{H})}\left(z_{0}\right)\right)}{\mu_{\tau}\left(\mathcal{B}_{\frac{d(\mathcal{H})}{2}}\left(x_{c}\right)\right)}\dashint_{\mathcal{B}_{10\sqrt{n}d(\mathcal{H})}\left(z_{0}\right)}|D^{\tau}d_{\alpha+s}u|^{\tilde{p}}\dmut\\
    &=(20\sqrt{n})^{n+p\tau}\frac{1}{\tau}\dashint_{\mathcal{B}_{10\sqrt{n}d(\mathcal{H})}\left(z_{0}\right)}|D^{\tau}d_{\alpha+s}u|^{\tilde{p}}\dmut,
\end{aligned}
\end{equation*}
which implies that
\begin{equation}
\label{contr h}
    (\kappa\lambda)^{\tilde{p}}<\frac{\tau}{(20n)^{2n}4^{\tilde{p}}}\lambda^{\tilde{p}}<\dashint_{\mathcal{B}_{10\sqrt{n}d(\mathcal{H})}\left(z_{0}\right)}|D^{\tau}d_{\alpha+s}u|^{\tilde{p}}\dmut.
\end{equation}
Since \eqref{k0 cond} gives that $10\sqrt{n}d(\mathcal{H})<\frac{r_{2}-r_{1}}{10}$,  we find that $\mathcal{B}_{10\sqrt{n}d(\mathcal{H})}\left(z_{0}\right)\subset \bigcup\limits_{i}\mathcal{B}_{5\rho_{x_{i}}}(x_{i})$, where we have used \eqref{bexitradius}, \eqref{contr h} and the fact that $z_{0}\in B_{r_{1}}$. Thus, this contradicts the definition of the family $\mathcal{A}$ given in \eqref{tildeadefn}. 
We similarly prove the case for $d=2$.
\end{proof}
With the above lemma, we now prove the following measure estimates for the cubes in $ND_{\lambda}$.
\begin{lem}
There is a constant $c_{nd}=c_{nd}(n,s,p,\tilde{p})$ such that
\begin{equation}
\label{sumofmeasqnd}
    \sum_{\mathcal{Q}\in ND_{\lambda}}\mu_{\tau}(\mathcal{Q})\leq\frac{c_{nd}}{\lambda^{\tilde{p}}}\int_{\mathcal{B}_{r_{2}}\cap\left\{\,|D^{\tau}d_{\alpha+s}u|>\kappa\lambda\right\}}|D^{\tau}d_{\alpha+s}u|^{\tilde{p}}\dmut.
\end{equation} 
\end{lem}
\begin{proof}
Let $\mathcal{Q}\in ND_{\lambda}$.
We then note from \eqref{tildeA} that
\begin{equation}
\label{simple cal}
\begin{aligned}
    &\left(\frac{d(\mathcal{Q})}{\dist(Q^{1},Q^{2})}\right)^{\tilde{p}\left(s+\alpha+\tau\right)}\left[\left(\frac{1}{\tau}\dashint_{P^{d}\mathcal{Q}}|D^{\tau}d_{\alpha+s}u|^{\tilde{p}}\dmut\right)\right]\\
    &\leq \frac{(\kappa\lambda)^{\tilde{p}}}{\tau}+\left(\frac{d(\mathcal{Q})}{\dist(Q^{1},Q^{2})}\right)^{\tilde{p}\left(s+\alpha+\tau\right)}\left[\frac{1}{\tau}\left(\frac{1}{\mu_{\tau}(P^{d}\mathcal{Q})}\int_{P^{d}\mathcal{Q}\cap\left\{|D^{\tau}d_{\alpha+s}u|>\kappa\lambda\right\}}|D^{\tau}d_{\alpha+s}u|^{\tilde{p}}\dmut\right)\right]
\end{aligned}
\end{equation}
for $d=1$ and 2.
In light of \eqref{czdere}, \eqref{kappa0}, Jensen's inequality, \eqref{ptermesti} and  \eqref{simple cal}, we have that
\begin{equation}
\label{adlndlcalforme}
\begin{aligned}
    &\frac{\lambda^{\tilde{p}}}{2}\mu_{\tau}(\mathcal{Q})\\
    &\leq 3^{\tilde{p}-1}\int_{\mathcal{Q}\cap\left\{|D^{\tau}d_{\alpha+s}u|>\kappa\lambda\right\}}|D^{\tau}d_{\alpha+s}u|^{\tilde{p}}\dmut\\
    &\quad+3^{\tilde{p}-1}\left(\frac{d(\mathcal{Q})}{\dist(Q^{1},Q^{2})}\right)^{\tilde{p}\left(s+\alpha+\tau\right)}\left[\sum_{d=1}^{2}\left(\frac{1}{\tau}\frac{\mu_{\tau}(\mathcal{Q})}{\mu_{\tau}(P^{d}\mathcal{Q})}\int_{P^{d}\mathcal{Q}\cap\left\{|D^{\tau}d_{\alpha+s}u|>\kappa\lambda\right\}}|D^{\tau}d_{\alpha+s}u|^{\tilde{p}}\dmut\right)\right].
\end{aligned}
\end{equation}
We now denote  
\begin{equation}
\label{Kd12}
\begin{aligned}
    K^{d}=\sum_{\mathcal{Q}\in ND_{\lambda}^{d}}\left(\frac{d(\mathcal{Q})}{\dist(Q^{1},Q^{2})}\right)^{\tilde{p}\left(s+\alpha+\tau\right)}\left[\left(\frac{1}{\tau}\frac{\mu_{\tau}(\mathcal{Q})}{\mu_{\tau}(P^{d}\mathcal{Q})}\int_{P^{d}\mathcal{Q}\cap\left\{|D^{\tau}d_{\alpha+s}u|>\kappa\lambda\right\}}|D^{\tau}d_{\alpha+s}u|^{\tilde{p}}\dmut\right)\right]    
\end{aligned}
\end{equation}
for $d=1$ and $2$.
To estimate $K^{d}$ for each $d=1$ and 2, we investigate families of countable disjoint cubes  
\begin{equation}
\label{lengthqh}
    ND^{d}_{\lambda}(\mathcal{H})=\left\{\mathcal{Q}\in ND_{\lambda}^{d}\,;\, P^{d}\mathcal{Q}\subset\mathcal{H}\right\}\quad(d=1,2)
\end{equation}
for $\mathcal{H}\in NH$. We note from \eqref{hnd} that 
\begin{equation*}
   ND^{d}_{\lambda}=\bigcup\limits_{\mathcal{H}\in NH}ND_{\lambda}^{d}(\mathcal{H}),
\end{equation*}
where $\left\{ND_{\lambda}^{d}(\mathcal{H})\right\}_{\mathcal{H}\in NH}$ are mutually disjoint classes.
Using \eqref{lengthqh}, we first find countable mutually disjoint families of cubes $\left[ND^{d}_{\lambda}(\mathcal{H})\right]_{i}$ such that
\begin{equation*}
    \left[ND^{d}_{\lambda}(\mathcal{H})\right]_{i}=\left\{\mathcal{Q}\in ND^{d}_{\lambda}(\mathcal{H})\,;\, d(\mathcal{Q})=2^{-i}d(\mathcal{H})\right\} \quad (i\geq0)
\end{equation*}
and
\begin{equation*}
    ND^{d}_{\lambda}(\mathcal{H})=\bigcup\limits_{i=0}^{\infty}\left[ND^{d}_{\lambda}(\mathcal{H})\right]_{i}.
\end{equation*}
In light of Lemma \ref{dist of ndl}, we next decompose 
$\left[ND^{d}_{\lambda}(\mathcal{H})\right]_{i}$ into subfamilies $\left[ND^{d}_{\lambda}(\mathcal{H})\right]_{i,j}$ which are defined by 
\begin{equation*}
    \left[ND^{d}_{\lambda}(\mathcal{H})\right]_{i,j}=\left\{\mathcal{Q}\in \left[ND^{d}_{\lambda}(\mathcal{H})\right]_{i}\,;\, 2^{j}d(\mathcal{H})\leq\dist(Q^{1},Q^{2})< 2^{j+1}d(\mathcal{H})\text{ for }\mathcal{Q}=Q^{1}\times Q^{2}\right\}
\end{equation*}
for $j\geq0$. To further decompose  $\left[ND^{d}_{\lambda}(\mathcal{H})\right]_{i,j}$ into the smaller families, we denote all diagonal dyadic cubes in $\mathcal{H}$ with the side-length $2^{-i}d(\mathcal{H})$ by $\mathcal{H}_{i}(k)=H_{i}(k)\times H_{i}(k)$ for $i\geq0$ and $k=1,2,\ldots,2^{in}$.
From the above notation, we set
\begin{equation*}
    \left[ND^{d}_{\lambda}(\mathcal{H})\right]_{i,j,k}=\left\{\mathcal{Q}\in \left[ND^{d}_{\lambda}(\mathcal{H})\right]_{i,j}\,;\, P^{d}\mathcal{Q}=\mathcal{H}_{i}(k) \right\}.
\end{equation*}
Then $\left\{\left[ND^{d}_{\lambda}(\mathcal{H})\right]_{i,j,k}\right\}_{i,j,k\geq0}$ is a class of the mutually disjoint families such that
\begin{equation*}
    ND^{d}_{\lambda}(\mathcal{H})=\bigcup\limits_{i,j,k\geq0}\left[ND^{d}_{\lambda}(\mathcal{H})\right]_{i,j,k}.
\end{equation*}
We first note that for any $\mathcal{Q}\in \left[ND^{1}_{\lambda}(\mathcal{H})\right]_{i,j,k}$, $\mathcal{Q}$ is of the form $H_{i}(k)\times Q^{2}$ for some dyadic cube $Q^{2}$ satisfying 
\begin{equation}
\label{ndlambdaijk}
d(Q^{2})=2^{-i}d(\mathcal{H})\quad\text{and}\quad 2^{j}d(\mathcal{H})\leq\dist(H_{i}(k),Q^{2})< 2^{j+1}d(\mathcal{H}).
\end{equation}
Thus, we find that there exists a constant $c$ depending only on $n$ such that
\begin{equation}
\label{numberofcube}
\left|\left[ND^{d}_{\lambda}(\mathcal{H})\right]_{i,j,k}\right|\leq c2^{n(i+j)},
\end{equation}
where $\left|\left[ND^{d}_{\lambda}(\mathcal{H})\right]_{i,j,k}\right|$ denotes the number of the cubes in $\left[ND^{d}_{\lambda}(\mathcal{H})\right]_{i,j,k}$.
From the decomposition of $ND_{\lambda}^{d}(\mathcal{H})$, we now estimate $K^{d}$ for each $d=1$ and $2$.
We first note from \eqref{size of b mut} and \eqref{measure of qintau} that there is a constant $c=c(n)$ such that
\begin{equation}
\label{qqh}
    \frac{\mu_{\tau}(\mathcal{Q})}{\mu_{\tau}(P^{d}\mathcal{Q})}\leq c\frac{\tau}{\nu_{0}} \left(\frac{d(\mathcal{Q})}{\dist(Q^{1},Q^{2})}\right)^{n-p\tau}.
\end{equation}
According to \eqref{qqh}, the fact that $p\tau\leq \tilde{p}\tau$ and \eqref{ndlambdaijk}, we find
\begin{equation*}
\begin{aligned}
K^{d}&\leq \frac{c}{\nu_{0}}\sum_{\mathcal{H}\in H}\sum_{i,j,k}\sum_{\mathcal{Q}\in\left[ND^{d}_{\lambda}(\mathcal{H})\right]_{i,j,k}} \left(\frac{d(\mathcal{Q})}{\dist(Q^{1},Q^{2})}\right)^{n-p\tau+\tilde{p}(s+\alpha)+\tilde{p}\tau}\\
    &\qquad\qquad\qquad\qquad\qquad\qquad\times\left(\int_{P^{d}\mathcal{Q}\cap\left\{|D^{\tau}d_{\alpha+s}u|>\kappa\lambda\right\}}|D^{\tau}d_{\alpha+s}u|^{\tilde{p}}\dmut\right)\\
    &\leq \frac{c}{\nu_{0}}\sum_{\mathcal{H}\in H}\sum_{i,j\geq0}\sum_{k=1}^{2^{in}}\sum_{\mathcal{Q}\in\left[ND^{d}_{\lambda}(\mathcal{H})\right]_{i,j,k}}\left(2^{-(i+j)}\right)^{n+\tilde{p}(s+\alpha)}\\
    &\qquad\qquad\qquad\qquad\qquad\qquad\qquad\times\left(\int_{\mathcal{H}_{i}(k)\cap\left\{|D^{\tau}d_{\alpha+s}u|>\kappa\lambda\right\}}|D^{\tau}d_{\alpha+s}u|^{\tilde{p}}\dmut\right).
\end{aligned}
\end{equation*}
Using \eqref{numberofcube} and the fact that $\mathcal{H}_{i}(k)$ is disjoint for each $i$ and $k$, we further estimate $K^{d}$ as
\begin{equation*}
\begin{aligned}
    K^{d}&\leq \frac{1}{\nu_{0}}\sum_{\mathcal{H}\in H}\sum_{i,j\geq0}\sum_{k=1}^{2^{in}} c\left(2^{-(i+j)}\right)^{\tilde{p}(s+\alpha)}\left(\int_{\mathcal{H}_{i}(k)\cap\left\{|D^{\tau}d_{\alpha+s}u|>\kappa\lambda\right\}}|D^{\tau}d_{\alpha+s}u|^{\tilde{p}}\dmut\right)\\
    &\leq \frac{1}{\nu_{0}}\sum_{\mathcal{H}\in H}\sum_{i,j\geq0} c\left(2^{-(i+j)}\right)^{sp}\left(\int_{\mathcal{H}\cap\left\{|D^{\tau}d_{\alpha+s}u|>\kappa\lambda\right\}}|D^{\tau}d_{\alpha+s}u|^{\tilde{p}}\dmut\right)\\
    &\leq c\sum_{\mathcal{H}\in H}\left(\int_{\mathcal{H}\cap\left\{|D^{\tau}d_{\alpha+s}u|>\kappa\lambda\right\}}|D^{\tau}d_{\alpha+s}u|^{\tilde{p}}\dmut\right)
\end{aligned}
\end{equation*}
for some constant $c=c(n,s,p)$,
For the last inequality, we have used the fact that
\begin{equation}
\label{sumof2minus}
    \sum_{i,j\geq0} \left(2^{-(i+j)}\right)^{a}\leq \frac{2^{a}}{a\mathrm{ln}2}\quad(a>0).
\end{equation}
 Since $NH$ is the family of disjoint dyadic cubes and $\mathcal{H}\subset \mathcal{B}_{r_{2}}$, the above estimate yields that
\begin{equation}
\label{kh}
    K^{d}\leq c\int_{\mathcal{B}_{r_{2}}\cap\left\{|D^{\tau}d_{\alpha+s}u|>\kappa\lambda\right\}}|D^{\tau}d_{\alpha+s}u|^{\tilde{p}}\dmut.
\end{equation}
Then we note from \eqref{almost diagonal cubes2}, \eqref{nd app}, \eqref{adlndlcalforme}, \eqref{Kd12} and \eqref{kh} that
\begin{equation*}
\begin{aligned}
    \sum_{\mathcal{Q}\in ND_{\lambda}^{1}\cap AD_{\lambda}^{2}}\frac{\lambda^{\tilde{p}}}{4}\mu_{\tau}(\mathcal{Q})&\leq 3^{\tilde{p}-1}\sum_{\mathcal{Q}\in ND_{\lambda}^{1}\cap AD_{\lambda}^{2}}\int_{\mathcal{Q}\cap\left\{|D^{\tau}d_{\alpha+s}u|>\kappa\lambda\right\}}|D^{\tau}d_{\alpha+s}u|^{\tilde{p}}\dmut+3^{\hat{p}-1}K^{1}\\
    &\leq c\int_{\mathcal{B}_{r_{2}}\cap\left\{|D^{\tau}d_{\alpha+s}u|>\kappa\lambda\right\}}|D^{\tau}d_{\alpha+s}u|^{\tilde{p}}\dmut,
\end{aligned}
\end{equation*}
for some constant $c=c(n,s,p,\tilde{p})$. Likewise, we deduce that there is a constant $c=c(n,s,p,\tilde{p})$ such that
\begin{equation*}
    \sum_{\mathcal{Q}\in ND_{\lambda}^{2}\cap AD_{\lambda}^{1}}\frac{\lambda^{\tilde{p}}}{4}\mu_{\tau}(\mathcal{Q})\leq c\int_{\mathcal{B}_{r_{2}}\cap\left\{|D^{\tau}d_{\alpha+s}u|>\kappa\lambda\right\}}|D^{\tau}d_{\alpha+s}u|^{\tilde{p}}\dmut
\end{equation*}
and
\begin{equation*}
    \sum_{\mathcal{Q}\in ND_{\lambda}^{1}\cap ND_{\lambda}^{2}}\frac{\lambda^{\tilde{p}}}{4}\mu_{\tau}(\mathcal{Q})\leq c\int_{\mathcal{B}_{r_{2}}\cap\left\{|D^{\tau}d_{\alpha+s}u|>\kappa\lambda\right\}}|D^{\tau}d_{\alpha+s}u|^{\tilde{p}}\dmut.
\end{equation*}
By combining the above three estimates, we obtain the desired estimate \eqref{sumofmeasqnd}. 
\end{proof}
\noindent
\textbf{Step 8. The choice of constants.} We now take 
\begin{equation*}
\kappa=\frac{\tau^{\frac{1}{p}}}{(\nu_{0}+1)(80n)^{8n}}.
\end{equation*}
Using the fact that $\tau^{\frac{1}{p}}\leq n\tau^{\frac{1}{\tilde{p}}}$ by \eqref{firstcondoftau}, we have $\kappa\leq \kappa_{0}$ which is given in \eqref{kappa0}.
We next choose 
\begin{equation*}
    c_{d}=\frac{\tau^{\frac{1}{p}}}{\kappa},\quad c_{u}=\frac{\tau^{\frac{1}{p}}}{c_{d}}\quad\mbox{and}\quad c_{od}=c_{ad}+c_{nd},
\end{equation*}
where $c_{d}=c_{d}(n,s,p)$ and $c_{od}=c_{od}(n,s,p,\tilde{p})$.

\noindent
\textbf{Step 9. Conclusion of the proof.}
Let $\lambda\geq \lambda_{0}$ given in \eqref{lambda0}. Then $\lambda$ satisfies \eqref{cond of lambda0}. Therefore, in light of \eqref{bcoveringdia}, \eqref{coveringforoffdiagonal} and \eqref{tildeadefn}, we find \eqref{coveringforelambda}. In addition, taking into account
\eqref{bgamma norm}, \eqref{measqad} and  \eqref{sumofmeasqnd}, we prove \eqref{exitradius}, \eqref{gamma norm} and \eqref{sumofmeasq}. This completes the proof of Lemma \ref{coveringoflevelset}.
\end{proof}
\noindent
We finish this section by providing  estimates which will be needed for the above comparison lemma.
\begin{rmk}
For each diagonal ball $\mathcal{B}_{\rho_{x_{i}}}(x_{i})$ obtained by Lemma \ref{coveringoflevelset}, we want to show that if 
\begin{equation}
\label{alphatausum}
\alpha+\tau<\frac{s}{p-1},
\end{equation}
then 
\begin{equation}
\label{diagonal lambda estimate}
\begin{aligned}
    &\left(\frac{1}{\tau}\dashint_{\mathcal{B}_{20\rho_{x_{i}}}(x_{i})}|D^{\tau}d_{\alpha+s}u|^{p}\dmu_{\tau}\right)^{\frac{1}{p}}+\Tailp\left(\frac{u-(u)_{B_{20\rho_{x_{i}}}}}{(20\rho_{x_{i}})^{\alpha+\tau+s}};B_{20\rho_{x_{i}}}(x_{i})\right)\leq \frac{c}{\frac{s}{p-1}-(\alpha+\tau)}\lambda,\\
    &\left(\frac{1}{\tau}\dashint_{\mathcal{B}_{20\rho_{x_{i}}}(x_{i})}|D^{\tau}d_{\alpha}f|^{\tilde{q}}\dmu_{\tau}\right)^{\frac{1}{\tilde{q}}}+\Tails\left(\frac{f-(f)_{B_{20\rho_{x_{i}}}}}{(20\rho_{x_{i}})^{\alpha+\tau}};B_{20\rho_{x_{i}}}(x_{i})\right)\leq \frac{c}{\frac{s}{p-1}-(\alpha+\tau)}\delta\lambda,
\end{aligned}
\end{equation}
where $c=c(n,s,p)$. To this end, we divide the range of $20\rho_{x_{i}}$ into 
\begin{equation}
\label{rhoxifirst}
\frac{r_{2}-r_{1}}{20}< 20\rho_{x_{i}}\leq 2(r_{2}-r_{1})
\end{equation}
and
\begin{equation}
\label{rhoxisecond}
0<20\rho_{x_{i}}\leq\frac{r_{2}-r_{1}}{20}.
\end{equation}

\noindent
(1). Assume \eqref{rhoxifirst} is true. Using \eqref{size of b mut}, H\"older's inequality and \eqref{lambda0}, we get that
\begin{equation*}
\begin{aligned}
    \left(\frac{1}{\tau}\dashint_{\mathcal{B}_{20\rho_{x_{i}}}(x_{i})}|D^{\tau}d_{\alpha+s}u|^{p}\dmu_{\tau}\right)^{\frac{1}{p}}&\leq \left(\frac{1}{\tau}\frac{\mu_{\tau}\left(\mathcal{B}_{2}\right)}{\mu_{\tau}\left(\mathcal{B}_{20\rho_{x_{i}}}(x_{i})\right)}\dashint_{\mathcal{B}_{2}}|D^{\tau}d_{\alpha+s}u|^{p}\dmu_{\tau}\right)^{\frac{1}{p}}\\
    &\leq \frac{1}{\tau^{\frac{1}{p}}}\left(\frac{40}{r_{2}-r_{1}}\right)^{2n}\left(\dashint_{\mathcal{B}_{2}}|D^{\tau}d_{\alpha+s}u|^{\tilde{p}}\dmu_{\tau}\right)^{\frac{1}{\tilde{p}}}\\
    &\leq c\lambda
\end{aligned}
\end{equation*}
for some constant $c=c(n)$. Similarly, we find
\begin{equation*}
\left(\frac{1}{\tau}\dashint_{\mathcal{B}_{20\rho_{x_{i}}}(x_{i})}|D^{\tau}d_{\alpha}f|^{\tilde{q}}\dmu_{\tau}\right)^{\frac{1}{\tilde{q}}}\leq \frac{1}{\tau^{\frac{1}{p}}}\left(\frac{\mu_{\tau}\left(\mathcal{B}_{2}\right)}{\mu_{\tau}\left(\mathcal{B}_{20\rho_{x_{i}}}(x_{i})\right)}\dashint_{\mathcal{B}_{2}}|D^{\tau}d_{\alpha}f|^{\tilde{q}}\dmu_{\tau}\right)^{\frac{1}{\tilde{q}}}\leq c\delta\lambda,
\end{equation*}
where we have used the fact that $\tau^{\frac{1}{p}}\leq n\tau^{\frac{1}{\tilde{q}}}$ by \eqref{firstcondoftau}. On the other hand, in light of \eqref{tildeptildeqalphacond} and the fact that
\begin{equation*}
    \frac{1}{(20\rho_{x_{i}})^{\alpha+\tau+s}}\leq \left(\frac{20}{r_{2}-r_{1}}\right)^{\alpha+\tau+s}\leq \left(\frac{20}{r_{2}-r_{1}}\right)^{\frac{sp}{p-1}},
\end{equation*} we estimate the tail term of $u$ as 
\begin{equation*}
\begin{aligned}
    \Tailp\left(\frac{u-(u)_{B_{20\rho_{x_{i}}}(x_{i})}}{(20\rho_{x_{i}})^{\alpha+\tau+s}};B_{20\rho_{x_{i}}}(x_{i})\right)&\leq\left(\frac{20}{r_{2}-r_{1}}\right)^{\frac{sp}{p-1}}\Tailp\left(u-(u)_{B_{2}};B_{20\rho_{x_{i}}}(x_{i})\right)\\
    &\quad+\left(\frac{20}{r_{2}-r_{1}}\right)^{\frac{sp}{p-1}}\Tailp\left((u)_{B_{2}}-(u)_{B_{20\rho_{x_{i}}}(x_{i})};B_{20\rho_{x_{i}}}(x_{i})\right)\\
    &\eqqcolon \left(\frac{20}{r_{2}-r_{1}}\right)^{\frac{sp}{p-1}}T_{1}+\left(\frac{20}{r_{2}-r_{1}}\right)^{\frac{sp}{p-1}}T_{2}.
\end{aligned}
\end{equation*}
By the fact that $(|a|+|b|)^{\frac{1}{p-1}}\leq 2|a|^{\frac{1}{p-1}}+2|b|^{\frac{1}{p-1}}$ and $|y-x_{i}|\geq\frac{r_{2}-r_{1}}{2}|y|$ for any $y\in\mathbb{R}^{n}\setminus B_{2}$, we further estimate $T_{1}$ as
\begin{equation*}
\begin{aligned}
    T_{1}&\leq 2\left(\left(20\rho_{x_{i}}\right)^{sp}\int_{\mathbb{R}^{n}\setminus B_{2}}\frac{|u(y)-(u)_{B_{2}}|^{p-1}}{|y-x_{i}|^{n+sp}}\dy\right)^{\frac{1}{p-1}}\\
    &\quad+2\left(\left(20\rho_{x_{i}}\right)^{sp}\int_{B_{2}\setminus B_{20\rho_{x_{i}}}(x_{i})}\frac{|u(y)-(u)_{B_{2}}|^{p-1}}{|y-x_{i}|^{n+sp}}\dy\right)^{\frac{1}{p-1}}\\
    &\leq 2\left(2^{n+2sp}\left(\frac{1}{r_{2}-r_{1}}\right)^{n}\int_{\mathbb{R}^{n}\setminus B_{2}}\frac{|u(y)-(u)_{B_{2}}|^{p-1}}{|y|^{n+sp}}\dy\right)^{\frac{1}{p-1}}\\
    &\quad+2\left(\left(\frac{40}{r_{2}-r_{1}}\right)^{n}|B_{1}|\dashint_{B_{2}}|u(y)-(u)_{B_{2}}|^{p-1}\dy\right)^{\frac{1}{p-1}}.
\end{aligned}
\end{equation*}
Applying H\"older's inequality and \eqref{tkestimate} to the second term in the above last inequality, we observe that there is a constant $c=c(n,s,p)$ such that
\begin{equation*}
\begin{aligned}
    T_{1}&\leq \frac{c}{(r_{2}-r_{1})^{\frac{n}{p-1}}}\Tailp\left(\frac{u-(u)_{B_{2}}}{2^{\alpha+\tau+s}};B_{2}\right)+\frac{c}{(r_{2}-r_{1})^{\frac{n}{p-1}}}\left(\frac{1}{\tau}\dashint_{\mathcal{B}_{2}}|D^{\tau}d_{\alpha+s}u|^{p}\dmut\right)^{\frac{1}{p}}\\
    &\leq \frac{c}{(r_{2}-r_{1})^{\frac{n}{p-1}}}\Tailp\left(\frac{u-(u)_{B_{2}}}{2^{\alpha+\tau+s}};B_{2}\right)+\frac{c}{(r_{2}-r_{1})^{\frac{n}{p-1}}}\left(\frac{1}{\tau}\dashint_{\mathcal{B}_{2}}|D^{\tau}d_{\alpha+s}u|^{\tilde{p}}\dmut\right)^{\frac{1}{\tilde{p}}}.
\end{aligned}
\end{equation*}
For the last inequality, we have used H\"older's inequality. We next estimate $T_{2}$ as
\begin{equation*}
\begin{aligned}
    T_{2}&\leq \frac{c}{(r_{2}-r_{1})^{\frac{n}{p-1}}}\left(\dashint_{B_{2}}|u(y)-(u)_{B_{2}}|^{p-1}\dy\right)^{\frac{1}{p-1}}\leq \frac{c}{(r_{2}-r_{1})^{\frac{n}{p-1}}}\left(\frac{1}{\tau}\dashint_{\mathcal{B}_{2}}|D^{\tau}d_{\alpha+s}u|^{\tilde{p}}\dmut\right)^{\frac{1}{\tilde{p}}},
\end{aligned}
\end{equation*}
where we have used \eqref{scalingpropertytail},  \eqref{meansize}, H\"older's inequality and \eqref{tkestimate}. Combine all the estimates $T_{1}$ and $T_{2}$ together with \eqref{lambda0} and the fact that $\frac{sp}{p-1}+\frac{n}{p-1}\leq 2n$ to see that
\begin{equation}
\label{firsttailu}
\begin{aligned}
   \left(\frac{1}{\tau}\dashint_{\mathcal{B}_{20\rho_{x_{i}}}(x_{i})}|D^{\tau}d_{\alpha+s}u|^{p}\dmu_{\tau}\right)^{\frac{1}{p}}+ \Tailp\left(\frac{u-(u)_{B_{20\rho_{x_{i}}}(x_{i})}}{(20\rho_{x_{i}})^{\alpha+\tau+s}};B_{20\rho_{x_{i}}}(x_{i})\right)\leq c\lambda
\end{aligned}
\end{equation}
for some constant $c=c(n,s,p)$. Likewise, by following the same lines for the proof of \eqref{firsttailu} with $u(x)$, $s$, $\tilde{p}$, $sp$ and $\lambda$ there, replaced by $f(x)$, $0$, $\tilde{q}$, $s$ and $\delta\lambda$, respectively,  we have
\begin{equation}
\label{firsttailf}
\begin{aligned}
    \left(\frac{1}{\tau}\dashint_{\mathcal{B}_{20\rho_{x_{i}}}(x_{i})}|D^{\tau}d_{\alpha}f|^{\tilde{q}}\dmu_{\tau}\right)^{\frac{1}{\tilde{q}}}+\Tails\left(\frac{f-(f)_{B_{20\rho_{x_{i}}}(x_{i})}}{(20\rho_{x_{i}})^{\alpha+\tau}};B_{20\rho_{x_{i}}}(x_{i})\right)\leq c\delta\lambda.
\end{aligned}
\end{equation}
Since $1\leq \frac{1}{\frac{s}{p-1}-(\alpha+\tau)}$ which follows by \eqref{alphatausum}, \eqref{firsttailu} and \eqref{firsttailf} imply \eqref{diagonal lambda estimate}.

\noindent
(2). We now assume \eqref{rhoxisecond}. Then there is a positive integer $k\in\mathbb{N}$ such that 
\begin{equation}
\label{icond}
    \frac{r_{2}-r_{1}}{20}< 2^{k}20\rho_{x_{i}}\leq \frac{r_{2}-r_{1}}{10}.
\end{equation}
Applying Lemma \ref{tail estimate for local notioin} with $\Omega=B_{2}$, $t=s$, $\rho=20\rho_{x_{i}}$, $x_{0}=x_{i}$ and $i=k$, we have that
\begin{equation*}
\begin{aligned}
    \Tailp\left(\frac{u-(u)_{B_{20\rho_{x_{i}}}(x_{i})}}{(20\rho_{x_{i}})^{\alpha+\tau+s}};B_{20\rho_{x_{i}}}(x_{i})\right)&\leq c2^{-k\frac{sp}{p-1}}\Tailp\left(\frac{u-(u)_{B_{2^{k}20\rho_{x_{i}}}(x_{i})}}{(20\rho_{x_{i}})^{\alpha+\tau+s}};B_{2^{k}20\rho_{x_{i}}}(x_{i})\right)\\
    &\quad+c\sum_{j=1}^{k}2^{-j\left(\frac{sp}{p-1}-(\alpha+\tau+s)\right)}\left(\frac{1}{\tau}\dashint_{\mathcal{B}(x_{0},2^{j}20\rho_{x_{i}})}|D^{\tau}d_{\alpha+s}u|^{p}\dmut\right)^{\frac{1}{p}},
\end{aligned}
\end{equation*}
where we denote the two terms on the right-hand side by $T_{3}$ and $T_{4}$.
we first estimate $T_{3}$ as
\begin{equation*}
\begin{aligned}
T_{3}\leq \Tailp\left(\frac{u-(u)_{B_{2^{k}20\rho_{x_{i}}}(x_{i})}}{(2^{k}20\rho_{x_{i}})^{\alpha+\tau+s}};B_{2^{k}20\rho_{x_{i}}}(x_{i})\right)\leq c\lambda,
\end{aligned}
\end{equation*}
where we have used \eqref{firsttailu} along with \eqref{icond} and the fact that 
\begin{equation*}
    2^{-k\frac{sp}{p-1}}\leq 2^{-k(\alpha+\tau+s)}.
\end{equation*}
By H\"older's inequality and \eqref{exitradius}, we next estimate $T_{4}$ as  
\begin{equation*}
\begin{aligned}
T_{4}\leq\sum_{j=1}^{k}\frac{2^{-j\left(\frac{sp}{p-1}-(\alpha+\tau+s)\right)}}{\tau^{\frac{1}{p}}}\left(\dashint_{\mathcal{B}(x_{0},2^{j}20\rho_{x_{i}})}|D^{\tau}d_{\alpha+s}u|^{\tilde{p}}\dmut\right)^{\frac{1}{\tilde{p}}}\leq c\sum_{j=1}^{k}2^{-j\left(\frac{sp}{p-1}-(\alpha+\tau+s)\right)}\lambda
\end{aligned}
\end{equation*}
for some constant $c=c(n,s,p)$. 
Combining the estimates of $T_{3}$, $T_{4}$ and the fact that
\begin{equation*}
    \left(\frac{1}{\tau}\dashint_{\mathcal{B}_{20\rho_{x_{i}}}(x_{i})}|D^{\tau}d_{\alpha+s}u|^{p}\dmu_{\tau}\right)^{\frac{1}{p}}\leq \left(\frac{1}{\tau}\dashint_{\mathcal{B}_{20\rho_{x_{i}}}(x_{i})}|D^{\tau}d_{\alpha+s}u|^{\tilde{p}}\dmu_{\tau}\right)^{\frac{1}{\tilde{p}}}\leq\lambda,
\end{equation*}
which follows from H\"older's inequality and \eqref{exitradius}, we have 
\begin{equation}
\label{secondtailu}
    \left(\frac{1}{\tau}\dashint_{\mathcal{B}_{20\rho_{x_{i}}}(x_{i})}|D^{\tau}d_{\alpha+s}u|^{p}\dmu_{\tau}\right)^{\frac{1}{p}}+\Tails\left(\frac{u-(u)_{B_{20\rho_{x_{i}}}(x_{i})}}{(20\rho_{x_{i}})^{\alpha+\tau+s}};B_{20\rho_{x_{i}}}(x_{i})\right)\leq c\sum_{j=0}^{k}2^{-j\left(\frac{sp}{p-1}-(\alpha+\tau+s)\right)}\lambda
\end{equation}
for some constant $c=c(n,s,p)$. We follow the same lines for proving \eqref{secondtailu} with $u(x)$, $s$, $\tilde{p}$, $sp$ and $\lambda$ there, replaced by $f(x)$, $0$, $\tilde{q}$, $s$ and $\delta\lambda$, respectively, in order to obtain that  
\begin{equation}
\label{secondtailf}
    \left(\frac{1}{\tau}\dashint_{\mathcal{B}_{20\rho_{x_{i}}}(x_{i})}|D^{\tau}d_{\alpha}f|^{\tilde{q}}\dmu_{\tau}\right)^{\frac{1}{\tilde{q}}}+\Tails\left(\frac{f-(f)_{B_{20\rho_{x_{i}}}(x_{i})}}{(20\rho_{x_{i}})^{\alpha+\tau}};B_{20\rho_{x_{i}}}(x_{i})\right)\leq c\sum_{j=0}^{k}2^{-j\left(\frac{s}{p-1}-(\alpha+\tau)\right)}\delta\lambda.
\end{equation}
Applying \eqref{sumof2minus} to the last terms in \eqref{secondtailu} and \eqref{secondtailf}, we obtain \eqref{diagonal lambda estimate}.
\end{rmk}

\section{\texorpdfstring{$L^{q}$}{Lq}-estimate of \texorpdfstring{$d_{s}u$}{dtaudsu}}
\label{section5}
In the previous section, we have proved Lemma \ref{coveringoflevelset} which have constructed coverings of upper level sets for $D^{\tau}d_{\alpha+s}u$. We are now going to prove our main theorem \ref{main theorem} using a bootstrap argument.
To treat an iterative process, we need to introduce new parameters.
We first note that there is the smallest positive integer $l=l(n,s,p)$ such that
\begin{equation}
\label{rangeofl}
    n\leq (l+1)p\frac{s}{2} .
\end{equation}
Let us start with $p_{0}=p$.
If $l=1$, then define
\begin{equation}
\label{gamma0defn}
\gamma_{0}=2q.
\end{equation}
Otherwise, for any nonnegative integer $h<l-1$, we inductively define
\begin{equation}
\label{pgamma}
    p_{h+1}=\frac{np_{h}}{n-p_{h}\frac{s}{2}}\quad\text{and}\quad \gamma_{h}=\frac{np_{h}}{n-p_{h}s}
\end{equation}
with $\gamma_{l-1}=2q$. We note from \eqref{rangeofl} and \eqref{pgamma} that
\begin{equation*}
    p_{h+1}=\frac{np}{n-(h+1)\frac{sp}{2}}\quad\text{and}\quad\gamma_{h}=\frac{np}{n-(h+2)\frac{sp}{2}} \quad\text{for }h<l-1.
\end{equation*}
Then, we find the smallest nonnegative integer $l_{q}=l_{q}(n,s,p,q)$ such that
\begin{equation}
\label{lqrange}
     q<\gamma_{l_{q}}.
\end{equation}
Let $u\in W^{s,p}_{\mathrm{loc}}(B_{4})\cap L^{p-1}_{sp}(\mathbb{R}^{n})$ be a weak solution to the localized problem
\begin{equation}
\label{localizedproblem}
(-\Delta_{p})_{A}^{s}u=(-\Delta_{p})^{\frac{s}{p}}f\quad\text{in }B_{4}.
\end{equation}
Suppose that $f\in L^{p-1}_{s}(\mathbb{R}^{n})$ satisfies $d_{0}f\in L^{q}_{\mathrm{loc}}\left(\mathcal{B}_{4};\frac{\dx\dy}{|x-y|^{n+\sigma q}}\right)$ for any 
\begin{equation*}
    q\in(p,\infty)\quad\text{and}\quad\sigma\in\left(0, \min\left\{\frac{s}{p-1},1-s\right\}\right).
\end{equation*} 
We now divide this section into two subsections depending on the range of $\sigma$ below, \eqref{firstsigmares} and \eqref{assumsigmaforhig}.
\subsection{Restricted range of fractional differentiability $\sigma$.}
\label{subsection51}
We first show the following lemma which will be used to apply a bootstrap argument.  
\begin{lem}
\label{inductiveicz}
Let $u$ be a weak solution to \eqref{localizedproblem} with 
 \begin{equation}
\label{firstsigmares}
\sigma\in\left(0, \left(1-\frac{p}{q}\right)\min\left\{\frac{s}{p-1},1-s\right\}\right)
\end{equation}
and
$D^{\tau}d_{s}u\in L^{p_{h}}_{\mathrm{loc}}\left(\mathcal{B}_{4}\,;\mu_{\tau}\right)$, where $h\in\{0,1,\ldots, l_{q}\}$ and $\tau=\frac{q}{q-p}\sigma$.
Then there are a small positive constant $\delta=\delta(\mathsf{data})$ and a positive constant $c=c(\mathsf{data})$ independent of $h$ such that if $A$ is $(\delta,2)$-vanishing in $B_{4}\times B_{4}$ only at the diagonal, then
\begin{equation}
\label{ph1norm}
\begin{aligned}
\left(\dashint_{\mathcal{B}_{1}}|D^{\tau}d_{s}u|^{\hat{q}}\dmut\right)^{\frac{1}{\hat{q}}}&\leq c\left(\left(\dashint_{\mathcal{B}_{2}}|D^{\tau}d_{s}u|^{p_{h}}\dmut\right)^{\frac{1}{p_{h}}}+\Tailp\left(\frac{u-(u)_{B_{2}}}{2^{\tau+s}};B_{2}\right)\right)\\
&\quad+c\left(\left(\dashint_{\mathcal{B}_{2}}|D^{\tau}d_{0}f|^{\hat{q}}\dmu_{\tau}\right)^{\frac{1}{\hat{q}}}+\Tails\left(\frac{f-(f)_{B_{2}}}{2^{\tau}};B_{2}\right)\right),
\end{aligned}
\end{equation}
where the constant $\hat{q}$ is defined by
\begin{equation}
\label{tildeq}
\hat{q}=\begin{cases}
    p_{h+1}&\quad\text{if }\gamma_{h}\leq q,\\
    q&\quad\text{if }\gamma_{h}>q.
\end{cases}
\end{equation}
\end{lem}
\begin{proof}
By H\"older's inequality and the fact that $\tau=\frac{q}{q-p}\sigma$, we first note 
\begin{equation*}
\left(\dashint_{\mathcal{B}_{2}}|D^{\tau}d_{0}f|^{p}\dmu_{\tau}\right)^{\frac{1}{p}}\leq \left(\dashint_{\mathcal{B}_{2}}|D^{\tau}d_{0}f|^{q}\dmu_{\tau}\right)^{\frac{1}{q}}=\left(\frac{1}{\mu_{\tau}\left(\mathcal{B}_{2}\right)}\iint_{\mathcal{B}_{2}}|d_{0}f|^{q}\frac{\dx\dy}{|x-y|^{n+\sigma q}}\right)^{\frac{1}{q}}<\infty.
\end{equation*}
Let $\epsilon\in(0,1)$ which will be determined later. We then take $\delta=\delta(n,s,p,\Lambda,\epsilon)$ given in Lemma \ref{digagonal comparison2}. Let $1\leq r_{1}<r_{2}\leq2$.
We now apply Lemma \ref{coveringoflevelset} with $\alpha=0,\, \tilde{p}=p_{h}$ , $\tilde{q}=p$ and $\tilde{\gamma}=\gamma_{h}$ to find  families of countable disjoint diagonal balls and off-diagonal cubes, $\{\mathcal{B}_{\rho_{x_{i}}}\left(x_{i}\right)\}_{i\in\mathbb{N}}$ and $\{\mathcal{Q}\}_{\mathcal{Q}\in\mathcal{A}}$, such that
\begin{equation*}
   \left\{(x,y)\in \mathcal{B}_{r_{1}}\,;\,|D^{\tau}d_{s}u|(x,y)>\lambda\right\}\subset\left(\bigcup\limits_{i\in\mathbb{N}}\mathcal{B}_{5\rho_{x_{i}}}\left(x_{i}\right)\right)\bigcup \left(\bigcup\limits_{\mathcal{Q}\in\mathcal{A}}\mathcal{Q}\right)
\end{equation*}
for any fixed $\lambda\geq\lambda_{0}$, where $\lambda_{0}$ is given in \eqref{lambda0} with $\alpha=0$, $\tilde{p}=p_{h}$ and $\tilde{q}=p$. Furthermore, Lemma \ref{coveringoflevelset} yields \eqref{exitradius}, \eqref{gamma norm} and \eqref{sumofmeasq} with $\alpha=0,\, \tilde{p}=p_{h}$ , $\tilde{q}=p$ and $\tilde{\gamma}=\gamma_{h}$.
For $L\geq\lambda_{0}$, we define a function $\phi_{L}(r):[1,2]\to\mathbb{R}$ by
\begin{equation*}
    \phi_{L}(r)=\left(\dashint_{\mathcal{B}_{r}}\left|D^{\tau}d_{s}u\right|_{L}^{\hat{q}}\dmut\right)^{\frac{1}{\hat{q}}},
\end{equation*}
where $\left|D^{\tau}d_{s}u\right|_{L}=\max\{\left|D^{\tau}d_{s}u\right|,L\}$. We now want to show that if $L\geq \lambda_{0}$, then
\begin{equation}
\label{desiredgoal1}
\begin{aligned}
    \phi_{L}(r_{1})&\leq \frac{1}{2}\phi_{L}(r_{2})+\frac{c}{(r_{2}-r_{1})^{2n}}\left(\left(\dashint_{\mathcal{B}_{2}}|D^{\tau}d_{s}u|^{p_{h}}\dmut\right)^{\frac{1}{p_{h}}}+\Tailp\left(\frac{u-(u)_{B_{2}}}{2^{\tau+s}};B_{2}\right)\right)\\
&\quad+\frac{c}{(r_{2}-r_{1})^{2n}}\left(\left(\dashint_{\mathcal{B}_{2}}|D^{\tau}d_{0}f|^{\hat{q}}\dmu_{\tau}\right)^{\frac{1}{\hat{q}}}+\Tails\left(\frac{f-(f)_{B_{2}}}{2^{\tau}};B_{2}\right)\right)
\end{aligned}
\end{equation}
for some constant $c=c(\mathsf{data})$ independent of $L$. By \eqref{diagonal lambda estimate} with $\tilde{p}=p_{h}, \alpha=0$ and $\tilde{q}=p$, and the fact that $\tau<\frac{s}{p-1}$, we deduce that
\begin{equation*}
\begin{aligned}
    &\frac{1}{\tau}\dashint_{\mathcal{B}_{20\rho_{x_{i}}}(x_{i})}|D^{\tau}d_{s}u|^{{p}}\dmu_{\tau}+\Tailp\left(\frac{u-(u)_{B_{20\rho_{x_{i}}}}}{(20\rho_{x_{i}})^{\alpha+\tau+s}};B_{20\rho_{x_{i}}}(x_{i})\right)^{p}\leq (c_{1}\lambda)^{{p}},\\
    &\frac{1}{\tau}\dashint_{\mathcal{B}_{20\rho_{x_{i}}}(x_{i})}|D^{\tau}d_{\alpha}f|^{p}\dmu_{\tau}+\Tails\left(\frac{f-(f)_{B_{20\rho_{x_{i}}}}}{(20\rho_{x_{i}})^{\alpha+\tau}};B_{20\rho_{x_{i}}}(x_{i})\right)^{p}\leq (\delta c_{1}\lambda)^{p}
\end{aligned}
\end{equation*}
for some constant $c_{1}=c_{1}(n,s,p,q,\sigma)$. By the above choice of $\delta=\delta(n,s,p,\Lambda,\epsilon)$, we apply Lemma \ref{digagonal comparison} with $\Omega$ and $\lambda$ there, replaced by $B_{4}$ and $c_{1}\lambda$, respectively, to find a weak solution $v$ of \eqref{ncompsol} satisfying
\begin{equation}
\label{ncomp main res1}
    \frac{1}{\tau}\dashint_{\mathcal{B}_{5\rho_{x_{i}}}(x_{i})}|D^{\tau}d_{s}(v-u)|^{p}\dmu_{\tau}\leq (\epsilon c_{1}\lambda)^{p}\quad\text{and}\quad\|D^{\tau}d_{s}v\|_{L^{\infty}(\mathcal{B}_{5\rho_{x_{i}}}(x_{i}))}\leq c_{0}c_{1}\lambda
\end{equation} for some constant $c_{0}=c_{0}(\mathsf{data})$, as $\tau$ depends only on $p,q$ and $\sigma$. From Fubini's theorem, we have that
\begin{equation*}
\begin{aligned}
    \int_{\mathcal{B}_{r_{1}}}|D^{\tau}d_{s}u|_{L}^{\hat{q}}\dmut&=\int_{0}^{\infty}\hat{q}\lambda^{\hat{q}-1}\mu_{\tau}\left\{(x,y)\in\mathcal{B}_{r_{1}}\,;\,|D^{\tau}d_{s}u|_{L}(x,y)>\lambda\right\}\dlambda\\
    &=\int_{0}^{M\lambda_{0}}\hat{q}\lambda^{\hat{q}-1}\mu_{\tau}\left\{(x,y)\in\mathcal{B}_{r_{1}}\,;\,|D^{\tau}d_{s}u|(x,y)>\lambda\right\}\dlambda\\
    &\quad+\int_{M\lambda_{0}}^{L}\hat{q}\lambda^{\hat{q}-1}\mu_{\tau}\left\{(x,y)\in\mathcal{B}_{r_{1}}\,;\,|D^{\tau}d_{s}u|(x,y)>\lambda\right\}\dlambda\coloneqq I+J,
\end{aligned}
\end{equation*}
where $M>1$ will be selected later to control the off-diagonal upper level set of $D^{\tau}d_{s}u$ and $L>M\lambda_{0}$. We first observe that
\begin{equation*}
    I\leq (M\lambda_{0})^{\hat{q}}\mu_{\tau}(\mathcal{B}_{r_{1}}).
\end{equation*}
We next estimate $J$ as
\begin{equation}
\label{estimateofJ}
\begin{aligned}
    J&=\int_{\lambda_{0}}^{LM^{-1}} \hat{q}M^{\hat{q}}\lambda^{\hat{q}-1}\mu_{\tau}\left(\left\{(x,y)\in\mathcal{B}_{r_{1}}\,;\,|D^{\tau}d_{s}u(x,y)|>M\lambda\right\}\right)\dlambda\\
    &\leq \int_{\lambda_{0}}^{LM^{-1}}\hat{q} M^{\hat{q}}\lambda^{\hat{q}-1}\mu_{\tau}\left(\left\{(x,y)\in\left(\bigcup\limits_{i}\mathcal{B}_{5\rho_{x_{i}}}(x_{i})\right)\bigcup \left(\bigcup\limits_{\mathcal{Q}\in\tilde{\mathcal{A}}}\mathcal{Q}\right)\,;\, |D^{\tau}d_{s}u(x,y)|>M\lambda \right\}\right)\dlambda\\
    &\leq \int_{\lambda_{0}}^{LM^{-1}} \hat{q}M^{\hat{q}}\lambda^{\hat{q}-1}\mu_{\tau}\left(\left\{(x,y)\in\left(\bigcup\limits_{i}\mathcal{B}_{5\rho_{x_{i}}}(x_{i})\right) \,;\, |D^{\tau}d_{s}u(x,y)|>M\lambda  \right\}\right)\dlambda\\
    &\quad+\int_{\lambda_{0}}^{LM^{-1}} \hat{q}M^{\hat{q}}\lambda^{\hat{q}-1}\mu_{\tau}\left(\left\{(x,y)\in\left(\bigcup\limits_{\mathcal{Q}\in\tilde{\mathcal{A}}}\mathcal{Q}\right) \,;\, |D^{\tau}d_{s}u(x,y)|>M\lambda \right\}\right)\dlambda\eqqcolon J_{1}+J_{2},
\end{aligned}
\end{equation}
where we have used the change of variables for the first equality and the fact that 
\begin{equation*}
\begin{aligned}
    \left\{(x,y)\in\mathcal{B}_{r_{1}}\,;\,|D^{\tau}d_{s}u(x,y)|>M\lambda\right\}&\subset \left\{(x,y)\in\mathcal{B}_{r_{1}}\,;\,|D^{\tau}d_{s}u(x,y)|>\lambda\right\}\\
    &\subset\left(\bigcup\limits_{i}\mathcal{B}_{5\rho_{x_{i}}}(x_{i})\right)\bigcup \left(\bigcup\limits_{\mathcal{Q}\in\tilde{\mathcal{A}}}\mathcal{Q}\right) 
\end{aligned}
\end{equation*}
for the second inequality. By taking $M>c_{0}c_{1}$, we find that there is a constant $c=c(n,s,p,q,\sigma)$ such that
\begin{equation}
\label{diagonal level set estimate1}
\begin{aligned}
    &\int_{\lambda_{0}}^{LM^{-1}}M^{\hat{q}}\lambda^{\hat{q}-1}\mu_{\tau}\left(\left\{(x,y)\in\mathcal{B}_{5\rho_{x_{i}}}\left(x_{i}\right) \,;\, |D^{\tau}d_{s}u(x,y)|>M\lambda  \right\}\right)\dlambda\\
    &\leq \int_{\lambda_{0}}^{LM^{-1}}M^{\hat{q}}\lambda^{\hat{q}-1}\mu_{\tau}\left(\left\{(x,y)\in\mathcal{B}_{5\rho_{x_{i}}}\left(x_{i}\right) \,;\, |D^{\tau}d_{s}(u-v)(x,y)|>M\lambda  \right\}\right)\dlambda\\
    &\quad+\int_{\lambda_{0}}^{LM^{-1}}M^{\hat{q}}\lambda^{\hat{q}-1}\mu_{\tau}\left(\left\{(x,y)\in\mathcal{B}_{5\rho_{x_{i}}}\left(x_{i}\right) \,;\, |D^{\tau}d_{s}v(x,y)|>M\lambda  \right\}\right)\dlambda\\
    &= \int_{\lambda_{0}}^{LM^{-1}}M^{\hat{q}}\lambda^{\hat{q}-1}\mu_{\tau}\left(\left\{(x,y)\in\mathcal{B}_{5\rho_{x_{i}}}\left(x_{i}\right) \,;\, |D^{\tau}d_{s}(u-v)(x,y)|>M\lambda  \right\}\right)\dlambda\\
    &\leq \int_{\lambda_{0}}^{LM^{-1}}M^{\hat{q}}\lambda^{\hat{q}-1}\int_{\mathcal{B}_{5\rho_{x_{i}}}\left(x_{i}\right)}\frac{|D^{\tau}d_{s}(u-v)(x,y)|^{p}}{(M\lambda)^{p}}\dmut\dlambda\\
    &\leq cM^{\hat{q}-p}\epsilon^{p}\int_{\lambda_{0}}^{LM^{-1}}\lambda^{\hat{q}-1}\mu_{\tau}\left(\mathcal{B}_{5\rho_{x_{i}}}\left(x_{i}\right)\right)\dlambda\\
    &\leq cM^{\hat{q}-p}\epsilon^{p}\int_{\lambda_{0}}^{LM^{-1}}\lambda^{\hat{q}-1}\mu_{\tau}\left(\mathcal{B}_{\rho_{x_{i}}}\left(x_{i}\right)\right)\dlambda,
\end{aligned}
\end{equation}
where we have used the second inequality in \eqref{ncomp main res1} for the fourth equality, weak 1-1 estimate for the fifth inequality and the first inequality in \eqref{ncomp main res1} for the sixth inequality. We next note from \eqref{exitradius} with $\alpha=0$, $\tilde{p}=p_{h}$ and $\tilde{q}=p$ that either
\begin{equation*}
\begin{aligned}
    \frac{\lambda}{2c}\leq \left(\dashint_{\mathcal{B}_{\rho_{x_{i}}}(x_{i})}|D^{\tau}d_{s}u|^{p_{h}}\dmut\right)^{\frac{1}{p_{h}}}
\end{aligned}
\end{equation*}
or
\begin{equation*}
\begin{aligned}
    \frac{\lambda}{2c}\leq \frac{1}{\delta}\left(\dashint_{\mathcal{B}_{\rho_{x_{i}}}(x_{i})}|D^{\tau}d_{0}f|^{p}\dmut\right)^{\frac{1}{p}}
\end{aligned}
\end{equation*}
holds for some constant $c=c(n,s,p,q,\sigma)$. Considering the above two alternatives, we observe that
\begin{equation}
\label{mut dia2}
\begin{aligned}
    \mu_{\tau}\left(\mathcal{B}_{\rho_{x_{i}}}(x_{i})\right)&\leq \left(\frac{4c}{\lambda}\right)^{p_{h}}\int_{\mathcal{B}_{\rho_{x_{i}}}(x_{i})\cap\left\{ |D^{\tau}d_{s}u|>a\lambda\right\}}|D^{\tau}d_{s}u|^{p_{h}}\dmut\\
    &\quad+\left(\frac{4c}{\lambda}\right)^{p}\int_{\mathcal{B}_{\rho_{x_{i}}}(x_{i})\cap\left\{ |D^{\tau}d_{0}f|>b\lambda\right\}}\frac{1}{\delta^{p}}|D^{\tau}d_{0}f|^{p}\dmut,
\end{aligned}
\end{equation}
where we have chosen $a=\frac{1}{4c}$ and $b=\frac{\delta}{4c}$.
Using \eqref{diagonal level set estimate1}, \eqref{mut dia2} and the fact that $\left\{\mathcal{B}_{\rho_{x_{i}}}(x_{i})\right\}_{i}$ is the family of disjoint sets contained in $\mathcal{B}_{r_{2}}$, we have
\begin{equation*}
\begin{aligned}
    J_{1}&\leq \sum_{i}cM^{\hat{q}-p}\epsilon^{p}\int_{\lambda_{0}}^{LM^{-1}}\lambda^{\hat{q}-1-p_{h}}\int_{\mathcal{B}_{\rho_{x_{i}}}(x_{i})\cap\left\{ |D^{\tau}d_{s}u|>a\lambda\right\}}|D^{\tau}d_{s}u|^{p_{h}}\dmut\dlambda\\
    &\quad+\sum_{i}cM^{\hat{q}-p}\epsilon^{p}\int_{\lambda_{0}}^{LM^{-1}}\lambda^{\hat{q}-1-p}\int_{\mathcal{B}_{\rho_{x_{i}}}(x_{i})\cap\left\{ |D^{\tau}d_{0}f|>b\lambda\right\}}\frac{1}{\delta^{p}}|D^{\tau}d_{0}f|^{p}\dmut\dlambda\\
    &\leq cM^{\hat{q}-p}\epsilon^{p}\int_{\lambda_{0}}^{\infty}\lambda^{\hat{q}-p_{h}-1}\int_{\mathcal{B}_{r_{2}}\cap\left\{ |D^{\tau}d_{s}u|_{LM^{-1}}>a\lambda\right\}}|D^{\tau}d_{s}u|^{p_{h}}\dmut\dlambda\\
    &\quad+cM^{\hat{q}-p}\epsilon^{p}\int_{\lambda_{0}}^{\infty}\lambda^{\hat{q}-p-1}\int_{\mathcal{B}_{r_{2}}\cap\left\{ |D^{\tau}d_{0}f|>b\lambda\right\}}\frac{1}{\delta^{p}}|D^{\tau}d_{0}f|^{p}\dmut\dlambda\\
\end{aligned}
\end{equation*}
for some constant $c=c(n,s,p,q,\sigma)$. Apply Fubini's theorem to the  last inequality in order to discover that
\begin{equation}
\label{j1estimate}
J_{1}\leq cM^{\hat{q}-p}\epsilon^{p}\int_{\mathcal{B}_{r_{2}}}|D^{\tau}d_{s}u|_{LM^{-1}}^{\hat{q}-p_{h}}|D^{\tau}d_{s}u|^{p_{h}}\dmut+cM^{\hat{q}-p}\epsilon^{p}\int_{\mathcal{B}_{r_{2}}}\frac{1}{\delta^{\hat{q}}}|D^{\tau}d_{0}f|^{\hat{q}}\dmut.
\end{equation}
We now estimate the remaining term $J_{2}$. We first observe that there are constants $c=c(n,s,p,q,\sigma)$ and $c_{u}=c_{u}(n,s,p,q,\sigma)$ such that
\begin{equation}
\label{estimateofj2r}
\begin{aligned}
    J_{2}&\leq \sum_{\mathcal{Q}\in\tilde{\mathcal{A}}}\int_{\lambda_{0}}^{LM^{-1}} \hat{q}M^{\hat{q}}\lambda^{\hat{q}-1}\mu_{\tau}\left(\left\{(x,y)\in\mathcal{Q} \,;\, |D^{\tau}d_{s}u(x,y)|>M\lambda \right\}\right)\dlambda\\
    &\leq \sum_{\mathcal{Q}\in\tilde{\mathcal{A}}}\int_{\lambda_{0}}^{LM^{-1}} \hat{q}M^{\hat{q}}\lambda^{\hat{q}-1}\left(\int_{\mathcal{Q}\cap\left\{|D^{\tau}d_{s}u|>M\lambda\right\}}\left(\frac{|D^{\tau}d_{s}u|}{M\lambda}\right)^{\gamma_{h}}\dmut\right)\dlambda\\
    &\leq \sum_{\mathcal{Q}\in\tilde{\mathcal{A}}}\int_{\lambda_{0}}^{LM^{-1}} cM^{\hat{q}-\gamma_{h}}\lambda^{\hat{q}-1}\mu_{\tau}\left(\mathcal{Q}\right)\dlambda\\
    &\leq \int_{\lambda_{0}}^{LM^{-1}} cM^{\hat{q}-\gamma_{h}}\lambda^{\hat{q}-1-p_{h}}\int_{\mathcal{B}_{r_{2}\cap\left\{|D^{\tau}d_{s}u|>c_{u}\lambda\right\}}}|D^{\tau}d_{s}u|^{p_{h}}\dmut\dlambda,
\end{aligned}
\end{equation}
where we have used weak 1-1 estimate, \eqref{gamma norm} and \eqref{sumofmeasq} with $\alpha=0$, $\tilde{p}=p_{h}$ and $\tilde{\gamma}=\gamma_{h}$. Since the value of $\min\{\gamma_{h}-\hat{q}\,;\, h=0,1,\ldots l_{q}\}$ is positive and depends only on $n,s,p,q$ and $\sigma$, Fubini's theorem yields that
\begin{equation}
\label{chooseofM}
\begin{aligned}
    J_{2}\leq cM^{\hat{q}-\gamma_{h}}\int_{\mathcal{B}_{r_{2}}}\left|D^{\tau}d_{s}u\right|_{LM^{-1}}^{\hat{q}-p_{h}}|D^{\tau}d_{s}u|^{p_{h}}\dmut\leq \frac{1}{2^{2n+2q}} \int_{\mathcal{B}_{r_{2}}}\left|D^{\tau}d_{s}u\right|_{LM^{-1}}^{\hat{q}}\dmut
\end{aligned}
\end{equation}
by taking $cM^{\hat{q}-\gamma_{h}}\leq\frac{1}{2^{2n+2q}}$ and $M> c_{0}c_{1}$.
Since $M$ depends only on $\mathsf{data}$, we now choose $\epsilon=\epsilon(\mathsf{data})$ sufficiently small so that \eqref{j1estimate} becomes
\begin{equation*}
   J_{1}\leq \frac{1}{2^{2n+2q}}\int_{\mathcal{B}_{r_{2}}}\left|D^{\tau}d_{s}u\right|_{LM^{-1}}^{\hat{q}}\dmut+c\int_{\mathcal{B}_{r_{2}}}|D^{\tau}d_{0}f|^{\hat{q}}\dmut
\end{equation*}
for some constant $c=c(\mathsf{data})$. Combine all the estimates $I,J_{1}$ and $J_{2}$ to derive that for any $L>M\lambda_{0}$,
\begin{equation}
\label{techlemmabef}
\begin{aligned}
     \phi_{LM^{-1}}(r_{1})&\leq \frac{1}{2}\phi_{LM^{-1}}(r_{2})+\frac{c}{(r_{2}-r_{1})^{2n}}\left(\left(\dashint_{\mathcal{B}_{2}}|D^{\tau}d_{s}u|^{p_{h}}\dmut\right)^{\frac{1}{p_{h}}}+\Tailp\left(\frac{u-(u)_{B_{2}}}{2^{\tau+s}};B_{2}\right)\right)\\
&\quad+\frac{c}{(r_{2}-r_{1})^{2n}}\left(\left(\dashint_{\mathcal{B}_{2}}|D^{\tau}d_{0}f|^{\hat{q}}\dmu_{\tau}\right)^{\frac{1}{\hat{q}}}+\Tails\left(\frac{f-(f)_{B_{2}}}{2^{\tau}};B_{2}\right)\right),
\end{aligned}
\end{equation}
where we have used the fact that
\begin{equation*}
    \phi_{LM^{-1}}(r_{1})\leq \phi_{L}(r_{1})
\end{equation*}
and 
\begin{equation*}
    \frac{\mu_{\tau}(\mathcal{B}_{r_
    {2}})}{\mu_{\tau}(\mathcal{B}_{r_{1}})}\leq 2^{2n}.
\end{equation*}
Thus \eqref{desiredgoal1} follows from \eqref{techlemmabef}. Then, using Lemma \ref{technicallemma}, we get that if $L>\lambda_{0}$, then
\begin{equation*}
\begin{aligned}
    \phi_{L}(1)&\leq c\left(\left(\dashint_{\mathcal{B}_{2}}|D^{\tau}d_{s}u|^{p_{h}}\dmut\right)^{\frac{1}{p_{h}}}+\Tailp\left(\frac{u-(u)_{B_{2}}}{2^{\tau+s}};B_{2}\right)\right)\\
&\quad+c\left(\left(\dashint_{\mathcal{B}_{2}}|D^{\tau}d_{0}f|^{\hat{q}}\dmu_{\tau}\right)^{\frac{1}{\hat{q}}}+\Tails\left(\frac{f-(f)_{B_{2}}}{2^{\tau}};B_{2}\right)\right)
\end{aligned}
\end{equation*}
for some constant $c=c(\mathsf{data})$. By passing to the limit $L\to \infty$, we have the desired estimate \eqref{ph1norm}.
\end{proof}


With Lemma \ref{inductiveicz}, we now prove Theorem \ref{main theorem} for the restricted range of $\sigma$ given in \eqref{firstsigmares}.

\noindent
\textbf{Proof of Theorem \ref{main theorem} when \eqref{firstsigmares}.} We first take $\delta>0$ determined in Lemma \ref{inductiveicz}. 
We are now going to show that for any $B_{2r}(x_{0})\Subset\Omega$ with $r\in(0,R]$, there is a constant $c=c(\mathsf{data})$ such that
\begin{equation}
\label{indgoalr}
\begin{aligned}
\left(\dashint_{\mathcal{B}_{r}(x_{0})}|D^{\tau}d_{s}u|^{\hat{q}}\dmut\right)^{\frac{1}{\hat{q}}}&\leq c\left(\left(\dashint_{\mathcal{B}_{2r}(x_{0})}|D^{\tau}d_{s}u|^{p}\dmut\right)^{\frac{1}{p}}+\Tailp\left(\frac{u-\left(u\right)_{B_{2r}(x_{0})}}{(2r)^{\tau+s}};B_{2r}(x_{0})\right)\right)\\
&\quad+c\left(\left(\dashint_{\mathcal{B}_{2r}(x_{0})}|D^{\tau}d_{0}f|^{\hat{q}}\dmu_{\tau}\right)^{\frac{1}{\hat{q}}}+\Tails\left(\frac{f-\left(f\right)_{B_{2r}(x_{0})}}{(2r)^{\tau}};B_{2r}(x_{0})\right)\right),
\end{aligned}
\end{equation}
where 
\begin{equation}
\label{condoftau1}
    \tau=\frac{q}{q-p}\sigma
\end{equation}
and
\begin{equation}
\label{hatq}
\hat{q}=\begin{cases}
    p_{1}&\quad\text{if }\gamma_{0}\leq q,\\
    q&\quad\text{if }\gamma_{0}>q.
\end{cases}
\end{equation}
Let $y_{0}\in B_{2r}(x_{0})$. We define
\begin{equation*}
\begin{aligned}
    &\tilde{u}(x)=\left(\frac{r}{8}\right)^{-(\tau+s)}u\left(\frac{r}{8}x+y_{0}\right),\quad \tilde{f}(x)=\left(\frac{r}{8}\right)^{-\tau}f\left(\frac{r}{8}x+y_{0}\right),\quad\tilde{A}(x,y)=A\left(\frac{r}{8}x+y_{0},\frac{r}{8}y+y_{0}\right)
\end{aligned}
\end{equation*}
for $x,y\in \mathbb{R}^{n}$.  Then $\tilde{u}\in W^{s,p}(B_{4})\cap L^{p-1}_{sp}(\mathbb{R}^{n})$ is a weak solution to
\begin{equation*}
    (-\Delta)_{p,\tilde{A}}^{s}\tilde{u}=(-\Delta_{p})^{\frac{s}{p}}\tilde{f}\quad\text{in }B_{4},
\end{equation*}
where $\Tilde{A}$ is diagonal $(\delta,2)$-vanishing in $B_{4}\times B_{4}$. In addition, we observe that
\begin{equation*}
    d_{0}\tilde{f}\in L^{q}_{\mathrm{loc}}\left(\mathcal{B}_{4};\frac{\dx\dy}{|x-y|^{n+\sigma q}}\right)\quad\text{and}\quad D^{\tau}d_{s}\tilde{u}\in L^{p}_{\mathrm{loc}}\left(\mathcal{B}_{4}\,;\mu_{\tau}\right).
\end{equation*}
Therefore, we apply Lemma \ref{inductiveicz} with $u=\tilde{u}$, $f=\tilde{f}$, $A=\tilde{A}$ and $h=0$ to see that
\begin{equation*}
\begin{aligned}
\left(\dashint_{\mathcal{B}_{1}}|D^{\tau}d_{s}\tilde{u}|^{\hat{q}}\dmut\right)^{\frac{1}{\hat{q}}}&\leq c\left(\left(\dashint_{\mathcal{B}_{2}}|D^{\tau}d_{s}\tilde{u}|^{p}\dmut\right)^{\frac{1}{p}}+\Tailp\left(\frac{\tilde{u}-\left(\tilde{u}\right)_{B_{2}}}{2^{\tau+s}};B_{2}\right)\right)\\
&\quad+c\left(\left(\dashint_{\mathcal{B}_{2}}|D^{\tau}d_{0}\tilde{f}|^{\hat{q}}\dmu_{\tau}\right)^{\frac{1}{\hat{q}}}+\Tails\left(\frac{\tilde{f}-\left(\tilde{f}\right)_{B_{2}}}{2^{\tau}};B_{2}\right)\right)
\end{aligned}
\end{equation*}
for some constant $c=c(\mathsf{data})$. Using the change of variables, we then observe that 
\begin{equation}
\label{p0esti}
\begin{aligned}
\left(\dashint_{\mathcal{B}_{\frac{r}{8}}(y_{0})}|D^{\tau}d_{s}u|^{\hat{q}}\dmut\right)^{\frac{1}{\hat{q}}}&\leq c\left(\left(\dashint_{\mathcal{B}_{\frac{r}{4}}(y_{0})}|D^{\tau}d_{s}u|^{p}\dmut\right)^{\frac{1}{p}}+\Tailp\left(\frac{u-(u)_{B_{\frac{r}{4}}(y_{0})}}{\left(\frac{r}{4}\right)^{\tau+s}};B_{\frac{r}{4}}(y_{0})\right)\right)\\
&+c\left(\left(\dashint_{\mathcal{B}_{\frac{r}{4}}(y_{0})}|D^{\tau}d_{0}f|^{\hat{q}}\dmu_{\tau}\right)^{\frac{1}{\hat{q}}}+\Tails\left(\frac{f-(f)_{B_{\frac{r}{4}}(y_{0})}}{\left(\frac{r}{4}\right)^{\tau}};B_{\frac{r}{4}}(y_{0})\right)\right).
\end{aligned}
\end{equation}
Applying \eqref{tail estimate of u local2} with $\rho=\frac{r}{4}$, $R=2r$, $\alpha=0$ and $t=s$ and then using \eqref{condoftau1}, we find that
\begin{equation}
\label{tailestimate1}
\begin{aligned}
    \Tailp\left(\frac{u-(u)_{B_{\frac{r}{4}}(y_{0})}}{\left(\frac{r}{4}\right)^{\tau+s}};B_{\frac{r}{4}}(y_{0})\right)
    &\leq c\Tailp\left(\frac{u-(u)_{B_{2r}(x_{0})}}{\left(2r\right)^{\tau+s}};B_{2r}(x_{0})\right)+ c\left(\dashint_{\mathcal{B}_{2r}(x_{0})}|D^{\tau}d_{s}u|^{p}\dmut\right)^{\frac{1}{p}}
\end{aligned}
\end{equation}
for some constant $c=c(\mathsf{data})$. Likewise, \eqref{tail estimate of u local2} with $u(x)=f(x)$ $\rho=\frac{r}{4}$, $R=2r$, $s=\frac{s}{p}$, $\alpha=0$ and $t=0$, and \eqref{condoftau1} yield that
\begin{equation}
\label{tailestimate2}
\begin{aligned}
       \Tails\left(\frac{f-(f)_{B_{\frac{r}{4}}(y_{0})}}{\left(\frac{r}{4}\right)^{\tau}};B_{\frac{r}{4}}(y_{0})\right)&\leq c\Tails\left(\frac{f-(f)_{B_{2r}(x_{0})}}{\left(2r\right)^{\tau}};B_{2r}(x_{0})\right)+ c\left(\dashint_{\mathcal{B}_{2r}(x_{0})}|D^{\tau}d_{0}f|^{p}\dmut\right)^{\frac{1}{p}}\\
       &\leq c\Tails\left(\frac{f-(f)_{B_{2r}(x_{0})}}{\left(2r\right)^{\tau}};B_{2r}(x_{0})\right)+ c\left(\dashint_{\mathcal{B}_{2r}(x_{0})}|D^{\tau}d_{0}f|^{\hat{q}}\dmut\right)^{\frac{1}{\hat{q}}},
\end{aligned}
\end{equation}
where we have used H\"older's inequality for the last inequality. By inserting \eqref{tailestimate1} and \eqref{tailestimate2} into \eqref{p0esti} and using the fact that $\mathcal{B}_{\frac{r}{8}}(y_{0})\subset \mathcal{B}_{2r}(x_{0})$, we have 
\begin{equation}
\label{indesti}
\begin{aligned}
\left(\dashint_{\mathcal{B}_{\frac{r}{8}}(y_{0})}|D^{\tau}d_{s}u|^{\hat{q}}\dmut\right)^{\frac{1}{\hat{q}}}&\leq c\left(\left(\dashint_{\mathcal{B}_{2r}(x_{0})}|D^{\tau}d_{s}\tilde{u}|^{p}\dmut\right)^{\frac{1}{p}}+\Tailp\left(\frac{u-(u)_{B_{\frac{r}{2}}(y_{0})}}{\left(2r\right)^{\tau+s}};B_{2r}(x_{0})\right)\right)\\
&+c\left(\left(\dashint_{\mathcal{B}_{2r}(x_{0})}|D^{\tau}d_{0}\tilde{f}|^{\hat{q}}\dmu_{\tau}\right)^{\frac{1}{\hat{q}}}+\Tails\left(\frac{f-(f)_{B_{2r}(x_{0})}}{\left(2r\right)^{\tau}};B_{2r}(x_{0})\right)\right).
\end{aligned}
\end{equation}
On the other hand, using  H\"older's inequality along with the fact that $\hat{q}\leq \gamma_{0}$ by \eqref{gamma0defn}, \eqref{pgamma} and \eqref{hatq}, and then applying Lemma \ref{nd lemma} with $\tilde{\gamma}=\gamma_{0}$, $\alpha=0$ and $\tilde{p}=p$, we find that  there is a constant $c=c(n,s,p,q,\sigma)$ such that for any $\mathcal{Q}\equiv Q_{\frac{r}{8\sqrt{n}}}(z_{1})\times Q_{\frac{r}{8\sqrt{n}}}(z_{2})\Subset \Omega\times\Omega$ satisfying \eqref{dist cube}, 
\begin{equation}
\label{ndiagonallastest}
\begin{aligned}
\left(\dashint_{\mathcal{Q}}|D^{\tau}d_{s}u|^{\hat{q}}\dmut\right)^{\frac{1}{\hat{q}}}&\leq \left(\dashint_{\mathcal{Q}}|D^{\tau}d_{s}u|^{\gamma_{0}}\dmut\right)^{\frac{1}{\gamma_{0}}}\\
&\leq c\left(\dashint_{\mathcal{Q}}|D^{\tau}d_{s}u|^{p}\dmut\right)^{\frac{1}{p}}+c\left[\sum_{d=1}^{2}\left(\frac{1}{\tau}\dashint_{P^{d}\mathcal{Q}}|D^{\tau}d_{s}u|^{p}\dmut\right)^{\frac{1}{p}}\right].
\end{aligned}
\end{equation}
Let us assume that $z_{1}, z_{2}\in B_{r}(x_{0})$. From \eqref{measure of qintau} together with the observation that
\begin{equation*}
    \frac{r}{8\sqrt{n}}=d(\mathcal{Q})<|x-y|<r\quad\text{for any }(x,y)\in \mathcal{Q},
\end{equation*} we have
\begin{equation}
\label{mutaucube}
    \frac{r^{n+p\tau}}{c}\leq\mu_{\tau}(\mathcal{Q})\leq cr^{n+p\tau}
\end{equation}
for some constant $c=c(n,s,p,q,\sigma)$.
Using \eqref{condoftau1}, \eqref{mutaucube} and the fact that $\mathcal{Q},P^{1}\mathcal{Q},P^{2}\mathcal{Q}\subset \mathcal{B}_{2r}(x_{0})$, we further estimate the both sides of \eqref{ndiagonallastest} to see that
\begin{equation}
\label{offdiagonal0}
    \left(\frac{1}{r^{n+p\tau}}\int_{\mathcal{Q}}|D^{\tau}d_{s}u|^{\hat{q}}\dmut\right)^{\frac{1}{\hat{q}}}\leq c\left(\dashint_{\mathcal{B}_{2r}(x_{0})}|D^{\tau}d_{s}u|^{p}\dmut\right)^{\frac{1}{p}}
\end{equation}
for some constant $c=c(\mathsf{data})$. Since we can cover $\mathcal{B}_{r}(x_{0})$ with finitely many diagonal balls $\mathcal{B}_{\frac{r}{8}}\left(y_{i}\right)$ and off-diagonal cubes $\mathcal{Q}_{\frac{r}{8\sqrt{n}}}\left(z_{1,i},z_{2,i}\right)$ satisfying \eqref{dist cube} for some  $y_{i},\,z_{1,i},\,z_{2,i}\in B_{r}(x_{0})$ ,
the standard covering argument along with \eqref{condoftau1}, \eqref{indesti} and \eqref{offdiagonal0} leads to \eqref{indgoalr}. Since we have arbitrarily chosen $x_{0},z_{1},z_{2}\in \Omega$ and $r\in(0,R]$ satisfying $B_{2r}(x_{0})\Subset\Omega$ and $\mathcal{Q}_{\frac{r}{8\sqrt{n}}}\left(z_{1},z_{2}\right)\Subset\Omega\times\Omega$ with \eqref{dist cube}, it follows from \eqref{indgoalr} and \eqref{ndiagonallastest} that $D^{\tau}d_{s}u\in L^{\hat{q}}_{\mathrm{loc}}\left(\Omega\times\Omega\,;\mu_{\tau}\right)$. 
If $l_{q}=0$, where $l_{q}$ is given in \eqref{lqrange}, then $\hat{q}=q$.  Let us assume that $l_{q}>0$. Then the fact that $\hat{q}=p_{1}$ yields  $D^{\tau}d_{s}\tilde{u}\in L^{p_{1}}_{\mathrm{loc}}\left(\mathcal{B}_{4}\,;\mu_{\tau}\right)$. Thus, we apply Lemma \ref{inductiveicz} with $u=\tilde{u}$, $f=\tilde{f}$, $A=\tilde{A}$ and $h=1$, and follow the same arguments as in \eqref{p0esti} to find that
\begin{equation}
\label{estimate11forr}
\begin{aligned}
\left(\dashint_{\mathcal{B}_{\frac{r}{8}}(y_{0})}|D^{\tau}d_{s}u|^{\hat{q}_{1}}\dmut\right)^{\frac{1}{\hat{q}_{1}}}&\leq c\left(\left(\dashint_{\mathcal{B}_{\frac{r}{4}}(y_{0})}|D^{\tau}d_{s}u|^{p_{1}}\dmut\right)^{\frac{1}{p_{1}}}+\Tailp\left(\frac{u-(u)_{B_{\frac{r}{4}}(y_{0})}}{\left(\frac{r}{4}\right)^{\tau+s}};B_{\frac{r}{4}}(y_{0})\right)\right)\\
&+c\left(\left(\dashint_{\mathcal{B}_{\frac{r}{4}}(y_{0})}|D^{\tau}d_{0}f|^{\hat{q}_{1}}\dmu_{\tau}\right)^{\frac{1}{\hat{q}_{1}}}+\Tails\left(\frac{f-(f)_{B_{\frac{r}{4}}(y_{0})}}{\left(\frac{r}{4}\right)^{\tau}};B_{\frac{r}{4}}(y_{0})\right)\right),
\end{aligned}
\end{equation}
where 
\begin{equation}
\label{hatq1}
\hat{q}_{1}=\begin{cases}
    p_{2}&\quad\text{if }\gamma_{1}\leq q,\\
    q&\quad\text{if }\gamma_{1}>q.
\end{cases}
\end{equation}
Inserting \eqref{tail estimate of u local2} with $\rho=\frac{r}{4}$, $R=2r$, $\alpha=0$ and $t=s$ into \eqref{indgoalr} with $x_{0}$ and $r$ there, replaced by $y_{0}$ and $\frac{r}{4}$, respectively, and then using the fact that $\mathcal{B}_{\frac{r}{2}}(y_{0})\subset\mathcal{B}_{2r}(x_{0})$, we have  
\begin{equation}
\label{estimate12forr}
\begin{aligned}
\left(\dashint_{\mathcal{B}_{\frac{r}{4}}(y_{0})}|D^{\tau}d_{s}u|^{p_{1}}\dmut\right)^{\frac{1}{p_{1}}} &\leq c\left(\left(\dashint_{\mathcal{B}_{2r}(x_{0})}|D^{\tau}d_{s}u|^{p}\dmut\right)^{\frac{1}{p}}+\Tailp\left(\frac{u-\left(u\right)_{B_{2r}(x_{0})}}{\left(2r\right)^{\tau+s}};B_{2r}(x_{0})\right)\right)\\
&\quad+c\left(\left(\dashint_{\mathcal{B}_{2r}(x_{0})}|D^{\tau}d_{0}f|^{\hat{q}_{1}}\dmu_{\tau}\right)^{\frac{1}{\hat{q}_{1}}}+\Tails\left(\frac{f-\left(f\right)_{B_{\frac{r}{2}}(y_{0})}}{\left(\frac{r}{2}\right)^{\tau}};B_{\frac{r}{2}}(y_{0})\right)\right),
\end{aligned}
\end{equation}
where we have used H\"older's inequality along with the fact that $p_{1}\leq\hat{q}_{1}$ for the third term in the right-hand side of \eqref{estimate12forr}. We now combine \eqref{estimate11forr} and \eqref{estimate12forr} to obtain that
\begin{equation}
\label{indesti2}
\begin{aligned}
\left(\dashint_{\mathcal{B}_{\frac{r}{8}}(y_{0})}|D^{\tau}d_{s}u|^{\hat{q}_{1}}\dmut\right)^{\frac{1}{\hat{q}_{1}}}&\leq c\left(\left(\dashint_{\mathcal{B}_{2r}(x_{0})}|D^{\tau}d_{s}u|^{p}\dmut\right)^{\frac{1}{p}}+\Tailp\left(\frac{u-(u)_{B_{\frac{r}{2}}(y_{0})}}{\left(2r\right)^{\tau+s}};B_{2r}(x_{0})\right)\right)\\
&+c\left(\left(\dashint_{\mathcal{B}_{2r}(x_{0})}|D^{\tau}d_{0}f|^{\hat{q}_{1}}\dmu_{\tau}\right)^{\frac{1}{\hat{q}_{1}}}+\Tails\left(\frac{f-(f)_{B_{2r}(x_{0})}}{\left(2r\right)^{\tau}};B_{2r}(x_{0})\right)\right).
\end{aligned}
\end{equation}
For the fourth terms in the right-hand side of \eqref{estimate11forr} and \eqref{estimate12forr}, we have used \eqref{tail estimate of u local2} with $u(x)=f(x)$ $\rho=\frac{r}{2}$ or $\frac{r}{4}$, $R=2r$, $s=\frac{s}{p}$, $\alpha=0$ and $t=0$, and H\"older's inequality along with the fact that $p\leq \hat{q}_{1}$.
Furthermore, as in \eqref{ndiagonallastest}, Lemma \ref{nd lemma} with $\tilde{\gamma}=\gamma_{1}$, $\alpha=0$ and $\tilde{p}=p_{1}$ yields  
\begin{equation}
\label{ndiagonallastest2}
\begin{aligned}
\left(\dashint_{\mathcal{Q}}|D^{\tau}d_{s}u|^{\hat{q}_{1}}\dmut\right)^{\frac{1}{\hat{q}_{1}}}\leq c\left(\dashint_{\mathcal{Q}}|D^{\tau}d_{s}u|^{p}\dmut\right)^{\frac{1}{p}}+c\left[\sum_{d=1}^{2}\left(\dashint_{P^{d}\mathcal{Q}}|D^{\tau}d_{s}u|^{p_{1}}\dmut\right)^{\frac{1}{p_{1}}}\right],
\end{aligned}
\end{equation}
provided that $\mathcal{Q}=\mathcal{Q}_{\frac{r}{8\sqrt{n}}}\left(z_{1},z_{2}\right)\Subset \Omega\times\Omega$ satisfying \eqref{dist cube}. Let $z_{1},z_{2}\in B_{r}(x_{0})$. Since $P^{d}\mathcal{Q}\subset \mathcal{B}_{\frac{r}{8}}(z_{d})$, H\"older's inequality and the same computations as in \eqref{indesti2} with $y_{0}=z_{d}$ give that
\begin{equation}
\label{offprod}
\begin{aligned}
\left(\dashint_{P^{d}\mathcal{Q}}|D^{\tau}d_{s}u|^{p_{1}}\dmut\right)^{\frac{1}{p_{1}}}&\leq \left(\dashint_{\mathcal{B}_{\frac{r}{8}}(z_{d})}|D^{\tau}d_{s}u|^{\hat{q}_{1}}\dmut\right)^{\frac{1}{\hat{q}_{1}}}\\
&\leq c\left(\left(\dashint_{\mathcal{B}_{2r}(x_{0})}|D^{\tau}d_{s}u|^{p}\dmut\right)^{\frac{1}{p}}+\Tailp\left(\frac{u-(u)_{B_{\frac{r}{2}}(y_{0})}}{\left(2r\right)^{\tau+s}};B_{2r}(x_{0})\right)\right)\\
&+c\left(\left(\dashint_{\mathcal{B}_{2r}(x_{0})}|D^{\tau}d_{0}f|^{\hat{q}_{1}}\dmu_{\tau}\right)^{\frac{1}{\hat{q}_{1}}}+\Tails\left(\frac{f-(f)_{B_{2r}(x_{0})}}{\left(2r\right)^{\tau}};B_{2r}(x_{0})\right)\right).
\end{aligned}
\end{equation}
In light of \eqref{mutaucube}, \eqref{offprod} and the fact that $\mathcal{Q}\subset \mathcal{B}_{2r}(x_{0})$, we further estimate the both sides of \eqref{ndiagonallastest2} to see
\begin{equation}
\label{offdiagonal2}
\begin{aligned}
        \left(\frac{1}{r^{n+p\tau}}\int_{\mathcal{Q}}|D^{\tau}d_{s}u|^{\hat{q}_{1}}\dmut\right)^{\frac{1}{\hat{q}_{1}}}&\leq c\left(\left(\dashint_{\mathcal{B}_{2r}(x_{0})}|D^{\tau}d_{s}u|^{p}\dmut\right)^{\frac{1}{p}}+\Tailp\left(\frac{u-(u)_{B_{2r}(x_{0})}}{\left(2r\right)^{\tau+s}};B_{2r}(x_{0})\right)\right)\\
&+c\left(\left(\dashint_{\mathcal{B}_{2r}(x_{0})}|D^{\tau}d_{0}f|^{\hat{q}_{1}}\dmu_{\tau}\right)^{\frac{1}{\hat{q}_{1}}}+\Tails\left(\frac{f-(f)_{B_{2r}(x_{0})}}{\left(2r\right)^{\tau}};B_{2r}(x_{0})\right)\right).
\end{aligned}
\end{equation}
As in the case when $l_{q}=0$, the standard covering argument together with \eqref{condoftau1}, \eqref{indesti2},  \eqref{ndiagonallastest2} and \eqref{offdiagonal2} yields \eqref{indgoalr} with $\hat{q}=\hat{q}_{1}$ and $D^{\tau}d_{s}u\in L^{\tilde{q}_{1}}_{\mathrm{loc}}\left(\Omega\times\Omega\,;\,\mu_{\tau}\right)$. By iterating the above procedure $l_{q}-1$ times, we obtain \eqref{indgoalr} with $\hat{q}=q$ and $D^{\tau}d_{s}u\in L^{q}_{\mathrm{loc}}\left(\Omega\times\Omega\,;\,\mu_{\tau}\right)$. Thus, the desired estimate \eqref{goalestimate} follows by recalling the definition \eqref{condoftau1} and using a few algebraic computations. It completes the proof. \qed


\subsection{Improved range of fractional differentiability $\sigma$}
\label{section52}
We first prove a comparison estimate of the fractional gradient  $d_{s+\alpha}u$ for some $\alpha\in(0,1)$.
To this end, we introduce new parameters. 
We first assume that 
\begin{equation}
\label{assumsigmaforhig}
    \left(1-\frac{p}{q}\right)\min\left\{\frac{s}{p-1},1-s\right\}\leq\sigma< \min\left\{\frac{s}{p-1},1-s\right\}.
\end{equation}
We then choose $\beta\in(0,1)$ such that 
\begin{equation}
\label{betacond}
\sigma=\beta\min\left\{\frac{s}{p-1},1-s\right\}.
\end{equation}
For each $m=0,1,\ldots,$ we now take 
\begin{equation}
\label{tildetau}
\tilde{\tau}_{0}=\frac{\beta+1}{2}\min\left\{\frac{s}{p-1},1-s\right\},\quad q_{0}=\begin{cases} \min\{p_{1},q\}&\quad\text{if }sp\leq n\\
q&\quad\text{if }sp> n
\end{cases}\quad\text{and}\quad \tilde{\tau}_{m+1}=\frac{q_{0}+p}{2q_{0}}\tilde{\tau}_{m}
\end{equation}
and
\begin{equation}
\label{tildealpha}
\tilde{\alpha}_{m}=\tilde{\tau}_{0}-\tilde{\tau}_{m}
\end{equation}
to see that
\begin{equation}
\label{talphatau}
\tilde{\alpha}_{m}+\left(1-\frac{p}{q}\right)\tilde{\tau}_{m}=\tilde{\tau_{0}}-\frac{p}{q}\tilde{\tau}_{m}=\tilde{\tau}_{0}\left(1-\frac{p}{q}\left(\frac{q_{0}+p}{2q_{0}}\right)^{m}\right),
\end{equation}
where $p_{1}$ is given in \eqref{pgamma}.
From \eqref{assumsigmaforhig}, \eqref{betacond}, \eqref{tildetau} and the fact that 
\begin{equation*}
    \tilde{\tau}_{0}\left(1-\left(\frac{p}{q}\right)\right)<\sigma\quad\text{and}\quad \tilde{\tau}_{0}>\sigma,
\end{equation*} there is a positive integer $n_{\sigma}=n_{\sigma}(n,s,p,q,\sigma)$ such that
\begin{equation*}
\tilde{\tau}_{0}\left(1-\frac{p}{q}\left(\frac{q_{0}+p}{2q_{0}}\right)^{n_{\sigma}}\right)\geq\sigma\quad\text{and}\quad \tilde{\tau}_{0}\left(1-\frac{p}{q}\left(\frac{q_{0}+p}{2q_{0}}\right)^{n_{\sigma}-1}\right)<\sigma.
\end{equation*}
We now select $\tau_{0}\in(0,\tilde{\tau}_{0}]$ such that
\begin{equation}
\label{taucond}
\tau_{0}\left(1-\frac{p}{q}\left(\frac{q_{0}+p}{2q_{0}}\right)^{n_{\sigma}}\right)=\sigma\quad\text{and}\quad \tau_{0}\left(1-\frac{p}{q}\left(\frac{q_{0}+p}{2q_{0}}\right)^{n_{\sigma}-1}\right)<\sigma.
\end{equation}
Similarly, we take 
\begin{equation}
\label{taualphacond}
\tau_{m+1}=\frac{q_{0}+p}{2q_{0}}\tau_{m}\quad\text{and}\quad \alpha_{m}=\tau_{0}-\tau_{m}
\end{equation}
to observe that
\begin{equation}
\label{alphatauplus}
\alpha_{n_{\sigma}}+\left(1-\frac{p}{q}\right)\tau_{n_{\sigma}}=\sigma\quad\text{and}\quad \alpha_{i}+\left(1-\frac{p}{q}\right)\tau_{i}<\sigma\quad(i=0,1,\ldots, n_{\sigma}-1),
\end{equation}
where we have used \eqref{tildealpha}, \eqref{talphatau}, \eqref{taucond} and \eqref{taualphacond}.
In this setting, we define 
\begin{equation*}
\tilde{p}_{m}=p+\sum_{i=0}^{m}\frac{1}{2^{i+2}}\frac{q_{0}-p}{4}\quad\text{and}\quad \tilde{q}_{m}=q_{0}\left(\frac{3p+q_{0}}{2(q_{0}+p)}+\sum_{i=0}^{m}\frac{1}{2^{i+2}}\frac{q_{0}-p}{2(q_{0}+p)}\right)
\end{equation*}
to find that
\begin{equation}
\label{tildepqrange}
    p<\tilde{p}_{m}<\frac{3p+q_{0}}{4}< q_{0}\frac{3p+q_{0}}{2(q_{0}+p)}<\tilde{q}_{m}<q_{0}.
\end{equation}
In light of \eqref{tildepqrange} and \eqref{taualphacond}, we find 
\begin{equation}
\label{indnumforemb}
    \alpha_{m}+\left(1-\frac{p}{\tilde{p}_{m}}\right)\tau_{m}<\alpha_{m-1}+\left(1-\frac{p}{\tilde{q}_{m}}\right)\tau_{m-1}.
\end{equation}
Fix a nonnegative integer $m\in[0,n_{\sigma}]$. Throughout the proof of the next two lemmas, we assume that any weak solution 
$u\in W^{s,p}_{\mathrm{loc}}(B_{4})\cap L^{p-1}_{sp}(\mathbb{R}^{n})$ to the localized problem
\begin{equation}
\label{localizedproblemi}
(-\Delta_{p})_{A}^{s}u=(-\Delta_{p})^{\frac{s}{p}}f\quad\text{in }B_{4},
\end{equation}
where $f\in L^{p-1}_{s}(\mathbb{R}^{n})$ with $d_{0}f\in L^{q}_{\mathrm{loc}}\left(\mathcal{B}_{4};\frac{\dx\dy}{|x-y|^{n+\sigma q}}\right)$ satisfies  $D^{\tau_{m}}d_{\alpha_{m}+s}u\in L^{p}_{\mathrm{loc}}\left(\mathcal{B}_{4}\,;\,\mu_{\tau_{m}}\right)$ with the estimates
\begin{equation}
\label{indforhigh}
\begin{aligned}
    &\left(\frac{1}{\tau_{m}}\dashint_{\mathcal{B}_{r}(x_{0})}|D^{\tau_{m}}d_{\alpha_{m}+s}u|^{\tilde{p}_{m}}\dmu_{\tau_{m}}\right)^{\frac{1}{\tilde{p}_{m}}}\\
    &\leq d_{m}\left[\left(\frac{1}{\tau_{m}}\dashint_{\mathcal{B}_{2r}(x_{0})}|D^{\tau_{m}}d_{\alpha_{m}+s}u|^{p}\dmu_{\tau_{m}}\right)^{\frac{1}{p}}+\Tailp\left(\frac{u-(u)_{B_{2r}(x_{0})}}{(2r)^{\alpha_{m}+\tau_{m}+s}}; B_{2r}(x_{0})\right)\right]\\
    &\quad+d_{m}\left[\left(\frac{1}{\tau_{m}}\dashint_{\mathcal{B}_{2r}(x_{0})}|D^{\tau_{m}}d_{\alpha_{m}}f|^{\tilde{q}_{m}}\dmu_{\tau_{m}}\right)^{\frac{1}{\tilde{q}_{m}}}+\Tails\left(\frac{f-(f)_{B_{2r}(x_{0})}}{(2r)^{\alpha_{m}+\tau_{m}}}; B_{2r}(x_{0})\right)\right]
\end{aligned}
\end{equation}
for some constant $d_{m}>0$, provided that $B_{4r}(x_{0})\Subset B_{4}$. We remark that applying Lemma \ref{embedding sob} with $p=\tilde{q}_{m}$, $s_{1}=\alpha_{m}+\left(1-\frac{p}{\tilde{q}_{m}}\right)\tau_{m}$ and $s_{2}=\sigma$, we discover that
\begin{equation*}
    \left(\frac{1}{\tau_{m}}\dashint_{\mathcal{B}_{2r}(x_{0})}|D^{\tau_{m}}d_{\alpha_{m}}f|^{\tilde{q}_{m}}\dmu_{\tau_{m}}\right)^{\frac{1}{\tilde{q}_{m}}}\leq c[f]_{W^{\alpha_{m}+\left(1-\frac{p}{\tilde{q}_{m}}\right)\tau_{m},\tilde{q}_{m}}(B_{2r}(x_{0}))}\leq c[f]_{W^{\sigma,q}(B_{2r}(x_{0}))}<\infty,
\end{equation*}
which implies that the right-hand side of \eqref{indforhigh} is well-defined. We now improve the comparison lemma \ref{digagonal comparison2} using an interpolation argument. 

\begin{lem}
\label{diagonal induction}
For any $\epsilon>0$, there is a constant $\delta=\delta(\mathsf{data},\epsilon,d_{m})$ such that for any weak solution $u$ to \eqref{localizedproblemi} with
\begin{equation*}
    B_{20\rho_{x_{i}}}(x_{i})\Subset B_{4}
\end{equation*}and
\begin{equation}
\begin{aligned}
\label{scaling assumption2}
    &\frac{1}{\tau_{m}}\dashint_{\mathcal{B}_{20\rho_{x_{i}}}(x_{0})}|D^{\tau_{m}}d_{\alpha_{m}+s}u|^{p}\dmu_{\tau_{m}}+\Tailp\left(\frac{u-(u)_{B_{20\rho_{x_{i}}}(x_{i})}}{(4r)^{\alpha_{m}+\tau_{m}+s}};B_{20\rho_{x_{i}}}(x_{i})\right)^{p}\leq \lambda^{p},\\
    &\frac{1}{\tau_{m}}\dashint_{\mathcal{B}_{20\rho_{x_{i}}}(x_{i})}|D^{\tau_{m}}d_{\alpha_{m}}f|^{\tilde{q}_{m}}\dmu_{\tau_{m}}+\Tails\left(\frac{f-(f)_{B_{20\rho_{x_{i}}}(x_{i})}}{(4r)^{\alpha_{m}+\tau_{m}}};B_{20\rho_{x_{i}}}(x_{i})\right)^{\tilde{q}_{m}}\\
    &\quad+\left(\dashint_{B_{10\rho_{x_{i}}}(x_{i})}\dashint_{B_{10\rho_{x_{i}}}(x_{i})}\lambda|A(x,y)-(A)_{B_{10\rho_{x_{i}}}(x_{i})\times B_{10\rho_{x_{i}}}(x_{i})}|\dx\dy\right)^{\tilde{q}_{m}}\leq (\delta\lambda)^{\tilde{q}_{m} },
\end{aligned}
\end{equation}
there exists a weak solution $v\in W^{s,p}(B_{10\rho_{x_{i}}}(x_{i}))\cap L^{p-1}_{sp}(\mathbb{R}^{n})$ to 
\begin{equation}
\label{sol v for impr}
    (-\Delta)_{p,A_{10\rho_{x_{i}},x_{i}}}^{s}v=0\quad\text{in }B_{10\rho_{x_{i}}}(x_{i})
\end{equation}
such that 
\begin{equation}
\label{compar result2}
    \frac{1}{\tau_{m}}\dashint_{\mathcal{B}_{5\rho_{x_{i}}}(x_{i})}|D^{\tau_{m}}d_{\alpha_{m}+s}(u-v)|^{p}\dmu_{\tau_{m}}\leq(\epsilon\lambda)^{p}\quad\text{and}\quad\|D^{\tau_{m}}d_{\alpha_{m}+s}v\|_{L^{\infty}\left(\mathcal{B}_{5\rho_{x_{i}}}\left(x_{i}\right)\right)}\leq c_{2}\lambda
\end{equation}
for some constant $c_{2}=c_{2}(\mathsf{data})$.
\end{lem}
\begin{proof}
Let us define for $x,y\in\mathbb{R}^{n}$,
\begin{equation*}
    \tilde{u}(x)=\frac{u(5\rho_{x_{i}}x+x_{i})}{\left(5\rho_{x_{i}}\right)^{\alpha_{m}+\tau_{m}+s}\lambda},\quad\tilde{f}(x)=\frac{f(5\rho_{x_{i}}x+x_{i})}{\left(5\rho_{x_{i}}\right)^{\alpha_{m}+\tau_{m}}\lambda}\quad\text{and}\quad\tilde{A}(x,y)=A(5\rho_{x_{i}}x+x_{i},5\rho_{x_{i}}y+x_{i})
\end{equation*}
to see that 
\begin{equation*}
    (-\Delta)_{p,\tilde{A}}^{s}\tilde{u}=(-\Delta_{p})^{\frac{s}{p}}\tilde{f}\quad\text{in }B_{4}
\end{equation*}
and 
\begin{equation}
\begin{aligned}
\label{scaledscaling assumption2}
    &\frac{1}{\tau_{m}}\dashint_{\mathcal{B}_{4}}|D^{\tau_{m}}d_{\alpha_{m}+s}\tilde{u}|^{p}\dmu_{\tau_{m}}+\Tailp\left(\frac{\tilde{u}-(\tilde{u})_{B_{4}}}{4^{\alpha_{m}+\tau_{m}+s}};B_{4}\right)^{p}\leq 1,\\
    &\frac{1}{\tau_{m}}\dashint_{\mathcal{B}_{4}}|D^{\tau_{m}}d_{\alpha_{m}}\tilde{f}|^{\tilde{q}_{m}}\dmu_{\tau_{m}}+\Tails\left(\frac{\tilde{f}-(\tilde{f})_{B_{4}}}{4^{\alpha_{m}+\tau_{m}}};B_{4}\right)^{\tilde{q}_{m}}\\
    &\quad+\left(\dashint_{B_{2}}\dashint_{B_{2}}|\tilde{A}(x,y)-(\tilde{A})_{B_{2}\times B_{2}}|\dx\dy\right)^{\tilde{q}_{m}}\leq \delta^{\tilde{q}_{m} },
\end{aligned}
\end{equation}
which follows from \eqref{scaling assumption2}. We first observe that there is a constant $c=c(n,s,p,q,\sigma)$ such that
\begin{equation}
\label{holderf}
    \frac{1}{\tau_{m}}\dashint_{\mathcal{B}_{4}}|D^{\tau_{m}}d_{0}\tilde{f}|^{p}\dmu_{\tau_{m}}\leq \frac{1}{\tau_{m}^{1-\frac{p}{\tilde{q}_{m}}}}\left(\frac{1}{\tau_{m}}\dashint_{\mathcal{B}_{4}}|D^{\tau_{m}}d_{0}\tilde{f}|^{\tilde{q}_{m}}\dmu_{\tau_{m}}\right)^{\frac{p}{\tilde{q}_{m}}}\leq c,
\end{equation}
where we have used H\"older's inequality. Combining the estimates \eqref{scaledscaling assumption2} and \eqref{holderf}, and then using the fact
\begin{equation*}
1\leq \frac{8^{\alpha_{m}}}{|x-y|^{\alpha_{m}}}\leq \frac{c}{|x-y|^{\alpha_{m}}} \quad\text{for }x,y\in B_{4} 
\end{equation*}
and
\begin{equation}
\label{indm}
    \frac{1}{\tau_{m}}\leq \frac{c}{\tau_{n_{\sigma}}}\leq c,
\end{equation}
we find that there exists a constant $c$ depending only on $n,s,p,q$ and $\sigma$ such that
\begin{equation}
\label{assumscal}
\begin{aligned}
    &\frac{1}{\tau_{m}}\dashint_{\mathcal{B}_{4}}|D^{\tau_{m}}d_{s}\tilde{u}|^{p}\dmu_{\tau_{m}}+\Tailp\left(\frac{\tilde{u}-(\tilde{u})_{B_{4}}}{4^{\tau_{m}+s}};B_{4}\right)^{p}\leq c^{p},\\
    &\frac{1}{\tau_{m}}\dashint_{\mathcal{B}_{4}}|D^{\tau_{m}}d_{0}\tilde{f}|^{p}\dmu_{\tau_{m}}+\Tails\left(\frac{\tilde{f}-(\tilde{f})_{B_{4}}}{4^{\tau_{m}}};B_{4}\right)^{p}\\
    &\quad+\left(\dashint_{B_{2}}\dashint_{B_{2}}c\left|\tilde{A}(x,y)-(\tilde{A})_{B_{2}\times B_{2}}\right|\dx\dy\right)^{p}\leq (\delta c)^{p}.
\end{aligned}
\end{equation}
By following the same lines as in the proof of \eqref{UminusVd} and \eqref{vtail}, we observe that there is a weak solution $\tilde{v}\in W^{s,p}(B_{2})\cap L^{p-1}_{sp}(\mathbb{R}^{n})$ to 
\begin{equation}
\label{scaled sol v for impr}
(-\Delta_{p})_{\tilde{A}_{2}}^{s}\tilde{v}=0\quad\text{in }B_{2}
\end{equation}
such that
\begin{equation}
\label{Vesti}
    \frac{1}{\tau_{m}}\dashint_{\mathcal{B}_{2}}|D^{\tau_{m}}d_{s}(\tilde{u}-\tilde{v})|^{p}\dmu_{\tau_{m}}\leq c\left(\delta^{\nu}+\delta^{p}\right)\quad\text{and}\quad\Tailp(\tilde{v}-(\tilde{v})_{B_{2}};B_{2})\leq c
\end{equation}
for some constants $c=c(\mathsf{data})$ and $\nu=\nu(\mathsf{data})>0$. In light of Lemma \ref{almost lipschitz} along with the fact that $s+\alpha_{m}+\tau_{m}=s+\tau_{0}<\min\left\{1,\frac{sp}{p-1}\right\}$ by \eqref{tildetau} and \eqref{taucond}, Lemma \ref{SPIL}, \eqref{Vesti} and \eqref{assumscal}, we then find that there is a constant $c_{2}=c_{2}(\mathsf{data})$ such that
\begin{equation}
\label{compar resul2 of2}
\begin{aligned}
    \|D^{\tau_{m}}d_{\alpha_{m}+s}\tilde{v}\|_{L^{\infty}(\mathcal{B}_{1})}&
    =\|\tilde{v}\|_{C^{s+\tau_{0}}(B_{1})}\\
    &\leq c\left(\dashint_{B_{2}}|\tilde{v}-(\tilde{v})_{B_{2}}|^{p}\dx\right)^{\frac{1}{p}}+c\Tailp(\tilde{v}-(\tilde{v})_{B_{2}};B_{2})\\
    &\leq c\left(\frac{1}{\tau_{m}}\dashint_{\mathcal{B}_{2}}|D^{\tau_{m}}d_{s}\tilde{v}|^{p}\,d\mu_{\tau_{m}}\right)^{\frac{1}{p}}+c\\
    &\leq c\left(\frac{1}{\tau_{m}}\dashint_{\mathcal{B}_{2}}|D^{\tau_{m}}d_{s}(\tilde{u}-\tilde{v})|^{p}+|D^{\tau_{m}}d_{s}\tilde{u}|^{p}\,d\mu_{\tau_{m}}\right)^{\frac{1}{p}}+c\\
    &\leq c_{2}.
\end{aligned}
\end{equation}
We next note from \eqref{tail estimate of u local} that 
\begin{equation*}
\begin{aligned}
   \Tailp\left(\frac{\tilde{u}-(\tilde{u})_{B_{2}}}{2^{\alpha_{m}+\tau_{m}+s}}; B_{2}\right)\leq  c\left(\frac{1}{\tau_{m}}\dashint_{\mathcal{B}_{4}}\left|D^{\tau_{m}}d_{\alpha_{m}+s}\tilde{u}\right|^{p}\,d\mu_{\tau_{m}}\right)^{\frac{1}{p}}+c\Tailp\left(\frac{\tilde{u}-(\tilde{u})_{B_{4}}}{4^{\alpha_{m}+\tau_{m}+s}}; B_{4}\right)
\end{aligned}
\end{equation*}
and
\begin{equation*}
\Tails\left(\frac{\tilde{f}-(\tilde{f})_{B_{2}}}{2^{\alpha_{m}+\tau_{m}}}; B_{2}\right)\leq  c\left(\frac{1}{\tau_{m}}\dashint_{\mathcal{B}_{4}}\left|D^{\tau_{m}}d_{\alpha_{m}}\tilde{f}\right|^{\tilde{q}_{m}}\,d\mu_{\tau_{m}}\right)^{\frac{1}{\tilde{q}_{m}}}+c\Tails\left(\frac{\tilde{f}-(\tilde{f})_{B_{4}}}{4^{\alpha_{m}+\tau_{m}}}; B_{4}\right),
\end{equation*}
where we have used H\"older's inequality along with the fact that $p\leq \tilde{q}_{m}$ for the first term in the right-hand side of the above inequality. Combining the above two inequalities, the scaled version of \eqref{indforhigh} with $u=\tilde{u}$, $r=1$ and $x_{0}=0$, and \eqref{scaledscaling assumption2}, we have
\begin{equation}
\label{tildepmu}
\begin{aligned}
    &\left(\frac{1}{\tau_{m}}\dashint_{\mathcal{B}_{1}}|D^{\tau_{m}}d_{\alpha_{m}+s}\tilde{u}|^{\tilde{p}_{m}}\dmu_{\tau_{m}}\right)^{\frac{1}{\tilde{p_{m}}}}\\
    &\leq c\left[\left(\frac{1}{\tau_{m}}\dashint_{\mathcal{B}_{4}}|D^{\tau_{m}}d_{\alpha_{m}+s}\tilde{u}|^{p}\dmu_{\tau_{m}}\right)^{\frac{1}{p}}+\Tailp\left(\frac{\tilde{u}-(\tilde{u})_{B_{4}}}{4^{\alpha_{m}+\tau_{m}+s}}; B_{4}\right)\right]\\
    &\quad+c\left[\left(\frac{1}{\tau_{m}}\dashint_{\mathcal{B}_{4}}|D^{\tau_{m}}d_{\alpha_{m}}\tilde{f}|^{\tilde{q}_{m}}\dmu_{\tau_{m}}\right)^{\frac{1}{\tilde{q}_{m}}}+\Tails\left(\frac{\tilde{f}-(\tilde{f})_{B_{4}}}{4^{\alpha_{m}+\tau_{m}}}; B_{4}\right)\right]\\
    &\leq c
\end{aligned}
\end{equation}
for some constant $c=c(\mathsf{data}, d_{m})$. Applying Lemma \ref{embedding sob} with $q=\tilde{p}_{m}$, $s_{1}=s+\alpha_{m}+\frac{1}{2}\left(1-\frac{p}{\tilde{p}_{m}}\right)$ and $s_{2}=s+\alpha_{m}+\left(1-\frac{p}{\tilde{p}_{m}}\right)$, and then using \eqref{tildepmu}, \eqref{compar resul2 of2} and \eqref{indm}, we get that
\begin{equation*}
\begin{aligned}
    [\tilde{u}-\tilde{v}]_{W^{s+\alpha_{m}+\frac{1}{2}\left(1-\frac{p}{\tilde{p}_{m}}\right)\tau_{m},p}(B_{1})}&\leq c [\tilde{u}-\tilde{v}]_{W^{s+\alpha_{m}+\left(1-\frac{p}{\tilde{p}_{m}}\right)\tau_{m},\tilde{p}_{m}}(B_{1})}\\
    &\leq  c\left(\frac{1}{\tau_{m}}\dashint_{\mathcal{B}_{1}}|D^{\tau_{m}}d_{\alpha_{m}+s}(\tilde{u}-\tilde{v})|^{\tilde{p}_{m}}\dmu_{\tau_{m}}\right)^{\frac{1}{\tilde{p}_{m}}}\\
    &\leq  c\left(\frac{1}{\tau_{m}}\dashint_{\mathcal{B}_{1}}|D^{\tau_{m}}d_{\alpha_{m}+s}\tilde{u}|^{\tilde{p}_{m}}\dmu_{\tau_{m}}\right)^{\frac{1}{\tilde{p}_{m}}}\\
    &\quad+c\left(\frac{1}{\tau_{m}}\dashint_{\mathcal{B}_{1}}|D^{\tau_{m}}d_{\alpha_{m}+s}\tilde{v}|^{\tilde{p}_{m}}\dmu_{\tau_{m}}\right)^{\frac{1}{\tilde{p}_{m}}}\\
    &\leq c.
\end{aligned}
\end{equation*}
We next apply Lemma \ref{InterpoL} with  $s_{1}=s$, $s_{2}=s+\alpha_{m}+\frac{1}{2}\left(1-\frac{p}{\tilde{p}_{m}}\right)\tau_{m}$ and $t=t_{m}\in(0,1)$ such that
\begin{equation*}
     t_{m}s+(1-t_{m})\left(s+\alpha_{m}+\frac{1}{2}\left(1-\frac{p}{\tilde{p}_{m}}\right)\tau_{m}\right)=s+\alpha_{m},
\end{equation*}
to see that
\begin{equation}
\label{interpolation}
\begin{aligned}
[\tilde{u}-\tilde{v}]_{W^{s+\alpha_{m},p}(B_{1})}&
&\leq [\tilde{u}-\tilde{v}]^{t_{m}}_{W^{s,p}(B_{1})}[\tilde{u}-\tilde{v}]^{1-t_{m}}_{W^{s+\alpha_{m}+\frac{1}{2}\left(1-\frac{p}{\tilde{p}_{m}}\right)\tau_{m},p}(B_{1})}.
\end{aligned}
\end{equation}
We now take $t=\min\limits_{m=0,1,\ldots n_{\sigma}}t_{m}>0$ which depends only on $n,s,p,q$ and $\sigma$, since $t_{m}$ and $n_{\sigma}$ depend only on $n,s,p,q$ and $\sigma$.
In light of \eqref{Vesti}, \eqref{indm} and \eqref{interpolation}, we get that 
\begin{equation*}
\begin{aligned}
    \frac{1}{\tau_{m}}\dashint_{\mathcal{B}_{1}}|D^{\tau_{m}}d_{\alpha_{m}+s}(\tilde{u}-\tilde{v})|^{p}\dmu_{\tau_{m}}&= \frac{1}{\tau_{m}}\frac{1}{\mu_{\tau_{m}}(\mathcal{B}_{1})} [\tilde{u}-\tilde{v}]_{W^{s+\alpha_{m},p}(B_{1})}\\
    &\leq c\left[\left(\delta^{p}+\delta^{\nu}\right)\right]^{t}.
\end{aligned}
\end{equation*}
By taking $\delta$ sufficiently small depending only on $\mathsf{data}$,  $\epsilon$ and $d_{m}$, we have 
\begin{equation}
\label{scaledgoal2}
    \frac{1}{\tau_{m}}\dashint_{\mathcal{B}_{1}}|D^{\tau_{m}}d_{\alpha_{m}+s}(\tilde{u}-\tilde{v})|^{p}\dmu_{\tau_{m}}\leq \epsilon^{p}.
\end{equation} We define $v(x)=\left(5\rho_{x_{i}}\right)^{\alpha_{m}+\tau_{m}+s}\tilde{v}\left(\frac{x-x_{i}}{5\rho_{x_{i}}}\right)$ for $x\in \mathbb{R}^{n}$ to conclude that $v$ is a local weak solution to \eqref{sol v for impr} satisfying \eqref{compar result2} which follows from \eqref{scaled sol v for impr}, \eqref{compar resul2 of2} and \eqref{scaledgoal2}.
\end{proof}
\noindent
By following the same strategy as for the proof of Lemma \ref{inductiveicz}, we now obtain \eqref{indforhigh} with $m$, replaced by $m+1$, where the constant $d_{m+1}$ depends only on $\mathsf{data}$ and $d_{m}$.
\begin{lem}
\label{improved comparison}
Let $u$ be a weak solution to \eqref{localizedproblemi}, and let $m\leq n_{\sigma}-1$.
Then there is a small $\delta=\delta(\mathsf{data},d_{m})$ such that if $A$ is $(\delta,2)$-vanishing in $B_{4}\times B_{4}$ only at the diagonal, then $D^{\tau_{m+1}}d_{\alpha_{m+1}+s}u\in L^{p}_{\mathrm{loc}}\left(\mathcal{B}_{4}\,;\,\mu_{\tau_{m}}\right)$ and that \eqref{indforhigh} with $m$ replaced by $m+1$ holds, where $d_{m+1}=d_{m+1}(\mathsf{data},d_{m})$.
\end{lem}
\begin{proof}
Let $B_{4r}(x_{0})\Subset B_{4}$ for some $r>0$. 
We then define for $x,y\in\mathbb{R}^{n}$,
\begin{equation*}
\begin{aligned}
&\tilde{u}(x)=\frac{u\left(rx+x_{0}\right)}{r^{\alpha_{m}+\tau_{m}+s}},\quad \tilde{f}(x)=\frac{f\left(rx+x_{0}\right)}{r^{\alpha_{m}+\tau_{m}}}\quad\text{and}\quad\tilde{A}(x,y)=A\left(rx+x_{0},ry+x_{0}\right)
\end{aligned}
\end{equation*}
to see that 
\begin{equation*}
    (-\Delta)_{p,\tilde{A}}^{s}\tilde{u}=(-\Delta_{p})^{\frac{s}{p}}\tilde{f}\quad\text{in } B_{4},
\end{equation*}
where $\tilde{A}$ is $(\delta,2)$-vanishing in $B_{4}\times B_{4}$ only at the diagonal. We first want to show that
\begin{equation}
\label{goalforindhighdiff}
\begin{aligned}
\left(\dashint_{\mathcal{B}_{1}}|D^{\tau_{m}}d_{\alpha_{m}+s}\tilde{u}|^{\tilde{q}_{m+1}}\dmut\right)^{\frac{1}{\tilde{q}_{m+1}}}
&\leq c\left[\left(\dashint_{\mathcal{B}_{2}}|D^{\tau_{m}}d_{\alpha_{m}+s}\tilde{u}|^{p}\dmut\right)^{\frac{1}{p}}+\Tailp\left(\frac{\tilde{u}-(\tilde{u})_{B_{2}}}{2^{\alpha_{m}+\tau_{m}+s}};B_{2}\right)\right]\\
&\quad+c\left[\left(\dashint_{\mathcal{B}_{2}}|D^{\tau_{m}}d_{\alpha_{m}}f|^{\tilde{q}_{m+1}}\dmu_{\tau}\right)^{\frac{1}{\tilde{q}_{m+1}}}\Tails\left(\frac{\tilde{f}-(\tilde{f})_{B_{2}}}{2^{\alpha_{m}+\tau_{m}}};B_{2}\right)\right]
\end{aligned}
\end{equation}
for some constant $c=c(\mathsf{data},d_{m})$.
We fix $\epsilon\in(0,1)$, and then take $\delta=\delta(\mathsf{data},\epsilon,d_{m})$ given in Lemma \ref{diagonal induction}. Let $1\leq r_{1}<r_{2}\leq2$.
We now apply Lemma \ref{coveringoflevelset} with $\alpha=\alpha_{m}$ $ \tilde{p}=p$ , $\tilde{q}=\tilde{q}_{m}$, $\tilde{\gamma}=\gamma_{0}$ and $\tau=\tau_{m}$ to gain families of countable disjoint diagonal balls and off-diagonal cubes, $\{\mathcal{B}_{\rho_{x_{i}}}\left(x_{i}\right)\}_{i\in\mathbb{N}}$ and $\{\mathcal{Q}\}_{\mathcal{Q}\in\tilde{A}}$, respectively, 
such that
\begin{equation*}
   \left\{(x,y)\in \mathcal{B}_{r_{1}}\,;\,|D^{\tau_{m}}d_{\alpha_{m}+s}u|(x,y)>\lambda\right\}\subset\left(\bigcup\limits_{i\in\mathbb{N}}\mathcal{B}_{5\rho_{x_{i}}}\left(x_{i}\right)\right)\bigcup \left(\bigcup\limits_{\mathcal{Q}\in\tilde{\mathcal{A}}}\mathcal{Q}\right)
\end{equation*}
provided that $\lambda\geq\lambda_{0}$, where $\lambda_{0}$ is given in \eqref{lambda0} with $\alpha=\alpha_{m}$, $\tilde{p}=p$, $\tilde{q}=\tilde{q}_{m}$ and $\tau=\tau_{m}$. Moreover, by Lemma \ref{coveringoflevelset}, we have  \eqref{exitradius}, \eqref{gamma norm} and \eqref{sumofmeasq} with $\alpha=\alpha_{m},\, \tilde{p}=p$ , $\tilde{q}=\tilde{q}_{m}$, $\tilde{\gamma}=\gamma_{0}$ and $\tau=\tau_{m}$.
We now define for  $L\geq\lambda_{0}$, 
\begin{equation*}
    \tilde{\phi}_{L}(r)=\left(\dashint_{\mathcal{B}_{r}}\left|D^{\tau_{m}}d_{\alpha_{m}+s}\tilde{u}\right|_{L}^{\tilde{q}_{m+1}}\dmut\right)^{\frac{1}{\tilde{q}_{m+1}}},
\end{equation*}
where $\left|D^{\tau_{m}}d_{\alpha_{m}+s}\tilde{u}\right|_{L}=\min\{\left|D^{\tau_{m}}d_{\alpha_{m}+s}\tilde{u}\right|,L\}$. We are going to prove that if $L\geq \lambda_{0}$, then
\begin{equation}
\label{btechlemineq2}
\begin{aligned}
    \tilde{\phi}_{L}(r_{1})&\leq \frac{1}{2}\tilde{\phi}_{L}(r_{2})+\frac{c}{(r_{2}-r_{1})^{2n}}\left(\left(\dashint_{\mathcal{B}_{2}}|D^{\tau_{m}}d_{\alpha_{m}+s}\tilde{u}|^{p}\,d\mu_{\tau_{m}}\right)^{\frac{1}{p}}+\Tailp\left(\frac{\tilde{u}-(\tilde{u})_{B_{2}}}{2^{\alpha_{m}+\tau_{m}+s}};B_{2}\right)\right)\\
&\quad+\frac{c}{(r_{2}-r_{1})^{2n}}\left(\left(\dashint_{\mathcal{B}_{2}}|D^{\tau_{m}}d_{\alpha_{m}}\tilde{f}|^{\tilde{q}_{m}}\,d\mu_{\tau_{m}}\right)^{\frac{1}{\tilde{q}_{m}}}+\Tails\left(\frac{\tilde{f}-(\tilde{f})_{B_{2}}}{2^{\alpha_{m}+\tau_{m}}};B_{2}\right)\right)
\end{aligned}
\end{equation}
for some constant $c=c(\mathsf{data},d_{m})$ independent of $L$. Using \eqref{diagonal lambda estimate} with $\alpha=\alpha_{m}$, $\tilde{p}=p$, $\tilde{q}=q_{m}$, and $\tau=\tau_{m}$ and the fact that $\frac{s}{p-1}-(\alpha_{m}+\tau_{m})=\frac{s}{p-1}-\tau_{0}>0$, we have
\begin{equation*}
\begin{aligned}
    &\frac{1}{\tau_{m}}\dashint_{\mathcal{B}_{20\rho_{x_{i}}}(x_{i})}|D^{\tau_{m}}d_{\alpha_{m}+s}\tilde{u}|^{{p}}\dmu_{\tau_{m}}+\Tailp\left(\frac{\tilde{u}-(\tilde{u})_{B_{20\rho_{x_{i}}}}}{(20\rho_{x_{i}})^{\alpha_{m}+\tau_{m}+s}};B_{20\rho_{x_{i}}}(x_{i})\right)^{p}\leq (c_{1}\lambda)^{{p}},\\
    &\frac{1}{\tau_{m}}\dashint_{\mathcal{B}_{20\rho_{x_{i}}}(x_{i})}|D^{\tau_{m}}d_{\alpha_{m}}\tilde{f}|^{\tilde{q}_{m}}\dmu_{\tau_{m}}+\Tails\left(\frac{\tilde{f}-(\tilde{f})_{B_{20\rho_{x_{i}}}}}{(20\rho_{x_{i}})^{\alpha_{m}+\tau_{m}}};B_{20\rho_{x_{i}}}(x_{i})\right)^{\tilde{q}_{m}}\leq (c_{1}\lambda\delta)^{\tilde{q}_{m}}
\end{aligned}
\end{equation*}
for some constant $c_{1}=c_{1}(n,s,p,q,\sigma)$. We then apply Lemma \ref{diagonal induction} with $\lambda$ there, replaced by $c_{1}\lambda$, in order to obtain a weak solution $\tilde{v}$ of \eqref{sol v for impr} satisfying
\begin{equation}
\label{comparisonimprores}
    \frac{1}{\tau_{m}}\dashint_{\mathcal{B}_{5\rho_{x_{i}}}(x_{i})}|D^{\tau_{m}}d_{\alpha_{m}+s}(\tilde{u}-\tilde{v})|^{p}\dmu_{\tau_{m}}\leq (\epsilon c_{1}\lambda)^{p}\quad\text{and}\quad\|D^{\tau_{m}}d_{\alpha_{m}+s}\tilde{v}\|_{L^{\infty}(\mathcal{B}_{5\rho_{x_{i}}}(x_{i}))}\leq c_{2}c_{1}\lambda,
\end{equation} where the constant $c_{2}$ is given in Lemma \ref{diagonal induction}.
We first notice from Fubini's theorem that 
\begin{equation*}
\begin{aligned}
    \int_{\mathcal{B}_{r_{1}}}\left|D^{\tau_{m}}d_{\alpha_{m}+s}\tilde{u}\right|_{L}^{\tilde{q}_{m+1}}\,d\mu_{\tau_{m}}&=\int_{0}^{\infty}\tilde{q}_{m+1}\lambda^{\tilde{q}_{m+1}-1}\mu_{\tau_{m}}\left\{(x,y)\in\mathcal{B}_{r_{1}}\,;\,\left|D^{\tau_{m}}d_{\alpha_{m}+s}\tilde{u}\right|_{L}(x,y)>\lambda\right\}\dlambda\\
    &=\int_{0}^{M\lambda_{0}}\tilde{q}_{m+1}\lambda^{\tilde{q}_{m+1}-1}\mu_{\tau_{m}}\left\{(x,y)\in\mathcal{B}_{r_{1}}\,;\,|D^{\tau_{m}}d_{\alpha_{m}+s}\tilde{u}|(x,y)>\lambda\right\}\dlambda\\
    &\quad+\int_{M\lambda_{0}}^{L}\tilde{q}_{m+1}\lambda^{\tilde{q}_{m+1}-1}\mu_{\tau_{m}}\left\{(x,y)\in\mathcal{B}_{r_{1}}\,;\,|D^{\tau}d_{s}\tilde{u}|(x,y)>\lambda\right\}\dlambda\coloneqq I+J,
\end{aligned}
\end{equation*}
where $M>1$ will be selected later and $L>M\lambda_{0}$. We now estimate $I$ as
\begin{equation*}
    I\leq (M\lambda_{0})^{\tilde{q}_{m+1}}\mu_{\tau_{m}}(\mathcal{B}_{r_{1}}).
\end{equation*}
Proceeding as for \eqref{estimateofJ} leads to
\begin{equation*}
\begin{aligned}
    J&\leq \int_{\lambda_{0}}^{LM^{-1}} \tilde{q}_{m+1}M^{\tilde{q}_{m+1}}\lambda^{\tilde{q}_{m+1}-1}\mu_{\tau_{m}}\left(\left\{(x,y)\in\left(\bigcup\limits_{i}\mathcal{B}_{5\rho_{x_{i}}}(x_{i})\right) \,;\, |D^{\tau_{m}}d_{\alpha_{m}+s}\tilde{u}(x,y)|>M\lambda  \right\}\right)\dlambda\\
    &\quad+\int_{\lambda_{0}}^{LM^{-1}}\tilde{q}_{m+1} M^{\tilde{q}_{m+1}}\lambda^{\tilde{q}_{m+1}-1}\mu_{\tau_{m}}\left(\left\{(x,y)\in\left(\bigcup\limits_{\mathcal{Q}\in\tilde{\mathcal{A}}}\mathcal{Q}\right) \,;\, |D^{\tau_{m}}d_{\alpha_{m}+s}\tilde{u}(x,y)|>M\lambda \right\}\right)\dlambda\\
    &\eqqcolon J_{1}+J_{2}.
\end{aligned}
\end{equation*}
By following the same arguments as in \eqref{diagonal level set estimate1} along with \eqref{comparisonimprores}, we have that there is a constant $c=c(\mathsf{data})$ such that
\begin{equation}
\label{diagonal level set estimate2}
\begin{aligned}
    &\int_{\lambda_{0}}^{LM^{-1}}M^{\tilde{q}_{m+1}}\lambda^{\tilde{q}_{m+1}-1}\mu_{\tau_{m+1}}\left(\left\{(x,y)\in\left(\mathcal{B}_{5\rho_{x_{i}}}\left(x_{i}\right)\right) \,;\, |D^{\tau_{m}}d_{\alpha_{m}+s}\tilde{u}(x,y)|>M\lambda  \right\}\right)\dlambda\\
    &\leq cM^{\tilde{q}_{m+1}-p}\epsilon^{p}\int_{\lambda_{0}}^{LM^{-1}}\lambda^{\tilde{q}_{m+1}-1}\mu_{\tau_{m+1}}\left(\mathcal{B}_{\rho_{x_{i}}}\left(x_{i}\right)\right)\dlambda,
\end{aligned}
\end{equation}
where we have taken $M>c_{1}c_{2}$. We next note from \eqref{exitradius} with $\tilde{p}=p$, $\tilde{q}=\tilde{q}_{m}$, $\alpha=\alpha_{m}$ and $\tau=\tau_{m}$ that either 
\begin{equation}
\label{mut dia u1}
\begin{aligned}
    \frac{\tau_{m}^{\frac{1}{p}}}{c_{d}}\frac{\lambda}{2}\leq \left(\dashint_{\mathcal{B}_{\rho_{x_{i}}}(x_{i})}|D^{\tau_{m}}d_{\alpha_{m}+s}\tilde{u}|^{p}\dmu_{\tau_{m}}\right)^{\frac{1}{p}}
\end{aligned}
\end{equation}
or 
\begin{equation}
\label{mut dia f1}
\begin{aligned}
\frac{\tau_{m}^{\frac{1}{p}}}{c_{d}}\frac{\lambda}{2}\leq \frac{1}{\delta}\left(\dashint_{\mathcal{B}_{\rho_{x_{i}}}(x_{i})}|D^{\tau_{m}}d_{\alpha_{m}}\tilde{f}|^{\tilde{q}_{m}}\dmu_{\tau_{m}}\right)^{\frac{1}{\tilde{q}_{m}}}
\end{aligned}
\end{equation}
holds. After a few algebraic manipulations along with \eqref{indm}, \eqref{mut dia u1} and \eqref{mut dia f1}, we find that
\begin{equation*}
   \mu_{\tau}\left(\mathcal{B}_{\rho_{x_{i}}}(x_{i})\right)\leq \left(\frac{c}{\lambda}\right)^{p}\int_{\mathcal{B}_{\rho_{x_{i}}}(x_{i})}|D^{\tau_{m}}d_{\alpha_{m}+s}\tilde{u}|^{p}\dmu_{\tau_{m}}+\left(\frac{c}{\delta\lambda}\right)^{\tilde{q}_{m}}\int_{\mathcal{B}_{\rho_{x_{i}}}(x_{i})}|D^{\tau_{m}}d_{\alpha_{m}}\tilde{f}|^{\tilde{q}_{m}}\dmu_{\tau_{m}}
\end{equation*}
for some constant $c=c(\mathsf{data})$. We then observe that
\begin{equation}
\label{mut dia3}
\begin{aligned}
    \mu_{\tau_{m}}\left(\mathcal{B}_{\rho_{x_{i}}}(x_{i})\right)&\leq \left(\frac{2c}{\lambda}\right)^{p}\int_{\mathcal{B}_{\rho_{x_{i}}}(x_{i})\cap\left\{ |D^{\tau_{m}}d_{\alpha_{m}+s}\tilde{u}|>a\lambda\right\}}|D^{\tau_{m}}d_{\alpha_{m}+s}\tilde{u}|^{p}\dmu_{\tau_{m}}\\
    &\quad+\left(\frac{2c}{\delta\lambda}\right)^{\tilde{q}_{m}}\int_{\mathcal{B}_{\rho_{x_{i}}}(x_{i})\cap\left\{ |D^{\tau_{m}}d_{\alpha_{m}}\tilde{f}|>b\lambda\right\}}|D^{\tau_{m}}d_{\alpha_{m}}\tilde{f}|^{\tilde{q}_{m}}\dmu_{\tau_{m}},
\end{aligned}
\end{equation}
where $a=\frac{1}{4c}$ and $b=\frac{\delta}{4c}$.
Combine \eqref{diagonal level set estimate2} and \eqref{mut dia3} together with the fact that $\left\{\mathcal{B}_{\rho_{x_{i}}}(x_{i})\right\}$ is a collection of disjoint sets contained in $\mathcal{B}_{r_{2}}$ to see that
\begin{equation*}
\begin{aligned}
    J_{1}&\leq cM^{\tilde{q}_{m+1}-p}\epsilon^{p}\int_{\lambda_{0}}^{\infty}\lambda^{\tilde{q}_{m+1}-p-1}\int_{\mathcal{B}_{r_{2}}\cap\left\{ |D^{\tau_{m}}d_{\alpha_{m}+s}\tilde{u}|_{LM^{-1}}>a\lambda\right\}}|D^{\tau_{m}}d_{\alpha_{m}+s}\tilde{u}|^{p}\dmu_{\tau_{m}}\dlambda\\
    &\quad+cM^{\tilde{q}_{m+1}-p}\epsilon^{p}\int_{\lambda_{0}}^{\infty}\lambda^{\tilde{q}_{m+1}-\tilde{q}_{m}-1}\int_{\mathcal{B}_{r_{2}}\cap\left\{ |D^{\tau_{m}}d_{\alpha_{m}}\tilde{f}|>b\lambda\right\}}\frac{1}{\delta^{\tilde{q}_{m}}}|D^{\tau_{m}}d_{\alpha_{m}}\tilde{f}|^{\tilde{q}_{m}}\dmu_{\tau_{m}}\dlambda\\
\end{aligned}
\end{equation*}
for some constant $c=c(\mathsf{data})$. We then have
\begin{equation*}
\begin{aligned}
J_{1}&\leq cM^{\tilde{q}_{m+1}-p}\epsilon^{p}\int_{\mathcal{B}_{r_{2}}}\left|D^{\tau_{m}}d_{\alpha_{m}+s}\tilde{u}\right|_{LM^{-1}}^{\tilde{q}_{m+1}-p}|D^{\tau_{m}}d_{\alpha_{m}+s}\tilde{u}|^{p}\,d\mu_{\tau_{m}}\\
&\quad+cM^{\tilde{q}_{m+1}-\tilde{q}_{m}}\epsilon^{p}\int_{\mathcal{B}_{r_{2}}}\frac{1}{\delta^{\tilde{q}_{m}+1}}|D^{\tau_{m}}d_{\alpha_{m}}\tilde{f}|^{\tilde{q}_{m+1}}\,d\mu_{\tau_{m}},
\end{aligned}
\end{equation*} 
where we have used Fubini's theorem. Since $\gamma_{0}>\tilde{q}_{m+1}$, similar arguments performed on the estimates of \eqref{estimateofj2r} and \eqref{chooseofM} along with \eqref{gamma norm} and \eqref{sumofmeasq} with $\alpha=\alpha_{m},\, \tilde{p}=p$ ,  $\tilde{\gamma}=\gamma_{0}$ and $\tau=\tau_{m}$ yield that
\begin{equation*}
\begin{aligned}
    J_{2}
    &\leq cM^{\tilde{q}_{m+1}-\gamma_{0}}\int_{\mathcal{B}_{r_{2}}}\left|D^{\tau_{m}}d_{\alpha_{m}+s}\tilde{u}\right|_{LM^{-1}}^{\tilde{q}_{m+1}-p}|D^{\tau_{m}}d_{\alpha_{m}+s}\tilde{u}|^{p}\,d\mu_{\tau_{m}}\\
    &\leq \frac{1}{2^{2n+2q}} \int_{\mathcal{B}_{r_{2}}}\left|D^{\tau_{m}}d_{\alpha_{m}+s}\tilde{u}\right|_{LM^{-1}}^{\tilde{q}_{m+1}}\dmu_{\tau_{m}}
\end{aligned}
\end{equation*}
by taking $M$ sufficiently large such that $cM^{\tilde{q}_{m+1}-\gamma_{0}}\leq\frac{1}{2^{2n+2q}}$ and $M>c_{1}c_{2}$ depending only on $\mathsf{data}$. We now choose a sufficiently small $\epsilon=\epsilon(\mathsf{data})$ such that  $cM^{\tilde{q}_{m+1}-p}\epsilon^{p}\leq\frac{1}{2^{2n+2q}}$ to see that
\begin{equation*}
    J_{1}\leq \frac{1}{2^{2n+2q}} \int_{\mathcal{B}_{r_{2}}}\left|D^{\tau_{m}}d_{\alpha_{m}+s}\tilde{u}\right|_{LM^{-1}}^{\tilde{q}_{m+1}}\dmu_{\tau_{m}}+c\int_{\mathcal{B}_{r_{2}}}|D^{\tau_{m}}d_{\alpha_{m}}\tilde{f}|^{\tilde{q}_{m+1}}\,d\mu_{\tau_{m}}
\end{equation*}
for some constant $c=c(\mathsf{data},d_{m})$, as $\delta$ depends only on $\mathsf{data}$ and $d_{m}$.
In light of the estimates $I$, $J_{1}$ and $J_{2}$, the same arguments as for \eqref{techlemmabef} yield \eqref{btechlemineq2}. Then \eqref{goalforindhighdiff} follows by using the technical lemma \ref{technicallemma} and the limit procedure as in the last part for the proof of Lemma \ref{inductiveicz}. Using the change of variables and \eqref{goalforindhighdiff}, we find that
\begin{equation}
\label{indforhighm}
\begin{aligned}
    &\left(\dashint_{\mathcal{B}_{r}(x_{0})}|D^{\tau_{m}}d_{\alpha_{m}+s}u|^{\tilde{q}_{m+1}}\dmu_{\tau_{m}}\right)^{\frac{1}{\tilde{q}_{m+1}}}\\
    &\leq c\left[\left(\dashint_{\mathcal{B}_{2r}(x_{0})}|D^{\tau_{m}}d_{\alpha_{m}+s}u|^{p}\dmu_{\tau_{m}}\right)^{\frac{1}{p}}+\Tailp\left(\frac{u-(u)_{B_{2r}(x_{0})}}{(2r)^{\alpha_{m}+\tau_{m}+s}}; B_{2r}(x_{0})\right)\right]\\
    &+c\left[\left(\dashint_{\mathcal{B}_{2r}(x_{0})}|D^{\tau_{m}}d_{\alpha_{m}}f|^{\tilde{q}_{m+1}}\dmu_{\tau_{m}}\right)^{\frac{1}{\tilde{q}_{m+1}}}+\Tails\left(\frac{f-(f)_{B_{2r}(x_{0})}}{(2r)^{\alpha_{m}+\tau_{m}}}; B_{2r}(x_{0})\right)\right]
\end{aligned}
\end{equation}
for some constant $c=c(\mathsf{data},d_{m})$. By \eqref{indnumforemb}, we now apply Lemma \ref{embedding sob} with $p=\tilde{p}_{m+1}$, $q=\tilde{q}_{m+1}$, $s_{1}=\alpha_{m+1}+\left(1-\frac{p}{\tilde{p}_{m+1}}\right)\tau_{m+1}$ and $s_{2}=\alpha_{m}+\left(1-\frac{p}{\tilde{q}_{m}}\right)\tau_{m}$, and use \eqref{taualphacond} and \eqref{indm} to see that 
\begin{equation}
\label{holderpmqm}
    \left(\dashint_{\mathcal{B}_{r}(x_{0})}|D^{\tau_{m+1}}d_{\alpha_{m+1}+s}u|^{\tilde{p}_{m+1}}\dmu_{\tau_{m+1}}\right)^{\frac{1}{\tilde{p}_{m+1}}}\leq c\left(\dashint_{\mathcal{B}_{r}(x_{0})}|D^{\tau_{m}}d_{\alpha_{m}+s}u|^{\tilde{q}_{m+1}}\dmu_{\tau_{m}}\right)^{\frac{1}{\tilde{q}_{m+1}}}
\end{equation}
for some constant $c=c(\mathsf{data})$. Using \eqref{holderpmqm}, the fact that
\begin{equation}
\label{alphamplus1}
    1\leq\left(\frac{4r}{|x-y|}\right)^{\alpha_{m+1}-\alpha_{m}}\quad\text{for }x,y\in \mathcal{B}_{2r}(x_{0})
\end{equation}
and \eqref{taualphacond}, we further estimate the left-hand side and the third term in the right-hand side in \eqref{indforhighm} to find that
\begin{equation}
\label{indforhighm1}
\begin{aligned}
    &\left(\dashint_{\mathcal{B}_{r}(x_{0})}|D^{\tau_{m+1}}d_{\alpha_{m+1}+s}u|^{\tilde{p}_{m+1}}\dmu_{\tau_{m+1}}\right)^{\frac{1}{\tilde{p}_{m+1}}}\\
    &\leq c\left[\left(\dashint_{\mathcal{B}_{2r}(x_{0})}|D^{\tau_{m}}d_{\alpha_{m}+s}u|^{p}\dmu_{\tau_{m}}\right)^{\frac{1}{p}}+\Tailp\left(\frac{u-(u)_{B_{2r}(x_{0})}}{(2r)^{\alpha_{m+1}+\tau_{m+1}+s}}; B_{2r}(x_{0})\right)\right]\\
    &+c\left[\left(\dashint_{\mathcal{B}_{2r}(x_{0})}|D^{\tau_{m+1}}d_{\alpha_{m+1}}f|^{\tilde{q}_{m+1}}\dmu_{\tau_{m+1}}\right)^{\frac{1}{\tilde{q}_{m+1}}}+\Tails\left(\frac{f-(f)_{B_{2r}(x_{0})}}{(2r)^{\alpha_{m+1}+\tau_{m+1}}}; B_{2r}(x_{0})\right)\right]
\end{aligned}
\end{equation}
for some constant $c=c(\mathsf{data}, d_{m})$ provided that $\mathcal{B}_{4r}(x_{0})\Subset\mathcal{B}_{4}$. Using H\"older's inequality and \eqref{indforhighm1}, we have 
\begin{equation}
\label{diagonalp}
    \left(\dashint_{\mathcal{B}_{r}(x_{0})}|D^{\tau_{m+1}}d_{\alpha_{m+1}+s}u|^{p}\dmu_{\tau_{m+1}}\right)^{\frac{1}{p}}\leq\left(\dashint_{\mathcal{B}_{r}(x_{0})}|D^{\tau_{m+1}}d_{\alpha_{m+1}+s}u|^{\tilde{p}_{m+1}}\dmu_{\tau_{m+1}}\right)^{\frac{1}{\tilde{p}_{m+1}}}<\infty.
\end{equation}
On the other hand, for any cube $\mathcal{Q}\equiv\mathcal{Q}_{\frac{r}{\sqrt{n}}}\left(z_{1},z_{2}\right)$ where $z_{1},z_{2}\in B_{4}$ satisfying \eqref{dist cube}, we find that
\begin{equation}
\label{offdiagonalp}
\begin{aligned}
    \left(\int_{\mathcal{Q}}|D^{\tau_{m+1}}d_{\alpha_{m+1}+s}u|^{p}\dmu_{\tau_{m+1}}\right)^{\frac{1}{p}}\leq c(\mathsf{data},r)\left(\int_{\mathcal{Q}}|D^{\tau_{0}}d_{s}u|^{p}\dmu_{\tau_{m}}\right)^{\frac{1}{p}}<\infty,
\end{aligned}
\end{equation}
where we have used \eqref{taualphacond} and the fact that $\frac{1}{8}\leq \frac{1}{|x-y|}\leq \frac{\sqrt{n}}{r}$ for any $(x,y)\in \mathcal{Q}$. Thus, the standard covering argument along with \eqref{diagonalp} and \eqref{offdiagonalp} yields $D^{\tau_{m+1}}d_{\alpha_{m+1}+s}u\in L^{p}_{\mathrm{loc}}\left(\mathcal{B}_{4}\,;\,\mu_{\tau_{m+1}}\right)$. Lastly, using \eqref{alphamplus1}, we estimate the first term in the right-hand side of \eqref{indforhighm1} to get \eqref{indforhigh} with $m$ there, replaced by $m+1$, where $d_{m+1}=d_{m+1}(\mathsf{data},d_{m})$.
\end{proof}
\begin{cor}
\label{pnormnsigma}
Let $u$ be a weak solution to \eqref{localizedproblem} with \eqref{assumsigmaforhig}. Then there is a small constant $\delta=\delta(\mathsf{data})$ such that if $A$ is $(\delta,2)$-vanishing in $B_{4}\times B_{4}$ only at the diagonal, then we have that 
\eqref{indforhigh} with $m=n_{\sigma}$, where $d_{n_{\sigma}}=d_{n_{\sigma}}(\mathsf{data})$, and 
$D^{\tau_{n_{\sigma}}}d_{\alpha_{n_{\sigma}}+s}u\in L^{p}_{\mathrm{loc}}\left(\mathcal{B}_{4}\,;\,\mu_{\tau_{n\sigma}}\right)$ with the estimate 
\begin{equation}
\label{nsigmalast}
\begin{aligned}
&\left(\dashint_{\mathcal{B}_{\frac{r}{2^{n_{\sigma}-1}}}(x_{0})}|D^{\tau_{n_{\sigma}}}d_{\alpha_{n_{\sigma}}+s}u|^{p}\,d\mu_{\tau_{n_{\sigma}}}\right)^{\frac{1}{p}}\\
&\leq c\left[\left(\dashint_{\mathcal{B}_{2r}(x_{0})}|D^{\tau_{0}}d_{s}u|^{p}\,d\mu_{\tau_{0}}\right)^{\frac{1}{p}}+\Tailp\left(\frac{u-(u)_{B_{2}(x_{0})}}{(2r)^{\tau_{0}+s}}; B_{2r}(x_{0})\right)\right]\\
&\quad+c\left[\left(\dashint_{\mathcal{B}_{2r}(x_{0})}|D^{\tau_{n_{\sigma}}}d_{\alpha_{n_{\sigma}}}f|^{q}\,d\mu_{\tau_{n_{\sigma}}}\right)^{\frac{1}{q}}+\Tailp\left(\frac{f-(f)_{B_{2}(x_{0})}}{(2r)^{\tau_{0}}}; B_{2r}(x_{0})\right)\right]
\end{aligned}
\end{equation}
for some constant $c=c(\mathsf{data})$ whenever $B_{2r}(x_{0})\Subset B_{4}$.
\end{cor}
\begin{proof}
From Lemma \ref{embedding sob}, $f\in W^{\sigma, q}_{\mathrm{loc}}(B_{4})$ implies $f\in W^{\left(1-\frac{p}{\tilde{q}_{0}}\right)\tau_{0},\tilde{q}_{0}}_{\mathrm{loc}}(B_{4})$, which is equivalent to
\begin{equation*}
    d_{0}f\in L^{\tilde{q}_{0}}_{\mathrm{loc}}\left(\mathcal{B}_{4};\frac{\dx\dy}{|x-y|^{n+\left(\tilde{q}_{0}-p\right)\tau_{0}}}\right).
\end{equation*}
We now apply Theorem \ref{main theorem} with $\sigma=\left(1-\frac{p}{\tilde{q}_{0}}\right)\tau_{0}$ to find a small constant $\delta_{0}=\delta_{0}(\mathsf{data})\in(0,1)$. Thus if $\delta<\delta_{0}$, then we have
\begin{equation*}
\begin{aligned}
    &\left(\dashint_{\mathcal{B}_{r}(x_{0})}|D^{\tau_{0}}d_{s}u|^{\tilde{q}_{0}}\dmu_{\tau_{0}}\right)^{\frac{1}{\tilde{q}_{0}}}\\
    &\leq d_{0}\left[\left(\dashint_{\mathcal{B}_{2r}(x_{0})}|D^{\tau_{0}}d_{s}u|^{p}\dmu_{\tau_{0}}\right)^{\frac{1}{p}}+\Tailp\left(\frac{u-(u)_{B_{2r}(x_{0})}}{(2r)^{\tau_{0}+s}}; B_{2r}(x_{0})\right)\right]\\
    &+d_{0}\left[\left(\dashint_{\mathcal{B}_{2r}(x_{0})}|D^{\tau_{0}}d_{0}f|^{\tilde{q}_{0}}\dmu_{\tau_{0}}\right)^{\frac{1}{\tilde{q}_{0}}}+\Tails\left(\frac{f-(f)_{B_{2r}(x_{0})}}{(2r)^{\tau_{0}}}; B_{2r}(x_{0})\right)\right]
\end{aligned}
\end{equation*}
for some constant $d_{0}=d_{0}(\mathsf{data})$, provided that $B_{2r}(x_{0})\Subset B_{4}
$. Since H\"older's inequality yields that
\begin{equation*}
\left(\dashint_{\mathcal{B}_{r}(x_{0})}|D^{\tau_{0}}d_{s}u|^{\tilde{p}_{0}}\dmu_{\tau_{0}}\right)^{\frac{1}{\tilde{p}_{0}}}\leq \left(\dashint_{\mathcal{B}_{r}(x_{0})}|D^{\tau_{0}}d_{s}u|^{\tilde{q}_{0}}\dmu_{\tau_{0}}\right)^{\frac{1}{\tilde{q}_{0}}},
\end{equation*}
the above two inequalities give that \eqref{indforhigh} with $m=0$, where $d_{0}=d_{0}(\mathsf{data})$.
By applying Lemma \ref{improved comparison} with $m=0$, we have that \eqref{indforhigh} with $m=1$, where $d_{1}=d_{1}(\mathsf{data})$ whenever $\delta<\min\{\delta_{0},\delta_{1}\}$, where the constant $\delta_{1}$ is determined by Lemma \ref{improved comparison} with $m=0$. By iterating this procedure $n_{\sigma}-1$ times, we conclude that if $\delta<
\min\limits_{h=0,1,\ldots,n_{\sigma}}\delta_{h}$, where $\delta_{h}$ is the constant given in Lemma \ref{improved comparison} with $m=h-1$ for $h\geq1$, then $D^{\tau_{n_{\sigma}}}d_{\alpha_{n_{\sigma}}+s}u\in L^{p}_{\mathrm{loc}}\left(\mathcal{B}_{4}\,;\,\mu_{\tau_{n\sigma}}\right)$ and \eqref{indforhigh} with $m=n_{\sigma}$ is true, where $d_{n_{\sigma}}=d_{n_{\sigma}}(\mathsf{data})$.
We are now in the position to prove \eqref{nsigmalast}.
Combining the first inequality in \eqref{diagonalp} and \eqref{indforhighm1} with $r$ there, replaced by $\frac{r}{2^{m-1}}$, respectively, and then applying H\"older's inequality along with the fact that $\tilde{q}_{m+1}<q$ to the third term in the right-hand side of \eqref{indforhighm1}, we get that 
\begin{equation}
\label{nsigma1}
\begin{aligned}
    &\left(\dashint_{\mathcal{B}_{\frac{r}{2^{m}}}(x_{0})}|D^{\tau_{m+1}}d_{\alpha_{m+1}+s}u|^{p}\,d\mu_{\tau_{m+1}}\right)^{\frac{1}{p}}\\
    &\leq c\left[\left(\dashint_{\mathcal{B}_{\frac{r}{2^{m-1}}}(x_{0})}|D^{\tau_{m}+\alpha_{m}}d_{s}u|^{p}\dmu_{\tau_{m}}\right)^{\frac{1}{p}}+\Tailp\left(\frac{u-(u)_{\frac{r}{2^{m-1}}}(x_{0})}{(\frac{r}{2^{m-1}})^{\tau_{0}+s}}; B_{\frac{r}{2^{m-1}}}(x_{0})\right)\right]\\
    &\quad+c\left[\left(\dashint_{\mathcal{B}_{\frac{r}{2^{m-1}}}(x_{0})}|D^{\tau_{m+1}}d_{\alpha_{m+1}}f|^{q}\dmu_{\tau_{m+1}}\right)^{\frac{1}{q}}+\Tails\left(\frac{f-(f)_{B_{\frac{r}{2^{m-1}}}(x_{0})}}{(\frac{r}{2^{m-1}})^{\tau_{0}}}; B_{\frac{r}{2^{m-1}}}(x_{0})\right)\right].
\end{aligned}
\end{equation}
Combine the above estimates \eqref{nsigma1} for $m=0,1,\ldots,n_{\sigma}-1$ to see that
\begin{equation}
\label{firstfinal1}
\begin{aligned}
   &\left(\dashint_{\mathcal{B}_{\frac{r}{2^{n_{\sigma}-1}}}(x_{0})}|D^{\tau_{n_{\sigma}}}d_{\alpha_{n_{\sigma}}+s}u|^{p}\,d\mu_{\tau_{n_{\sigma}}}\right)^{\frac{1}{p}}\\
    &\leq c\left[\left(\dashint_{\mathcal{B}_{2r}(x_{0})}|D^{\tau_{0}}d_{0}u|^{p}\dmu_{\tau_{0}}\right)^{\frac{1}{p}}+\sum_{m=0}^{n_{\sigma}-1}\Tailp\left(\frac{u-(u)_{2r}(x_{0})}{(\frac{r}{2^{m-1}})^{\tau_{0}+s}}; B_{\frac{r}{2^{m-1}}}(x_{0})\right)\right]\\
    &\quad+c\sum_{m=0}^{n_{\sigma}-1}\left[\left(\dashint_{\mathcal{B}_{\frac{r}{2^{m-1}}}(x_{0})}|D^{\tau_{m+1}}d_{\alpha_{m+1}}f|^{q}\dmu_{\tau_{m+1}}\right)^{\frac{1}{q}}+\Tails\left(\frac{f-(f)_{B_{\frac{r}{2^{m-1}}}(x_{0})}}{(\frac{r}{2^{m-1}})^{\tau_{0}}}; B_{\frac{r}{2^{m-1}}}(x_{0})\right)\right],
\end{aligned}
\end{equation}
where $c=c(\mathsf{data})$. We next note from \eqref{taualphacond} and \eqref{indm} that for any $\mathcal{B}\Subset\mathcal{B}_{4}$,
\begin{equation}
\label{holderforflast}
    \left(\dashint_{\mathcal{B}}|D^{\tau_{m-1}}d_{\alpha_{m-1}}f|^{q}\dmu_{\tau_{m-1}}\right)^{\frac{1}{q}}\leq c\left(\dashint_{\mathcal{B}}|D^{\tau_{m}}d_{\alpha_{m}}f|^{q}\dmu_{\tau_{m}}\right)^{\frac{1}{q}},
\end{equation}
where $c=c(\mathsf{data})$. In addition, applying \eqref{tail estimate of u local} with $\tau=\tau_{0}$, $t=s$, $\alpha=0$, $\rho=\frac{r}{2^{m-1}}$ and $i=m$ leads to
\begin{equation}
\label{tailuimproved}
\begin{aligned}
\sum_{m=0}^{n_{\sigma}-1}\Tailp\left(\frac{u-(u)_{\frac{r}{2^{m-1}}}(x_{0})}{(\frac{r}{2^{m-1}})^{\tau_{0}+s}}; B_{\frac{r}{2^{m-1}}}(x_{0})\right)&\leq c\left(\dashint_{\mathcal{B}_{2r}(x_{0})}|D^{\tau_{0}}d_{s}u|^{p}\,d\mu_{\tau_{0}}\right)^{\frac{1}{p}}\\
&\quad+c\Tailp\left(\frac{u-(u)_{2r}(x_{0})}{(2r)^{\tau_{0}+s}}; B_{2r}(x_{0})\right)
\end{aligned}
\end{equation}
for some constant $c=c(\mathsf{data})$, as $n_{\sigma}$ and $\tau_{0}$ depend only on $\mathsf{data}$. Similarly, applying \eqref{tail estimate of u local} with $u(x)$, $s$, $\tau$, $t$, $\alpha$, $\rho$ and $i$ there, replaced by $f(x)$, $\frac{s}{p}$, $\tau_{0}$, $0$, $0$, $\frac{r}{2^{m-1}}$ and $m$, respectively, and then using H\"older's inequality along with the fact that $p<q$, we find 
\begin{equation}
\label{tailfimproved}
\begin{aligned}
\sum_{m=0}^{n_{\sigma}-1}\Tails\left(\frac{f-(f)_{\frac{r}{2^{m-1}}}(x_{0})}{(\frac{r}{2^{m-1}})^{\tau_{0}+s}}; B_{\frac{r}{2^{m-1}}}(x_{0})\right)&\leq c\left(\dashint_{\mathcal{B}_{2r}(x_{0})}|D^{\tau_{0}}d_{0}f|^{q}\,d\mu_{\tau_{0}}\right)^{\frac{1}{q}}\\
&\quad+c\Tailp\left(\frac{f-(f)_{2r}(x_{0})}{(2r)^{\tau_{0}}}; B_{2r}(x_{0})\right).
\end{aligned}
\end{equation}
We combine the estimates \eqref{firstfinal1}, \eqref{holderforflast},  \eqref{tailuimproved} and \eqref{tailfimproved} to show that \eqref{nsigmalast}. 
\end{proof}
We now give a similar  result corresponding to Lemma \ref{inductiveicz}.
\begin{lem}
\label{inductivedcz}
Let $u$ be a weak solution to \eqref{localizedproblem} with \eqref{assumsigmaforhig} and $D^{\tau_{n_{\sigma}}}d_{\alpha_{n_{\sigma}}+s}u\in L^{p_{h}}_{\mathrm{loc}}\left(\mathcal{B}_{4}\,;\,\mu_{\tau_{n_{\sigma}}}\right)$ for some nonnegative integer $h\leq l_{q}$.
Then there are a small positive constant $\delta=\delta(\mathsf{data})$ and a positive constant $c=c(\mathsf{data})$ such that if $A$ is $(\delta,2)$-vanishing in $B_{4}\times B_{4}$ only at the diagonal, then we have
\begin{equation}
\label{ph2norm}
\begin{aligned}
\left(\dashint_{\mathcal{B}_{1}}|D^{\tau_{n_{\sigma}}}d_{\alpha_{n_{\sigma}}+s}u|^{\hat{q}}\dmu_{\tau_{n_{\sigma}}}\right)^{\frac{1}{\hat{q}}}&\leq c\left(\left(\dashint_{\mathcal{B}_{2}}|D^{\tau_{n_{\sigma}}}d_{\alpha_{n_{\sigma}}+s}u|^{p_{h}}\dmu_{\tau_{n_{\sigma}}}\right)^{\frac{1}{p_{h}}}+\Tailp\left(\frac{u-(u)_{B_{2}}}{2^{\alpha_{n_{\sigma}}+\tau_{n_{\sigma}}+s}};B_{2}\right)\right)\\
&\quad+c\left(\left(\dashint_{\mathcal{B}_{2}}|D^{\tau_{n_{\sigma}}}d_{\alpha_{n_{\sigma}}}f|^{\hat{q}}\dmu_{\tau_{n_{\sigma}}}\right)^{\frac{1}{\hat{q}}}+\Tails\left(\frac{f-(f)_{B_{2}}}{2^{\alpha_{n_{\sigma}}+\tau_{n_{\sigma}}}};B_{2}\right)\right),
\end{aligned}
\end{equation}
where the constant $\hat{q}$ is given in \eqref{tildeq}.
\end{lem}
\begin{proof} 
Let us first assume $\delta\leq\hat{\delta}$ where the constant $\hat{\delta}$ is given in Corollary \ref{pnormnsigma} depending only on $\mathsf{data}$. Therefore, applying Corollary \ref{pnormnsigma}, we have \eqref{indforhigh} with $m=n_{\sigma}$, where $d_{n_{\sigma}}=d_{n_{\sigma}}(\mathsf{data})$. For a parameter $\epsilon\in(0,1)$ to be determined later, we take $\delta=\min\left\{\tilde{\delta},\hat{\delta}\right\}$ where the constant $\tilde{\delta}=\tilde{\delta}(\mathsf{data},\epsilon)$ is given in Lemma \ref{diagonal induction} with $m=n_{\sigma}$.  Let $1\leq r_{1}<r_{2}\leq2$.
We now apply Lemma \ref{coveringoflevelset} with $\alpha=\alpha_{n_{\sigma}}$ $ \tilde{p}=p_{h}$ , $\tilde{q}=\tilde{q}_{n_{\sigma}}$, $\tilde{\gamma}=\gamma_{h}$ and $\tau=\tau_{n_{\sigma}}$ to find families of countable disjoint diagonal balls and off-diagonal cubes, $\{\mathcal{B}_{\rho_{x_{i}}}\left(x_{i}\right)\}_{i\in\mathbb{N}}$ and $\{\mathcal{Q}\}_{\mathcal{Q}\in\tilde{A}}$,  
such that
\begin{equation*}
   \left\{(x,y)\in \mathcal{B}_{r_{1}}\,;\,|D^{\tau_{n_{\sigma}}}d_{\alpha_{n_{\sigma}}+s}u|(x,y)>\lambda\right\}\subset\left(\bigcup\limits_{i\in\mathbb{N}}\mathcal{B}_{5\rho_{x_{i}}}\left(x_{i}\right)\right)\bigcup \left(\bigcup\limits_{\mathcal{Q}\in\tilde{\mathcal{A}}}\mathcal{Q}\right)
\end{equation*}
provided that $\lambda\geq\lambda_{0}$, where $\lambda_{0}$ is given in \eqref{lambda0} with $\alpha=\alpha_{n_{\sigma}}$, $\tilde{p}=p_{h}$, $\tilde{q}=\tilde{q}_{n_{\sigma}}$ and $\tau=\tau_{n_{\sigma}}$. Furthermore, Lemma \ref{coveringoflevelset} yields  \eqref{exitradius}, \eqref{gamma norm} and \eqref{sumofmeasq} with $\alpha=\alpha_{n_{\sigma}},\, \tilde{p}=p_{h}$ , $\tilde{q}=\tilde{q}_{n_{\sigma}}$, $\tilde{\gamma}=\gamma_{h}$ and $\tau=\tau_{n_{\sigma}}$. In addition, it follows from \eqref{diagonal lambda estimate} with $\tilde{p}=p_{h}$, $\tilde{q}=q_{n_{\sigma}}$, $\tau=\tau_{n_{\sigma}}$ and $\alpha=\alpha_{n_{\sigma}}$ that
\begin{equation*}
\begin{aligned}
    &\frac{1}{\tau_{n_{\sigma}}}\dashint_{\mathcal{B}_{20\rho_{x_{i}}}(x_{i})}|D^{\tau_{n_{\sigma}}}d_{\alpha_{n_{\sigma}}+s}u|^{p}\dmu_{\tau_{m}}+\Tailp\left(\frac{u-(u)_{B_{20\rho_{x_{i}}}}}{(20\rho_{x_{i}})^{\alpha_{n_{\sigma}}+\tau_{n_{\sigma}}+s}};B_{20\rho_{x_{i}}}(x_{i})\right)^{p}\leq (c_{1}\lambda)^{p},\\
    &\frac{1}{\tau_{n_{\sigma}}}\dashint_{\mathcal{B}_{20\rho_{x_{i}}}(x_{i})}|D^{\tau_{n_{\sigma}}}d_{\alpha_{n_{\sigma}}}f|^{\tilde{q}_{n_{\sigma}}}\dmu_{\tau_{n_{\sigma}}}+\Tails\left(\frac{f-(f)_{B_{20\rho_{x_{i}}}}}{(20\rho_{x_{i}})^{\alpha_{n_{\sigma}}+\tau_{n_{\sigma}}}};B_{20\rho_{x_{i}}}(x_{i})\right)^{\tilde{q}_{n_{\sigma}}}\leq (c_{1}\lambda\delta)^{\tilde{q}_{n_{\sigma}}}
\end{aligned}
\end{equation*}
for some constant $c_{1}=c_{1}(n,s,p,q,\sigma)$ by the fact that $\frac{s}{p-1}-(\alpha_{n_{\sigma}}+\tau_{n_{\sigma}})=\frac{s}{p-1}-\tau_{0}>0$. Applying Lemma \ref{diagonal induction} with $\lambda$ there, replaced by $c_{1}\lambda$, we find a weak solution $v$ to \eqref{sol v for impr} such that
\begin{equation*}
    \frac{1}{\tau_{n_{\sigma}}}\dashint_{\mathcal{B}_{5\rho_{x_{i}}}(x_{i})}|D^{\tau_{n_{\sigma}}}d_{\alpha_{n_{\sigma}}+s}(u-v)|^{p}\dmu_{\tau_{n_{\sigma}}}\leq (\epsilon c_{1}\lambda)^{p}\quad\text{and}\quad\|D^{\tau_{n_{\sigma}}}d_{\alpha_{n_{\sigma}}+s}v\|_{L^{\infty}(\mathcal{B}_{5\rho_{x_{i}}}(x_{i}))}\leq c_{2}c_{1}\lambda,
\end{equation*} where the constant $c_{2}$ is determined in Lemma \ref{diagonal induction}. Let us define for $L\geq\lambda_{0}$,
\begin{equation*}
    {\phi}_{L}(r)=\left(\dashint_{\mathcal{B}_{r}}|(D^{\tau_{n_{\sigma}}}d_{\alpha_{n_{\sigma}}+s}u)_{L}|^{\hat{q}}\dmut\right)^{\frac{1}{\hat{q}}}\quad(1\leq r\leq 2).
\end{equation*}
Following the same arguments as for the estimates of $I$ and $J$ in Lemma \ref{inductiveicz} with $\mu_{\tau}$, $D^{\tau}d_{s}u$, $D^{\tau}d_{0}f$ and $|D^{\tau}d_{0}f|^{p}$ there, replaced by $\mu_{\tau_{n_{\sigma}}}$, $D^{\tau_{n_{\sigma}}}d_{\alpha_{n_{\sigma}}+s}u$, $D^{\tau_{n_{\sigma}}}d_{\alpha_{n_{\sigma}}}f$ and $|D^{\tau_{n_{\sigma}}}d_{\alpha_{n_{\sigma}}}f|^{\tilde{q}_{n_{\sigma}}}$, respectively,  along with the fact that $\gamma_{h}>\hat{q}>\tilde{q}_{n_{\sigma}}$, we have that if $L>\lambda_{0}$, then
\begin{equation*}
\begin{aligned}
    {\phi}_{L}(r_{1})&\leq \frac{1}{2}{\phi}_{L}(r_{2})+\frac{c}{(r_{2}-r_{1})^{2n}}\left(\left(\dashint_{\mathcal{B}_{2}}|D^{\tau_{n_{\sigma}}}d_{\alpha_{n_{\sigma}}+s}u|^{p_{h}}\,d\mu_{\tau_{n_{\sigma}}}\right)^{\frac{1}{p_{h}}}+\Tailp\left(\frac{u-(u)_{B_{2}}}{2^{\alpha_{n_{\sigma}}+\tau_{n_{\sigma}}+s}};B_{2}\right)\right)\\
&\quad+\frac{c}{(r_{2}-r_{1})^{2n}}\left(\left(\dashint_{\mathcal{B}_{2}}|D^{\tau_{n_{\sigma}}}d_{\alpha_{n_{\sigma}}}f|^{\hat{q}}\dmu_{\tau}\right)^{\frac{1}{\hat{q}}}+\Tails\left(\frac{f-(f)_{B_{2}}}{2^{\alpha_{n_{\sigma}}+\tau_{n_{\sigma}}}};B_{2}\right)\right)
\end{aligned}
\end{equation*}
for some constant $c=c(\mathsf{data})$ by taking $\epsilon\in(0,1)$ sufficiently small depending only on $\mathsf{data}$. Using the technical lemma \ref{technicallemma} and then passing to the limit $L\to \infty$ as in the last part for the proof of \eqref{ph1norm}, we get the desired result \eqref{ph2norm}.
\end{proof}
We now give the complete proof of our main theorem \ref{main theorem}.

\noindent
\textbf{Proof of Theorem \ref{main theorem}.} Since we proved Theorem \ref{main theorem} when \eqref{firstsigmares} in Section \ref{subsection51}, we may assume \eqref{assumsigmaforhig}. 
Take $\delta=\delta(\mathsf{data})\in(0,1)$ determined in Lemma \ref{inductivedcz}. We note from the first line in the proof of Lemma \ref{inductivedcz} that $\delta<\hat{\delta}$, where the constant $\hat{\delta}$ is given in Corollary \ref{pnormnsigma}. We first want to show that $D^{\tau_{n_{\sigma}}}d_{\alpha_{n_{\sigma}}+s}u\in L^{p}_{\mathrm{loc}}\left(\Omega\times\Omega\,;\,\mu_{\tau_{n_{\sigma}}}\right)$. Let $B_{\frac{r}{2}}(z_{0})\in \Omega$ with $r\in(0,R]$. Then we define for $x,y\in\mathbb{R}^{n}$,
\begin{equation*}
    \tilde{u}(x)=\left(\frac{2}{r}\right)^{s}u\left(\frac{r}{2}x+z_{0}\right),\quad \tilde{f}(x)=f\left(\frac{r}{2}x+z_{0}\right),\quad \tilde{A}(x,y)=A\left(\frac{r}{2}x+z_{0},\frac{r}{2}y+z_{0}\right)
\end{equation*}
to see that $\tilde{u}$ is a weak solution to \eqref{localizedproblem} with $f$ and $A$ there, replaced by $\tilde{f}$ and $\tilde{A}$, respectively. By Corollary \ref{pnormnsigma}, we have that $D^{\tau_{n_{\sigma}}}d_{\alpha_{n_{\sigma}}+s}\tilde{u}\in L^{p}_{\mathrm{loc}}\left(\mathcal{B}_{4}\,;\,\mu_{\tau_{n_{\sigma}}}\right)$ which is equivalent to $D^{\tau_{n_{\sigma}}}d_{\alpha_{n_{\sigma}}+s}u\in L^{p}_{\mathrm{loc}}\left(\mathcal{B}_{2r}(z_{0})\,;\,\mu_{\tau_{n_{\sigma}}}\right)$. On the other hand, for any cube $\mathcal{Q}\equiv \mathcal{Q}_{\frac{r}{2\sqrt{n}}}(z_{1},z_{2})\Subset \Omega\times \Omega$ satisfying \eqref{dist cube}, we have
\begin{equation*}
    \left(\int_{\mathcal{Q}}|D^{\tau_{n_{\sigma}}}d_{\alpha_{n_{\sigma}}+s}u|^{p}\dmu_{\tau_{n_{\sigma}}}\right)^{\frac{1}{p}}\leq c(\mathsf{data},\mathrm{diam}(\Omega))\left(\int_{\mathcal{Q}}|D^{\tau_{0}}d_{s}u|^{p}\dmu_{\tau_{m}}\right)^{\frac{1}{p}}<\infty,
\end{equation*}
where we have used \eqref{taualphacond} and the fact that
\begin{equation*}
    \frac{r}{2\sqrt{n}}\leq|x-y|\leq 2\mathrm{diam}(\Omega).
\end{equation*}
Since any compact subset of $\Omega\times \Omega$ is covered by finitely many balls $\mathcal{B}_{r}(z_{i})$ with $r\in(0,R]$ and cubes $\mathcal{Q}_{\frac{r}{2\sqrt{n}}}(z_{1,i},z_{2,i})$ with \eqref{dist cube} for $z_{i},z_{1,i},z_{2,i}\in\Omega$, we thus prove that $D^{\tau_{n_{\sigma}}}d_{\alpha_{n_{\sigma}}+s}u\in L^{p}_{\mathrm{loc}}\left(\Omega\times\Omega\,;\,\mu_{\tau_{n_{\sigma}}}\right)$. We next want to show that $D^{\tau_{n_{\sigma}}}d_{\alpha_{n_{\sigma}}+s}u\in L^{q}_{\mathrm{loc}}\left(\Omega\times\Omega\,;\,\mu_{\tau_{n_{\sigma}}}\right)$ and \eqref{goalestimate}.
Let $\mathcal{B}_{2r}(x_{0})\Subset\Omega$ and $r\in(0,R]$. Since Lemma \ref{inductivedcz} is the corresponding one for Lemma \ref{inductiveicz}, by following the same arguments as in the proof of \eqref{goalestimate} when \eqref{firstsigmares} with $D^{\tau}d_{s}u$, $D^{\tau}d_{0}f$ and $\mu_{\tau}$ there, replaced by $D^{\tau_{n_{\sigma}}}d_{\alpha_{n_{\sigma}}+s}u$, $D^{\tau_{n_{\sigma}}}d_{\alpha_{n_{\sigma}}}f$ and $\mu_{\tau_{n_{\sigma}}}$, respectively, we have that for $y_{0}\in B_{r}(x_{0})$ and $\hat{q}$ defined in \eqref{hatq},
\begin{equation}
\label{nsigma11}
\begin{aligned}
        &\left(\dashint_{\mathcal{B}_{\frac{r}{2^{n_{\sigma}}+3}}(y_{0})}|D^{\tau_{n_{\sigma}}}d_{\alpha_{n_{\sigma}}+s}u|^{\hat{q}}\,d\mu_{\tau_{n\sigma}}\right)^{\frac{1}{q}}\\
        &\leq c\left(\left(\dashint_{\mathcal{B}_{\frac{r}{2^{n_{\sigma}}+2}}(y_{0})}|D^{\tau_{n_{\sigma}}}d_{\alpha_{n_{\sigma}}+s}u|^{p}\,d\mu_{\tau_{n_{\sigma}}}\right)^{\frac{1}{p}}+\Tailp\left(\frac{u-(u)_{B_{\frac{r}{2^{n_{\sigma}}+2}}(y_{0})}}{\left(\frac{r}{2^{n_{\sigma}}+2}\right)^{\alpha_{n_{\sigma}}+\tau_{n_{\sigma}}+s}};B_{\frac{r}{2^{n_{\sigma}}+2}}(y_{0})\right)\right)\\
&+c\left(\left(\dashint_{\mathcal{B}_{\frac{r}{2^{n_{\sigma}}+2}}(y_{0})}|D^{\tau_{n_{\sigma}}}d_{\alpha_{n_{\sigma}}}f|^{\hat{q}}\,d\mu_{\tau_{n_{\sigma}}}\right)^{\frac{1}{\hat{q}}}+\Tails\left(\frac{f-(f)_{B_{\frac{r}{2^{n_{\sigma}}+2}}(y_{0})}}{\left(\frac{r}{2^{n_{\sigma}}+2}\right)^{\alpha_{n_{\sigma}}+\tau_{n_{\sigma}}}};B_{\frac{r}{2^{n_{\sigma}}+2}}(y_{0})\right)\right).
\end{aligned}
\end{equation}
We now insert \eqref{nsigmalast} with $\frac{r}{2^{n_{\sigma}-1}}$ and $x_{0}$ there, replaced by $\frac{r}{2^{n_{\sigma}+3}}$ and $y_{0}$, respectively, into \eqref{nsigma11}, and then use \eqref{taualphacond} and \eqref{tail estimate of u local} for the two tail terms in \eqref{nsigma11}, in order to get that
\begin{equation}
\label{nsigma111}
\begin{aligned}
        &\left(\dashint_{\mathcal{B}_{\frac{r}{2^{n_{\sigma}}+3}}(y_{0})}|D^{\tau_{n_{\sigma}}}d_{\alpha_{n_{\sigma}}+s}u|^{\hat{q}}\,d\mu_{\tau_{n\sigma}}\right)^{\frac{1}{\hat{q}}}\\
        &\leq c\left(\left(\dashint_{\mathcal{B}_{\frac{r}{4}}(y_{0})}|D^{\tau_{0}}d_{s}u|^{p}\,d\mu_{\tau_{0}}\right)^{\frac{1}{p}}+\Tailp\left(\frac{u-(u)_{B_{\frac{r}{4}}(y_{0})}}{\left(\frac{r}{4}\right)^{\tau_{0}+s}};B_{\frac{r}{4}}(y_{0})\right)\right)\\
&+c\left(\left(\dashint_{\mathcal{B}_{\frac{r}{4}}(y_{0})}|D^{\tau_{n_{\sigma}}}d_{\alpha_{n_{\sigma}}}f|^{q}\,d\mu_{\tau_{n_{\sigma}}}\right)^{\frac{1}{q}}+\Tails\left(\frac{f-(f)_{B_{\frac{r}{4}}(y_{0})}}{\left(\frac{r}{4}\right)^{\tau_{0}}};B_{\frac{r}{4}}(y_{0})\right)\right)
\end{aligned}
\end{equation}
for some constant $c=c(\mathsf{data})$. For the third term in the right-hand side of \eqref{nsigma111}, we have used H\"older's inequality along with the fact that $p,\hat{q}\leq q$. Furthermore, \eqref{tail estimate of u local2} with $\rho=\frac{r}{4}$, $R=2r$, $\tau=\tau_{0}$, $t=s$ and $\alpha=0$ gives that
\begin{equation}
\label{tailestimateofuimpr}
\Tailp\left(\frac{u-(u)_{B_{\frac{r}{4}}(y_{0})}}{\left(\frac{r}{4}\right)^{\tau_{0}+s}};B_{\frac{r}{4}}(y_{0})\right)\leq c \left(\dashint_{\mathcal{B}_{2r}(x_{0})}|D^{\tau_{0}}d_{s}u|^{p}\,d\mu_{\tau_{0}}\right)^{\frac{1}{p}}+c\Tailp\left(\frac{u-(u)_{B_{2r}(x_{0})}}{(2r)^{\tau_{0}+s}};B_{2r}(x_{0})\right).
\end{equation}
Similarly, \eqref{tail estimate of u local2} with 
$u=f$, $s=\frac{s}{p}$ $\rho=\frac{r}{4}$, $R=2r$, $\tau=\tau_{n_{\sigma}}$, $t=0$ and $\alpha=\alpha_{n_{\sigma}}$, and H\"older's inequality along with the fact that $p<q$ imply that 
\begin{equation}
\label{tailestimateoffimpr}
\begin{aligned}
\Tailp\left(\frac{f-(f)_{B_{\frac{r}{4}}(y_{0})}}{\left(\frac{r}{4}\right)^{\tau_{0}}};B_{\frac{r}{4}}(y_{0})\right)&\leq c \left(\dashint_{\mathcal{B}_{2r}(x_{0})}|D^{\tau_{n_{\sigma}}}d_{\alpha_{n_{\sigma}}}f|^{q}\,d\mu_{\tau_{n_{\sigma}}}\right)^{\frac{1}{q}}\\
&\quad+c\Tailp\left(\frac{f-(f)_{B_{2r}(x_{0})}}{(2r)^{\tau_{0}}};B_{2r}(x_{0})\right).
\end{aligned}
\end{equation}
Combine \eqref{nsigma111}, \eqref{tailestimateofuimpr} and \eqref{tailestimateoffimpr} together with the fact that $\mathcal{B}_{\frac{r}{4}}(y_{0})\subset\mathcal{B}_{2r}(x_{0})$ to show that
\begin{equation}
\label{ondiagonallast}
\begin{aligned}
        &\left(\dashint_{\mathcal{B}_{\frac{r}{2^{n_{\sigma}}+3}}(y_{0})}|D^{\tau_{n_{\sigma}}}d_{\alpha_{n_{\sigma}}+s}u|^{\hat{q}}\,d\mu_{\tau_{n\sigma}}\right)^{\frac{1}{\hat{q}}}\\
        &\leq c\left(\left(\dashint_{\mathcal{B}_{2r}(x_{0})}|D^{\tau_{0}}d_{s}u|^{p}\,d\mu_{\tau_{0}}\right)^{\frac{1}{p}}+\Tailp\left(\frac{u-(u)_{B_{2r}(x_{0})}}{\left(2r\right)^{\tau_{0}+s}};B_{2r}(x_{0})\right)\right)\\
&+c\left(\left(\dashint_{\mathcal{B}_{2r}(x_{0})}|D^{\tau_{n_{\sigma}}}d_{\alpha_{n_{\sigma}}}f|^{q}\,d\mu_{\tau_{n_{\sigma}}}\right)^{\frac{1}{q}}+\Tails\left(\frac{f-(f)_{B_{2r}(x_{0})}}{\left(2r\right)^{\tau_{0}}};B_{2r}(x_{0})\right)\right).
\end{aligned}
\end{equation}
On the other hand, as in \ref{ndiagonallastest}, Lemma \ref{nd lemma} with $\tilde{\gamma}=\gamma_{0}$, $\tau=\tau_{n_{\sigma}}$ and $\alpha=\alpha_{n_{\sigma}}$ yields that
\begin{equation}
\label{nsigma12}
\begin{aligned}
        &\left(\dashint_{\mathcal{Q}}|D^{\tau_{n_{\sigma}}}d_{\alpha_{n_{\sigma}}+s}u|^{\hat{q}}\,d\mu_{\tau_{n\sigma}}\right)^{\frac{1}{\hat{q}}}\\
        &\leq c\left(\dashint_{\mathcal{Q}}|D^{\tau_{n_{\sigma}}}d_{\alpha_{n_{\sigma}}+s}u|^{p}\,d\mu_{\tau_{n_{\sigma}}}\right)^{\frac{1}{p}}+c\left[\sum_{d=1}^{2}\left(\frac{1}{\tau_{n\sigma}}\dashint_{P^{d}\mathcal{Q}}|D^{\tau_{n\sigma}}d_{\alpha_{n_{\sigma}}+s}u|^{p}\,d\mu_{\tau_{n\sigma}}\right)^{\frac{1}{p}}\right],
\end{aligned}
\end{equation}
whenever $\mathcal{Q}\equiv\mathcal{Q}_{\frac{r}{\sqrt{n}2^{n_{\sigma}+3}}}(z_{1},z_{2})\Subset \Omega\times\Omega$  satisfying \eqref{dist cube}. Let us assume that $z_{1},z_{2}\in B_{r}(x_{0})$. We note that there is a constant $c=c(n,s,p,q,\sigma)$ such that 
\begin{equation}
\label{measureestimateofcubeontau}
    \frac{r^{n+p\tau_{m}}}{c}\leq\mu_{\tau_{m}}\left(\mathcal{Q}\right)\leq cr^{n+p\tau_{m}}\quad (m=0,1,\ldots, n_{\sigma})
\end{equation}
by following the same lines for the proof of \eqref{measure of qintau} along with the fact that
\begin{equation}
\label{distxyi}
    \frac{r}{\sqrt{n}2^{n_{\sigma}+3}}<|x-y|<2r\quad\text{for any }(x,y)\in\mathcal{Q}.
\end{equation}
Therefore, using \eqref{taualphacond},  \eqref{measureestimateofcubeontau}, \eqref{distxyi} and \eqref{indm} , we have
\begin{equation}
\label{aofirstfinal1}
    \dashint_{\mathcal{Q}}|D^{\tau_{n_{\sigma}}}d_{\alpha_{n_{\sigma}}+s}u|^{p}\,d\mu_{\tau_{n_{\sigma}}}\leq c\dashint_{\mathcal{Q}}|D^{\tau_{0}}d_{s}u|^{p}\,d\mu_{\tau_{0}}\leq c\dashint_{\mathcal{B}_{2r}(x_{0})}|D^{\tau_{0}}d_{s}u|^{p}\,d\mu_{\tau_{0}}.
\end{equation}
Since $P^{d}\mathcal{Q}\Subset \mathcal{B}_{\frac{r}{2^{n_{\sigma}+3}}}(z_{d})$, we estimate the the second term in the right-hand side of \eqref{nsigma12} as
\begin{equation*}
    c\left[\sum_{d=1}^{2}\left(\frac{1}{\tau_{n\sigma}}\dashint_{P^{d}\mathcal{Q}}|D^{\tau_{n\sigma}}d_{\alpha_{n_{\sigma}}+s}u|^{p}\,d\mu_{\tau_{n\sigma}}\right)^{\frac{1}{p}}\right]\leq c\left[\sum_{d=1}^{2}\left(\frac{1}{\tau_{n\sigma}}\dashint_{\mathcal{B}_{\frac{r}{2^{n_{\sigma}+3}}}(z_{d})}|D^{\tau_{n\sigma}}d_{\alpha_{n_{\sigma}}+s}u|^{\hat{q}}\,d\mu_{\tau_{n\sigma}}\right)^{\frac{1}{\hat{q}}}\right],
\end{equation*}
where we have applied H\"older's inequality together with the fact that $p\leq \hat{q}$. In light of \eqref{aofirstfinal1}, the above inequality and \eqref{ondiagonallast} with $y_{0}=z_{d}$, we estimate the right-hand side of \eqref{nsigma12} as
\begin{equation}
\label{offdiagonallast}
\begin{aligned}
        &\left(\dashint_{\mathcal{Q}}|D^{\tau_{n_{\sigma}}}d_{\alpha_{n_{\sigma}}+s}u|^{\hat{q}}\,d\mu_{\tau_{n\sigma}}\right)^{\frac{1}{\hat{q}}}\\
        &\leq c\left[\left(\dashint_{\mathcal{B}_{2r}(x_{0})}|D^{\tau_{0}}d_{s}u|^{p}\,d\mu_{\tau_{0}}\right)^{\frac{1}{p}}+\Tailp\left(\frac{u-(u)_{B_{2r}(x_{0})}}{(2r)^{\tau_{0}+s}};B_{2r}(x_{0})\right)\right]\\
        &\quad+c\left[ \left(\dashint_{\mathcal{B}_{2r}(x_{0})}|D^{\tau_{0}}d_{0}f|^{q}\,d\mu_{\tau_{0}}\right)^{\frac{1}{q}}+\Tailp\left(\frac{f-(f)_{B_{2r}(x_{0})}}{(2r)^{\tau_{0}}};B_{2r}(x_{0})\right)\right].
\end{aligned}
\end{equation}
Since $\mathcal{B}_{r}(x_{0})$ is covered by finitely many balls $\mathcal{B}_{\frac{r}{2^{n_{\sigma}+3}}}(y_{i})$ and cubes $\mathcal{Q}_{\frac{r}{\sqrt{n}2^{n_{\sigma}+3}}}\left(z_{1,i},z_{2,i}\right)$ satisfying \eqref{dist cube} for $y_{i},z_{1,i},z_{2,i}\in B_{r}(x_{0})$, 
\eqref{goalestimate} with $q=\hat{q}$ follows by using the standard covering argument along with \eqref{ondiagonallast} and \eqref{offdiagonallast}, and then using a few algebraic manipulations together with \eqref{taualphacond}. Moreover, we have $D^{\tau_{n_{\sigma}}}d_{\alpha_{n_{\sigma}}+s}u\in L^{\hat{q}}_{\mathrm{loc}}\left(\Omega\times\Omega\,;\,\mu_{\tau_{n_{\sigma}}}\right)$ which follows from the combinations of \eqref{nsigma111} and \eqref{nsigma12} along with the fact that $D^{\tau_{n_{\sigma}}}d_{\alpha_{n_{\sigma}}+s}u\in L^{p}_{\mathrm{loc}}\left(\Omega\times\Omega\,;\,\mu_{\tau_{n_{\sigma}}}\right)$. If $l_{q}=0$, then $\hat{q}=q$. Let $l_{q}>0$. Then $\hat{q}=p_{1}$. As in \eqref{estimate11forr} and \eqref{nsigma11}, we find that

\begin{equation}
\label{nsigma21}
\begin{aligned}
        &\left(\dashint_{\mathcal{B}_{\frac{r}{2^{n_{\sigma}}+3}}(y_{0})}|D^{\tau_{n_{\sigma}}}d_{\alpha_{n_{\sigma}}+s}u|^{\hat{q}_{1}}\,d\mu_{\tau_{n\sigma}}\right)^{\frac{1}{\hat{q}_{1}}}\\
        &\leq c\left(\left(\dashint_{\mathcal{B}_{\frac{r}{2^{n_{\sigma}}+2}}(y_{0})}|D^{\tau_{n_{\sigma}}}d_{\alpha_{n_{\sigma}}+s}u|^{p_{1}}\,d\mu_{\tau_{n_{\sigma}}}\right)^{\frac{1}{p_{1}}}+\Tailp\left(\frac{u-(u)_{B_{\frac{r}{2^{n_{\sigma}}+2}}(y_{0})}}{\left(\frac{r}{2^{n_{\sigma}}+2}\right)^{\alpha_{n_{\sigma}}+\tau_{n_{\sigma}}+s}};B_{\frac{r}{2^{n_{\sigma}}+2}}(y_{0})\right)\right)\\
&+c\left(\left(\dashint_{\mathcal{B}_{\frac{r}{2^{n_{\sigma}}+2}}(y_{0})}|D^{\tau_{n_{\sigma}}}d_{\alpha_{n_{\sigma}}}f|^{\hat{q}}\,d\mu_{\tau_{n_{\sigma}}}\right)^{\frac{1}{\hat{q}}}+\Tails\left(\frac{f-(f)_{B_{\frac{r}{2^{n_{\sigma}}+2}}(y_{0})}}{\left(\frac{r}{2^{n_{\sigma}}+2}\right)^{\alpha_{n_{\sigma}}+\tau_{n_{\sigma}}}};B_{\frac{r}{2^{n_{\sigma}}+2}}(y_{0})\right)\right),
\end{aligned}
\end{equation}
where $y_{0}\in B_{r}(x_{0})$ and $\hat{q}_{1}$ is defined in \eqref{hatq1}. Plugging \eqref{nsigma111} with $\frac{r}{2^{n_{\sigma}+3}}$ there, replaced by $\frac{r}{2^{n_{\sigma}+2}}$ into \eqref{nsigma21}, and using \eqref{tail estimate of u local2} as in \eqref{tailestimateofuimpr} and \eqref{tailestimateoffimpr}, we have
\begin{equation}
\label{ondiagonallast1}
\begin{aligned}
        &\left(\dashint_{\mathcal{B}_{\frac{r}{2^{n_{\sigma}}+3}}(y_{0})}|D^{\tau_{n_{\sigma}}}d_{\alpha_{n_{\sigma}}+s}u|^{\hat{q}_{1}}\,d\mu_{\tau_{n\sigma}}\right)^{\frac{1}{\hat{q}_{1}}}\\
        &\leq c\left(\left(\dashint_{\mathcal{B}_{2r}(x_{0})}|D^{\tau_{0}}d_{s}u|^{p}\,d\mu_{\tau_{0}}\right)^{\frac{1}{p}}+\Tailp\left(\frac{u-(u)_{B_{2r}(x_{0})}}{\left(2r\right)^{\tau_{0}+s}};B_{2r}(x_{0})\right)\right)\\
&+c\left(\left(\dashint_{\mathcal{B}_{2r}(x_{0})}|D^{\tau_{n_{\sigma}}}d_{\alpha_{n_{\sigma}}}f|^{q}\,d\mu_{\tau_{n_{\sigma}}}\right)^{\frac{1}{q}}+\Tails\left(\frac{f-(f)_{B_{2r}(x_{0})}}{\left(2r\right)^{\tau_{0}}};B_{2r}(x_{0})\right)\right).
\end{aligned}
\end{equation}
On the contrary, as in \eqref{ndiagonallastest}, Lemma \ref{nd lemma} with $\tilde{\gamma}=\gamma_{1}$, $p=p_{1}$, $\tau=\tau_{n_{\sigma}}$ and $\alpha=\alpha_{n_{\sigma}}$ gives that for $\mathcal{Q}\equiv\mathcal{Q}_{\frac{r}{\sqrt{n}2^{n_{\sigma}+3}}}(z_{1},z_{2})$ satisfying \eqref{dist cube}, 
\begin{equation}
\label{nsigma122}
\begin{aligned}
        &\left(\dashint_{\mathcal{Q}}|D^{\tau_{n_{\sigma}}}d_{\alpha_{n_{\sigma}}+s}u|^{\hat{q}_{1}}\,d\mu_{\tau_{n\sigma}}\right)^{\frac{1}{\hat{q}_{1}}}\\
        &\leq c\left(\dashint_{\mathcal{Q}}|D^{\tau_{n_{\sigma}}}d_{\alpha_{n_{\sigma}}+s}u|^{p}\,d\mu_{\tau_{n_{\sigma}}}\right)^{\frac{1}{p}}+c\left[\sum_{d=1}^{2}\left(\frac{1}{\tau_{n\sigma}}\dashint_{P^{d}\mathcal{Q}}|D^{\tau_{n\sigma}}d_{\alpha_{n_{\sigma}}+s}u|^{p_{1}}\,d\mu_{\tau_{n\sigma}}\right)^{\frac{1}{p_{1}}}\right].
\end{aligned}
\end{equation}
Let us assume $z_{1},z_{2}\in B_{r}(x_{0})$ to see that
\begin{equation*}
\begin{aligned}
        &\left(\dashint_{\mathcal{Q}}|D^{\tau_{n_{\sigma}}}d_{\alpha_{n_{\sigma}}+s}u|^{\hat{q}_{1}}\,d\mu_{\tau_{n\sigma}}\right)^{\frac{1}{\hat{q}_{1}}}\\
        &\leq c\left(\dashint_{\mathcal{B}_{2r}(x_{0})}|D^{\tau_{0}}d_{s}u|^{p}\,d\mu_{\tau_{0}}\right)^{\frac{1}{p}}+c\left[\sum_{d=1}^{2}\left(\dashint_{\mathcal{B}_{\frac{r}{2^{n_{\sigma}+3}}}(z_{d})}|D^{\tau_{n\sigma}}d_{\alpha_{n_{\sigma}}+s}u|^{p_{1}}\,d\mu_{\tau_{n\sigma}}\right)^{\frac{1}{p_{1}}}\right],
\end{aligned}
\end{equation*}
where we have used \eqref{aofirstfinal1}, \eqref{indm} and the fact that  $P_{d}\mathcal{Q}\subset \mathcal{B}_{\frac{r}{2^{n_{\sigma}+3}}}(z_{d})$. We now plug \eqref{ondiagonallast} with $\hat{q}=p_{1}$ and $y_{0}=z_{d}$ into the second term in the right-hand side of the above inequality, in order to find that
\begin{equation}
\label{offdiagonallast1}
\begin{aligned}
        &\left(\dashint_{\mathcal{Q}}|D^{\tau_{n_{\sigma}}}d_{\alpha_{n_{\sigma}}+s}u|^{\hat{q}_{1}}\,d\mu_{\tau_{n\sigma}}\right)^{\frac{1}{\hat{q}_{1}}}\\
        &\leq c\left[\left(\dashint_{\mathcal{B}_{2r}(x_{0})}|D^{\tau_{0}}d_{s}u|^{p}\,d\mu_{\tau_{0}}\right)^{\frac{1}{p}}+\Tailp\left(\frac{u-(u)_{B_{2r}(x_{0})}}{(2r)^{\tau_{0}+s}};B_{2r}(x_{0})\right)\right]\\
        &\quad+c\left[ \left(\dashint_{\mathcal{B}_{2r}(x_{0})}|D^{\tau_{0}}d_{0}f|^{q}\,d\mu_{\tau_{0}}\right)^{\frac{1}{q}}+\Tailp\left(\frac{f-(f)_{B_{2r}(x_{0})}}{(2r)^{\tau_{0}}};B_{2r}(x_{0})\right)\right].
\end{aligned}
\end{equation} 
Taking into account \eqref{ondiagonallast1} and \eqref{offdiagonallast1}, the standard covering argument and a few algebraic manipulation along with \eqref{taualphacond} give that \eqref{goalestimate} with $q=\hat{q}_{1}$. Furthermore, \eqref{ondiagonallast1} and \eqref{nsigma122} along with the fact $D^{\tau_{n_{\sigma}}}d_{\alpha_{n_{\sigma}}+s}u\in L^{p_{1}}_{\mathrm{loc}}\left(\Omega\times\Omega\,;\,\mu_{\tau_{n_{\sigma}}}\right)$ imply that $D^{\tau_{n_{\sigma}}}d_{\alpha_{n_{\sigma}}+s}u\in L^{\hat{q}}_{\mathrm{loc}}\left(\Omega\times\Omega\,;\,\mu_{\tau_{n_{\sigma}}}\right)$.
By iterating this procedure $l_{q}-1$ times, we have \eqref{goalestimate} and $D^{\tau_{n_{\sigma}}}d_{\alpha_{n_{\sigma}}+s}u\in L^{q}_{\mathrm{loc}}\left(\Omega\times\Omega\,;\,\mu_{\tau_{n_{\sigma}}}\right)$. Using \eqref{alphatauplus}, we complete that $u\in L^{q}_{\mathrm{loc}}\left(\Omega\times\Omega\,;\frac{\dx\dy}{|x-y|^{n+\sigma q}}\right)$. \qed

\end{document}